\documentclass[8pt, a4]{amsart}
\usepackage{lmodern}
\usepackage[utf8]{inputenc}
\usepackage{amsmath}
\usepackage{graphicx}
\usepackage{amssymb}
\usepackage{esint}
\usepackage[dvipsnames]{xcolor}
\usepackage{tikz}
\usepackage{xxcolor}
\usepackage{floatrow}
\usepackage{color}
\usepackage{amsthm}
\usepackage{amsfonts}
\usepackage{overpic}
\usepackage{epsfig}
\usepackage[english]{babel}
\usepackage{hyperref}
\usepackage[normalem]{ulem}
\usepackage{mathrsfs}
\usepackage{tikz}
\usetikzlibrary{arrows.meta}
\usepackage{pgfplots}
\usepackage{subfigure}
\usepackage{caption}
\usepackage{bbm}
\usepackage{adjustbox}

\usepackage{stmaryrd}
\usepackage{float}
\usepackage{pst-plot}
\hypersetup{
	colorlinks,
	linkcolor={red!80!black},
	citecolor={blue!50!black},
	urlcolor={blue!80!black}
}
 \usepackage{diagbox}
\usepackage[a4paper,top=3.5 cm,bottom=3 cm,left=2.5 cm,right=2.5 cm]{geometry}

\theoremstyle{plain}
\newtheorem{lemma}{Lemma}[section]
\newtheorem{proposition}[lemma]{Proposition}
\newtheorem{theorem}[lemma]{Theorem}

\theoremstyle{definition}
\newtheorem{definition}[lemma]{Definition}

\theoremstyle{remark}
\newtheorem{remark}[lemma]{Remark}

\numberwithin{equation}{section}

\usepackage{dsfont}
\newtheorem{mainassumptions}{ Assumptions}[section]

\newcommand{\R}{\mathbb{R}}

\usepackage{mathtools}

\definecolor{orange}{rgb}{1,.549,0}
\definecolor{GreenYellow }{rgb}{ 0.15,   0.69, 0}
\definecolor{Yellowone}{rgb}{ 0, 1., 0} \definecolor{Goldenrod }{rgb}{  0, 0.10, 0.84}
\definecolor{Dandelion }{rgb}{ 0, 0.29, 0.84} 
\definecolor{Apricot }{rgb}{ 0, 0.32, 0.52}
\definecolor{Peach }{rgb}{ 0, 0.50, 0.70} 
\definecolor{GreenYellow}{cmyk}{0.15,0,0.69,0}
\definecolor{RoyalPurple}{cmyk}{0.75,0.90,0,0}
\definecolor{Yellow}{cmyk}{0,0,1,0}
\definecolor{BlueViolet}{cmyk}{0.86,0.91,0,0.04}
\definecolor{Goldenrod}{cmyk}{0,0.10,0.84,0}
\definecolor{Periwinkle}{cmyk}{0.57,0.55,0,0}
\definecolor{Dandelion}{cmyk}{0,0.29,0.84,0}
\definecolor{CadetBlue}{cmyk}{0.62,0.57,0.23,0}
\definecolor{Apricot}{cmyk}{0,0.32,0.52,0}
\definecolor{CornflowerBlue}{cmyk}{0.65,0.13,0,0}
\definecolor{Peach}{cmyk}{0,0.50,0.70,0}
\definecolor{MidnightBlue}{cmyk}{0.98,0.13,0,0.43}
\definecolor{Melon}{cmyk}{0,0.46,0.5,0}
\definecolor{NavyBlue}{cmyk}{0.94,0.54,0,0}
\definecolor{YellowOrange}{cmyk}{0,0.42,1,0}
\definecolor{RoyalBlue}{cmyk}{1,0.50,0,0}
\definecolor{Orange}{cmyk}{0,0.61,0.87,0}
\definecolor{Blue}{cmyk}{1,1,0,0}
\definecolor{BurntOrange}{cmyk}{0,0.51,1,0}
\definecolor{Cerulean}{cmyk}{0.94,0.11,0,0}
\definecolor{Bittersweet}{cmyk}{0,0.75,1,0.24}
\definecolor{Cyan}{cmyk}{1,0,0,0}
\definecolor{RedOrange}{cmyk}{0,0.77,0.87,0}
\definecolor{ProcessBlue}{cmyk}{0.96,0,0,0}
\definecolor{Mahogany}{cmyk}{0,0.85,0.87,0.35}
\definecolor{SkyBlue}{cmyk}{0.62,0,0.12,0}
\definecolor{Maroon}{cmyk}{0,0.87,0.68,0.32}
\definecolor{Turquoise}{cmyk}{0.85,0,0.20,0}
\definecolor{BrickRed}{cmyk}{0,0.89,0.94,0.28}
\definecolor{TealBlue}{cmyk}{0.86,0,0.34,0.02}
\definecolor{Red}{cmyk}{0,1,1,0}
\definecolor{Aquamarine}{cmyk}{0.82,0,0.30,0}
\definecolor{OrangeRed}{cmyk}{0,1,0.50,0}
\definecolor{BlueGreen}{cmyk}{0.85,0,0.33,0}
\definecolor{RubineRed}{cmyk}{0,1,0.13,0}
\definecolor{Emerald}{cmyk}{1,0,0.50,0}
\definecolor{WildStrawberry}{cmyk}{0,0.96,0.39,0}
\definecolor{JungleGreen}{cmyk}{0.99,0,0.52,0}
\definecolor{Salmon}{cmyk}{0,0.53,0.38,0}
\definecolor{SeaGreen}{cmyk}{0.69,0,0.50,0}
\definecolor{CarnationPink}{cmyk}{0,0.63,0,0}
\definecolor{Green}{cmyk}{1,0,1,0}
\definecolor{Magenta}{cmyk}{0,1,0,0}
\definecolor{ForestGreen}{cmyk}{0.91,0,0.88,0.12}
\definecolor{VioletRed}{cmyk}{0,0.81,0,0}
\definecolor{PineGreen}{cmyk}{0.92,0,0.59,0.25}
\definecolor{Rhodamine}{cmyk}{0,0.82,0,0}
\definecolor{LimeGreen}{cmyk}{0.50,0,1,0}
\definecolor{Mulberry}{cmyk}{0.34,0.90,0,0.02}
\definecolor{YellowGreen}{cmyk}{0.44,0,0.74,0}
\definecolor{RedViolet}{cmyk}{0.07,0.90,0,0.34}
\definecolor{SpringGreen}{cmyk}{0.26,0,0.76,0}
\definecolor{Fuchsia}{cmyk}{0.47,0.91,0,0.08}
\definecolor{OliveGreen}{cmyk}{0.64,0,0.95,0.40}
\definecolor{Lavender}{cmyk}{0,0.48,0,0}
\definecolor{RawSienna}{cmyk}{0,0.72,1,0.45}
\definecolor{Thistle}{cmyk}{0.12,0.59,0,0}
\definecolor{Sepia}{cmyk}{0,0.83,1,0.70}
\definecolor{Orchid}{cmyk}{0.32,0.64,0,0}
\definecolor{Brown}{cmyk}{0,0.81,1,0.60}
\definecolor{DarkOrchid}{cmyk}{0.40,0.80,0.20,0}
\definecolor{Tan}{cmyk}{0.14,0.42,0.56,0}
\definecolor{Purple}{cmyk}{0.45,0.86,0,0}
\definecolor{Gray}{cmyk}{0,0,0,0.50}
\definecolor{Plum}{cmyk}{0.50,1,0,0}
\definecolor{Black}{cmyk}{0,0,0,1}
\definecolor{Violet}{cmyk}{0.79,0.88,0,0}
\definecolor{White}{cmyk}{0,0,0,0}
\definecolor{rltred}{rgb}{0.75,0,0}
\definecolor{rltgreen}{rgb}{0,0.5,0}
\definecolor{oneblue}{rgb}{0,0,0.75}
\definecolor{marron}{rgb}{0.64,0.16,0.16}
\definecolor{forestgreen}{rgb}{0.13,0.54,0.13}
\definecolor{purple}{rgb}{0.62,0.12,0.94}
\definecolor{dockerblue}{rgb}{0.11,0.56,0.98}
\definecolor{freeblue}{rgb}{0.25,0.41,0.88}
\definecolor{myblue}{rgb}{0,0.2,0.4}
\definecolor{Melon}{rgb}{ 0.46, 0.50, 0}
\definecolor{Melone}{rgb}{ 0, 0.46, 0.50}

\allowdisplaybreaks
\usepackage[colorinlistoftodos]{todonotes}
\usepackage{url}
\usepackage{graphicx,tikz} 
  \author{Yacouba Simporé}
\address[  Yacouba Simporé]{Chair for Dynamics, Control, Machine Learning and Numerics, Alexander Von Humboldt-
Professorship, Department of Mathematics, Friedrich-Alexander-Universität Erlangen-Nürnberg, Cauerstraße 11, 91058 Erlangen, Germany
\newline\indent
UYAT, Burkina Faso, Laboratoire LAMIA}
\email{simpore.yacouba@gmail.com}
\title[age-size structured]{Null Controllability of Size-Age Dependent Population Dynamics Models}
\begin{document}
\maketitle
\begin{abstract}
This article examines an infinite-dimensional linear control system that describes population models structured by age, size, and spatial position. The control is localized with respect to space, age and size; an estimate of the time required to bring the system to zero is provided. We demonstrate the null controllability of the model using a technique that avoids the explicit use of parabolic Carleman estimates. Instead, this method combines the observability estimates of the final state with the use of characteristics and estimates of the associated semigroup. Furthermore, we extend this work by highlighting the difference in controllability time for two models with different birth kernels.
\end{abstract}\begin{center}
Keywords Population dynamics, null controllability. \smallbreak
AMS subject classifications. 93B03, 93B05, 92D25
\end{center}

\section*{Introduction}

Mathematical models of populations, whether they incorporating age structure or other continuously varying properties, have a rich history. The pioneering works of Sharpe \cite{sharpe}, Lotka, and McKendrick \cite{mc} in the early 20th century laid the groundwork for this discipline, establishing age-structured population models. Subsequent decades saw rigorus analysis and  development of these models, highlighting the stabilization of age structure in populations with linear mortality and fertility processes. This theory found significant applications in demography, biology, and epidemiology.
However, researchers have recognized that age structure alone often fails to capture the complex dynamics of populations, especially in certain species \cite{b9,o,kk,z}. Individual size and various other physiological and behavioral paramaters can also differentiate cohorts.
Numerous age-size structured models, both linear and nonlinear, have been studied, with seminal treatments provided by Metz and Diekmann \cite{j}, as well as Tucker and Zimmerman \cite{z}. The spatial structure in age or size structured models, both linear and nonlinear, has also been explored, particularly regarding the existence of solutions using the semigroup method \cite{j,b9,z}. Significant results on the null controllability of the Lotka-McKendrick system with or without spatial diffusion, have been obtained \cite{yac3,dy,b1,B1,abst,b6,B}. This growing research in age and size structured models has opened new avenues for understanding and manipulating population dynamics, providing research opportunities in fields such as conservation biology and natural resource management.\smallbreak
In this paper, we are interested in the controllability of population dynamics models depending on age, size, and spatial diffusion, and is structured into three parts. The first part examines the null controllability of the age-size model with a probabilistic fertility function. The second part focuses on an age-size model with degenerate diffusion. Finally, the third part contains the Appendix.
\section{Size-age dependent population dynamics model with probabilistic birth }
\subsection{Presentation of the Model, Assumptions, and Notations}
In this part, we explore the null controllability of a population dynamics model dependent on age, size, and spatial diffusion introduced by Webb \cite{b9}. The model's specificity lies in the fertility function $\beta(a,u,s)$ where $(a,u,s) \in (0,A) \times (0,S)^2$, with $A$ and $S$ denoting the maximal age and size an individual can attain, respectively. This function can be interpreted probabilistically as the likelihood that an individual of age $a$ and size $s$ will give birth to an individual of size $u$.\smallbreak
In models referring back to Webb \cite{b9}, the state space of the system is $K=L^2((0,A)\times(0,S);L^2(\Omega)),$ where $\Omega\subset \mathbb{R}^n$ (with $n\in \mathbb{N}$ in general but $n=3$ for real-life applications) is an open bounded region that represents the variable of the space.\smallbreak
Let $y(x,a,s,t)$ be the density of the distribution of individuals with respect to age, size, and position $x\in \Omega$ and some time $t\in(0,T)\text{, }T>0.$\smallbreak 
According to Webb \cite{b9}, the function $y$ satisfies the partial differential equation.
\begin{equation}
\left\lbrace\begin{array}{ll}
\dfrac{\partial y}{\partial t}+\dfrac{\partial y}{\partial a}+\dfrac{\partial(g(s) y )}{\partial s}-L y+\mu(a,s)y=mu &\hbox{ in }Q ,\\ 
\dfrac{\partial y}{\partial\nu}=0&\hbox{ on }\Sigma,\\ 
y\left( x,0,s,t\right) =y_1(x,s,t)&\hbox{ in } Q_{S,T}\\
y\left(x,a,s,0\right)=y_{0}\left(x,a,s\right)&
\hbox{ in }Q_{A,S};\\
y(x,a,0,t)=0& \hbox{ in } Q_{A,T}.
\end{array}\right.
\label{2}
\end{equation}
where, 
    $$Ly=\sum_{i,j=1}^{n}\dfrac{\partial}{\partial x_i}\left(\sigma_{i,j}\dfrac{\partial y}{\partial x_j}\right)$$ 
    with $\sigma_{i,i}=\sigma_{j,i}\in C^2(\bar{\Omega})$  for $1\leq i,j\leq n$ and there exists a constant $C>0$ such that
$$\sum_{i,j=1}^{n}\sigma_{ij}(x)\xi_i\xi_j>C\|\xi\|_{2}^{2}\quad x\in\bar{\Omega}\hbox{, }\xi\in \R^n;$$
$y_{0}$ is given in $K=L^2((0,A)\times(0,S);L^2(\Omega))$ 
  $Q=\Omega\times (0,A)\times (0,S)\times(0,T)$,
 $\Sigma=\partial\Omega\times(0,A)\times (0,S)\times(0,T)$, $ Q_{S,T}=\Omega\times (0,S)\times (0,T)$, $Q_{A,S}=\Omega\times (0,A)\times (0,S)$ and $Q_{A,T}=\Omega\times (0,A)\times (0,T).$\smallbreak
In addition, the positive function $\mu$ denotes the natural mortality rate of individuals of age $a$ and size $s,$ supposed to be independent of spatial position $x$ and time $t$. The control function is $u$, depending on $x$, $a$, $s$, and $t$, where $m$ is the characteristic function. \smallbreak
We denote by $\beta$ the positive function that describes the fertility rate, which depends on the age $a,$ on size $\hat{s}$ and also depends on the size $s$ of the newborns. The fertility rate $\beta$ is supposed to be independent of spatial position $x$ and time $t,$ so that the density of newly born individuals at the point $x$ at time $t$ of size $s$ is given by
\[y_1(x,s,t)=\displaystyle\int\limits_{0}^{A}\int\limits_{0}^{S}\beta(a,\hat{s},s)y(x,a,\hat{s},t)da d\hat{s}.\]
In these models, size is considered to be an individual-specific variable and can include size, volume, length, maturity, bacterial or viral load, or other physiological or demographic property. The size is assumed to increase the same way for everyone in the population and is controlled by the growth function $g(s).$ The growth modulus is interpreted as \[\int\limits_{\tilde{s}}^{\bar{s}}\frac{ds}{g(s)},\] representing the time required for an individual to increase the size from $\tilde{s}$ to $\bar{s}$,where $0\leq \tilde{s}<\bar{s}\leq S.$  In subsequent sections, we define \( G(s) = \int_{0}^{s} \frac{du}{g(u)} \), establishing a bijection from \([0, S]\) to \([0, G(S)]\) with \( G^{-1} \) denoting its inverse bijection.\smallbreak
The size of individuals also serves to differentiate cohorts. Apart from age, various other factors like body size, dietary requirements, and other physiological and behavioral parameters can be utilized for this purpose. This multifaceted modeling approach is equally applicable in cellular dynamics scenarios.
Age and size-structured population dynamics models, when applied to cellular dynamics, offer essential insights into several biomedical fields. These models improve our understanding of tumor growth and treatment responses in oncology, thereby optimizing anticancer therapies. In regenerative medicine, they help improve the production and transplantation of stem cells for better tissue regeneration. They are also crucial for studying embryonic development and cellular differentiation, providing insight into the regulation of cell specialization processes. In addition, these models allow for the analysis of cellular responses to stress and damage, facilitating the development of strategies to improve cellular resistance and repair mechanisms. These applications demonstrate the importance of these models for significant advances in biology and medicine.\smallbreak
In this work, we examine null controllability in population dynamics models that depend on size, age, and spatial position. The control is localized with respect to these three variables. To achieve null controllability, we impose suitable constraints on the control's support, as well as constraints on size relative to age.\smallbreak
In the following, we assume that the fertility rate $\beta$ and the mortality rate $\mu(a,s) = \mu_1(a) + \mu_2(s)$ satisfy the demographic property. The choice of setting $\mu(a,s) = \mu_1(a) + \mu_2(s)$ is intended to simplify the calculations (see \cite{o}).
All along this paper, we assume that the fertility rate $\beta(\cdot)$ and the mortality rate $\mu(\cdot)$ satisfy the demographic properties:
\begin{mainassumptions}\label{overleaf}
\begin{align*}
(	{\bf H1}):&
	\left\lbrace
	\begin{array}{l}
		\mu_1(a)\geq 0 \text{ for every } a\in (0,A)\\
		\mu_1\in L^{1}\left([0,a^*]\right)\hbox{ for every }\; a^*\in [0,A) \\
\displaystyle\int\limits_{0}^{A}\mu_1(a)da=+\infty
	\end{array} \right.
,\quad
(	{\bf H2}):
	\left\lbrace\begin{array}{l}
		\mu_2(s)\geq 0 \hbox{ for every } s\in (0,S)\\
		\mu_2\in L^{1}\left([0,s^*]\right)\hbox{ for every }\hbox{ } s^*\in [0,A) \\
\displaystyle\int\limits_{0}^{S}\mu_2(s)ds=+\infty
	\end{array} \right.\cr
(	{\bf H3}):&
	\left\lbrace\begin{array}{l}
     g(x)\geq 0\hbox{ for all }x\in [0,S]\hbox{ and }g\in C([0,S]);\\
     \int\limits_{0}^{S}
     \dfrac{ds}{g(s)}<+\infty \hbox{ and  }g'\in L^{\infty} (0,S).
	\end{array}
	\right. \qquad
(	{\bf H4}):
\left\lbrace\begin{array}{l}
\beta(a,\hat{s},s)\in C\left([0,A]\times [0,S]^2\right)\\
      \beta\left(a,\hat{s},s\right)\geq0 \hbox{ in } \left([0,A]\times [0,S]^2\right)\\  
\end{array}
\right. \qquad
\end{align*}
\begin{align*}
    ({\bf H5}):
    \begin{array}{c}
\beta(a,\hat{s},s)=0 \text{ } \forall a\in (0,\hat{a}) \text{ for some } \hat{a}\in (0,A).
    \end{array}
\end{align*}
\end{mainassumptions}
 For more details about the modelling of such system and the biological signi cance of the hypotheses, we refer to Webb \cite{b9}.\smallbreak 
 In addition, we denote by \[Q_1=\omega\times (a_1,a_2)\times (s_1,s_2),\]with $0\leq a_1<a_2\leq A\text{ and }0\leq s_1<s_2\leq S$ and that $\omega\subset \Omega$ is an open set.
We denote by $m$ the characteristic function of $Q,$ and by $m$ the support of the control $u$.\smallbreak
 From the previous data, we denote by \[S^{*}_{1}=\int\limits_{0}^{s_1}\dfrac{ds}{g(s)}\hbox{, }S^{*}_{2}=\int\limits_{s_2}^{S}\dfrac{ds}{g(s)}\hbox{ and }T_0=\max\{S^{*}_{1},S^{*}_{2}\}\] and \[T_1=\max\left\{a_1+S^{*}_{2},S^{*}_{1}\right\}.\] 
Here, $S_{1}^*$ and $S_{2}^*$ represent the time required for an individual to increase the size from $0$ to $s_1$ and from $s_2$ to the maximum size $S,$, respectively.\smallbreak
In this context, \( G(S) \) represents the total time required for an individual to grow from its initial size to its maximum size \( S \). This time is calculated using the growth function \( g(s) \), which describes how the size of individuals evolves over time. Importantly, in this model, time and age are measured on the same scale (e.g., 2 years in time corresponds directly to 2 years in age). The model also considers a \textit{minimum fertility age} \( \hat{a} \), which is the age at which an individual becomes capable of reproducing. If the maximum growth time \( G(S) \) is less than or equal to this minimum fertility age \( \hat{a} \), it implies that no individual will reach the reproductive age. As a result, the \textit{renewal term} in the model, which accounts for births, becomes null. Without this renewal and in the absence of any external influx of individuals (such as migration), the population will gradually decline to extinction, and the state \( y \), representing the population density, will eventually become zero. To prevent this scenario of automatic extinction, the paper imposes the condition \( G(S) > \hat{a} \). This condition ensures that the maximum time required to reach size \( S \) is greater than the minimum fertility age \( \hat{a} \). Consequently, some individuals will be able to reproduce before reaching their maximum size, ensuring continuous population renewal and maintaining its dynamics.
\smallbreak
 \subsection{Main result and comments}
 \subsubsection{Main result}
 Our main result of this part as follows:
\begin{theorem}\label{MainT1}
Assume that $\beta\hbox{, }\mu$ and $g$ satisfy the conditions {\bf H1}-{\bf H2}-{\bf H3}-{\bf H4}-{\bf H5} above.  Assume that $a_1<\hat{a}$, \[S^{*}_{2}<\min\{a_2-a_1,\hat{a}-a_1\}\hbox{ and }S^{*}_{1}<\min\{a_2,\hat{a}\}.\]
Then for every $T>A-a_2+T_1+T_0$ and for every $y_0\in L^2(\Omega\times(0,A)\times (0,S)) $
 there exists a control $u\in L^2(Q_1)$ such that the solution $y$ of $(\ref{2})$ satisfies
\[ y(x,a,s,T)=0\hbox{ a.e }x \in \Omega\hbox{ } a \in (0,A)\hbox{ and } s \in (0,S).\]	
\end{theorem}
In the case where $g(S) = 0$, $\frac{1}{g}$ is locally integrable on $[0, S[$, and $\int_0^S \frac{ds}{g(s)} = +\infty$, we obtain the following result:
\begin{theorem}\label{omo}
    Assume that $\beta\hbox{, }\mu$ and $g$ satisfy the conditions {\bf H1}-{\bf H2}-{\bf H3}-{\bf H4}-{\bf H5} above.  Assume that $a_1<\hat{a}$, \[S^{*}_{1}<\min\{a_2,\hat{a}\}.\]
Then for every $T>A-a_2+T_0$ and for every $y_0\in L^2(\Omega\times(0,A)\times (0,S)) $
 there exists a control $$u\in L^2(\omega\times (a_1,a_2)\times (s_1,S)\times (0,T))$$ such that the solution $y$ of $(\ref{2})$ satisfies
\[ y(x,a,s,T)=0\hbox{ a.e }x \in \Omega\hbox{ } a \in (0,A)\hbox{ and } s \in (0,S).\]	
\end{theorem}
The result of Theorem \ref{omo} is obtained by setting  $s_2 = S$  in the main result.
\subsection{Comments}
\textbf{Interpretation of Constraints (see Figure 2-3 below)}\smallbreak
Condition \( S^*_1 < \min\{a_2,\hat{a}\}<\hat{a} \) means that all newborns must reach size \( s_1 \) before age \( \hat{a} \) (reproductive age). This imposes an appropriate choice of the bound \( s_1 \).\smallbreak
Condition \( S^*_2 + a_1 < \hat{a} \) means that all newborns who exceed size \( s_2 \) before age \( a_1 \) must reach maximum size $S$ before the age of fertility $\hat{a}$. This imposes an appropriate choice of the bound \( s_2 \).\smallbreak
An explanation of the control time for this problem is as follows.\smallbreak
\textbf{Interpretation of the result within the context of a positivity constraint from a heuristic perspective (see Figure 1-2):}\smallbreak
In this interpretation, we acknowledge that at \( t = 0 \), individuals of all ages and sizes may be present. We focus on controlling individuals whose age falls between \( a_1 \) and \( a_2 \) and whose size falls between \( s_1 \) and \( s_2 \). The extreme scenario assumes that any individual meeting the control criteria is automatically eliminated (extreme control with a positivity constraint).
For individuals existing before the control period, they may grow without reaching the size and age thresholds for elimination. However, individuals born after the start of the control period will either be eliminated upon reaching the threshold or may not reach the age required for reproduction before perishing (\( S_{1}^* < \max\{a_2, \hat{a}\} \) and \( S_{2}^* < \max\{a_2, \hat{a}\} \)).
Two situations may arise:
\begin{enumerate}
    \item An individual (with size less than \( s_1 \)) born just before the control period may survive for approximately \( G(s_1) \) time units plus \( A - a_2 \) time units before dying. Before their demise, this individual may give birth to offspring of all sizes, who will have at most \( T_1 \) time units before being eliminated or dying. The total lifespan of this individual and their offspring is therefore \( A - a_2 + S_{1}^* + T_1 \).
    \item An individual whose size exceeds \( s_2 \) can live for around \( S_{2}^* \) time units plus \( A - a_2 \) time units before dying. Before their death, this individual may give birth to offspring of all sizes, who will have at most \( T_1 \) time units before being eliminated. The total lifespan of this individual and their offspring is \( A - a_2 + S_{2}^* + T_1 \).
\end{enumerate}
In conclusion, the overarching condition set by these scenarios is \( T > (A - a_2) + T_0 + T_1 \). This condition encapsulates the collective times considered regarding the age constraint, the survival periods before and after control, and the limited lifespan of individuals born during or after the control period in the extreme feedback control mechanism outlined in the provided explanation.\smallbreak
The two figures below highlight two key aspects of the result: \smallbreak
1. The impact of the random size of newborns on the estimated controllability time. Specifically, the size at birth dictates a trajectory that must be followed to reach the control threshold.\smallbreak
2. A heuristic method to estimate the controllability time.\smallbreak
Moreover, the figures illustrate the controllability time in the case of extreme control with positivity constraint.
\begin{figure}[H]
\begin{adjustbox}{center,left}
        \begin{overpic}[scale=0.3]{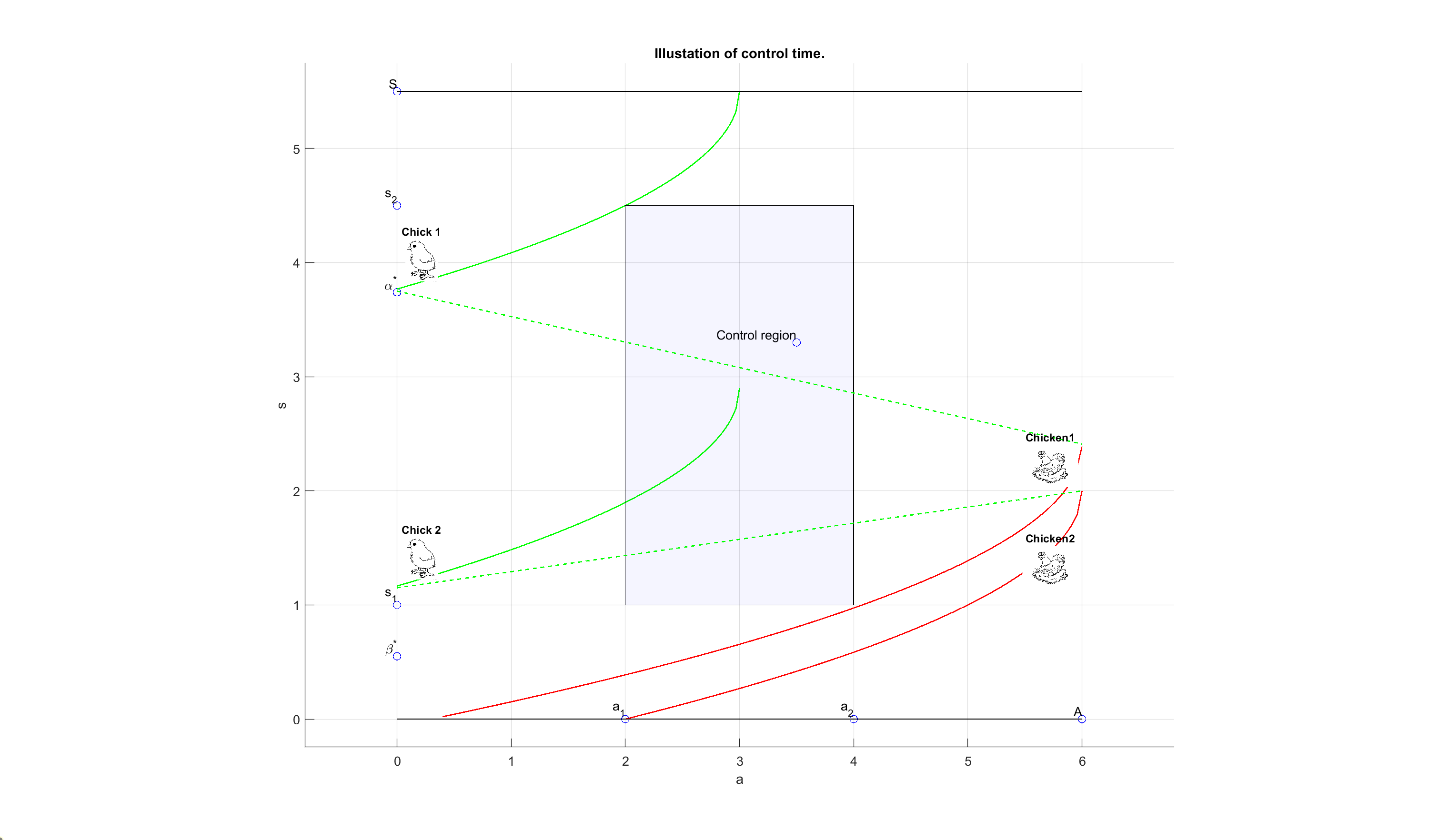}
        \end{overpic}
        \end{adjustbox}
        \caption{ We illustrate here the fact that a chicken at \( t = 0 \) can live for \( S^*_1 + A - a_2 \) time (red trajectories). The fertilization of its eggs can produce chicks of any size between \( 0 \) and \( S \), which will either take \( a_1 + S^*_2 \) or \( S^*_1 \) time to reach the maximum size or the size and age suitable for being eliminated (green trajectories). Thus, the total lifespan of these Chickens and offspring is $T>S^*_1+A-a_2+a_1+S^*_2$ .}
    \end{figure}
    \begin{figure}[H]
    \begin{adjustbox}{center,left}
        \begin{overpic}[scale=0.3]{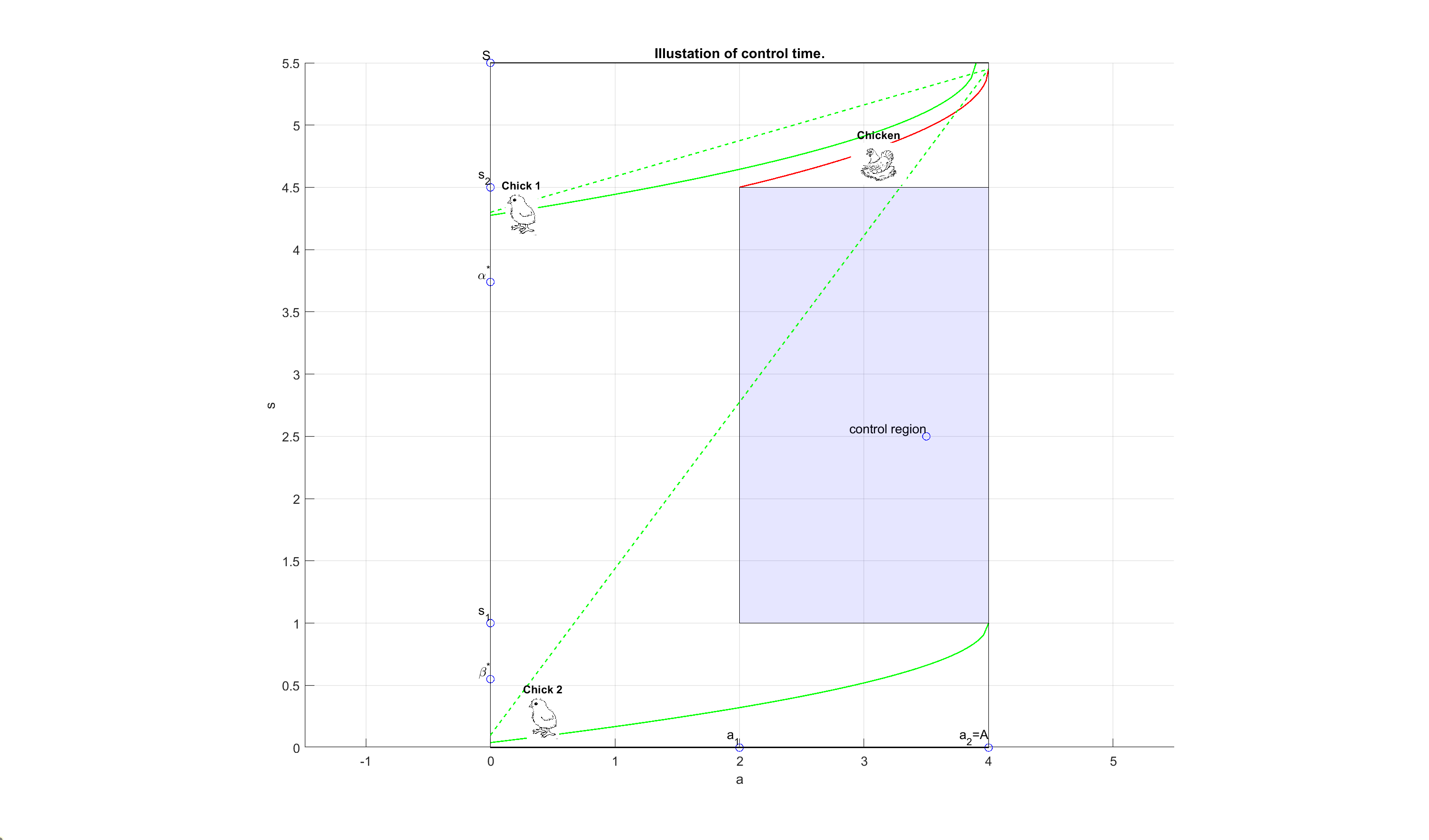}
            
        \end{overpic}
        \end{adjustbox}
    \caption{We illustrate here the fact that a chicken at \( t = 0 \) has a size slightly greater than \( s_2 \) at age \( a_1 \). There is a possibility that it lives approximately \( S^*_2 + A - a_2 \) (red trajectory) (assuming \( A = a_2 \)) time. The fertilization of its eggs can produce chicks of any size between \( 0 \) and \( S \), which will either take \( a_1 + S^*_2 \) or \( S^*_1 \) time to reach the maximum size or the size and age suitable for being eliminated (green trajectories). Thus, the total lifespan of this Chicken and offspring is $T>S^*_1+A-a_2+S^*_2$ or $T>A-a_2+2S^*_1.$}
\end{figure}
\begin{remark}
    It is important to note that the main theorem takes into account all initial conditions \( y_0 \in L^2(Q) \). In the analysis below, we will address a specific case where the initial condition is assumed to be zero over part of \( Q \), which is a special scenario within the broader framework of the theorem.
\end{remark}
\textbf{Analysis of the case where all individuals follow the growth function \(g(s)\) from birth, which imposes a size constraint \(G^{-1}(a) \leq s \leq S\) for individuals aged \(a \in (0, A)\) (refer to Figure 3).}

Assuming that size increases uniformly for all individuals in the population, as determined by the growth function \( g(s) \), the study domain of the model can be restricted to the following:
\[
V = \{(x,a,s,t) \in \Omega \times (0,A) \times (0,S) \times (0,T) \;|\; 0 \leq a \leq A;\; G^{-1}(a) \leq s \leq S;\; 0 \leq t \leq T \}.
\]
This guarantees that the sizes of all individuals remain within the biologically plausible range defined by \( G^{-1}(a) \leq s \leq S \).

As a result, no individual of age \(a\) and size \(s\) can belong to the domain:
\[
W = \{(x, a, s, t) \in \Omega \times (0, A) \times (0, S) \times (0, T) \;|\; 0 \leq a \leq A,\; 0 \leq s < G^{-1}(a),\; 0 \leq t \leq T \}.
\]
The domain \(W\) represents a biologically unfeasible region where the size \(s\) is below the minimum threshold required for an individual of age \(a\). Such cases are excluded to ensure consistency with the growth function \(g(s)\).

Based on the constraints defined in $V,$ we define the control set $Q_2$ to ensure proper management of the system.
\[
Q_2 = \{(x,a,s,t) \in Q \;|\; x \in \omega, \; a_1 \leq a \leq a_2, \; \max\{s_1,G^{-1}(a)\} \leq s \leq s_2, \; 0 \leq t \leq T \},
\]
where \( \omega \) represents the spatial domain, \( [a_1, a_2] \) the range of ages, and \( [\max\{s_1, G^{-1}(a)\}, s_2] \) the size interval. The time \( T \) is determined by the condition:
\[
T > A - a_2 + T_1 + S_2^*,
\]
which ensures that the population reaches the desired state within the required time frame. Importantly, this result is achieved without the need for the condition \( S_1^* < a_2 \). A detailed proof for this scenario will be provided at the end of this part.

Furthermore, if the size of newborns cannot exceed \( s_e \), ensuring that each individual can reach fertility age without surpassing the maximum size \( S \), the domain \( V \) is redefined as:
\[
V = \{(x,a,s,t) \in \Omega \times (0,A) \times (0,S) \times (0,T) \;|\; 0 \leq a \leq A;\; G^{-1}(a) \leq s \leq G^{-1}(a + G(s_e));\; 0 \leq t \leq T \}.
\]
This adjustment reflects the biological constraint on the size of newborns. Correspondingly, the control set is updated as:
\[
Q_3 = \{(x,a,s,t) \in Q \;|\; x \in \omega, \; a_1 \leq a \leq a_2, \; G^{-1}(a) \leq s \leq G^{-1}(a + G(s_e)), \; 0 \leq t \leq T\}.
\]

In this case, the null controllability result can be established without the constraints \( S_2^* < \min\{a_2 - a_1, \hat{a} - a_1\} \) and \( S_1^* < \min\{a_2, \hat{a}\} \). Instead, the condition for the minimum time is:
\[
T > a_1 + A - a_2.
\]
These results illustrate how the model adapts to various biological constraints and control settings, and are visually represented in the figure provided (Figure 3).  
 \begin{figure}[H]
    \begin{adjustbox}{center,left}
        \begin{overpic}[scale=0.24]{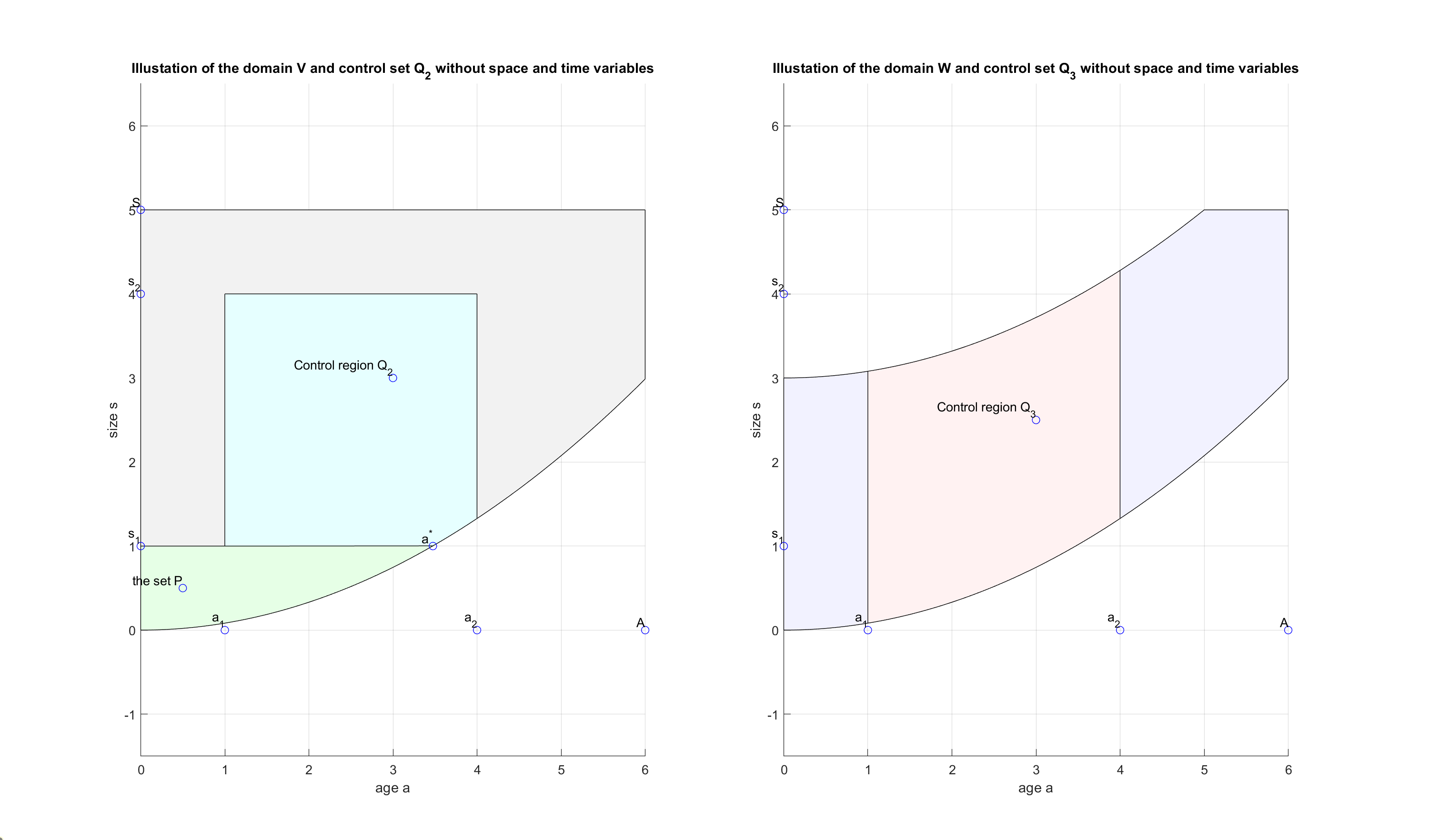}  
        \end{overpic}
        \end{adjustbox}
    \caption{Illustration of the domain of study and control in the case where the initial condition $y_0(x,a,s) = 0$ in the set $W.$ The figure on the right indicates that, in addition to the constraint on the initial condition, the maximal age of newborns is $s_e$.}
\end{figure}
\begin{remark}
\begin{itemize}
 The assumptions {\bf H3} and $a_1<\hat{a}$ are fundamental for the proof of the mains results of this paper.	
\end{itemize}
\end{remark}
Let us now mention some related works from the literature.\smallbreak
The null controllablility of age or size structured population system with spatial diffusion has been addressed in several works \cite{b1,B1,B,b6,dy,abst,yac3}. Indeed, this problem is important because it reflects our ability to control the evolution in time of our target population. The first null controllability result was obtained by B. Ainseba and S. Anita \cite{b1}. In this article, the authors showed that the so-called Lotka-Mckendrick model with spatial diffusion can be driven to a steady state in any positive time, preserving the positivity of the trajectory, provided that the initial data are close to the steady state and the control acts in a spatial subdomain $\omega\subset\Omega$, but for all ages. Recently, Hegoburu and Tucsnak \cite{b3} proved that the same system is null controllable for all ages and in any positive time by controls localized with respect to the spatial variable but active for all ages.  More recently, Maity, Tucsnak and Zuazua \cite{dy} solved the problem of null controllability of a Lotka-Mckendrick model with spatial diffusion, where the control is localized in the space variable as well as with respect to the age. The method combines final-state observability estimates with the use of characteristics and $L^{\infty}$-estimates of the associated semigroup. In \cite{yac2}, using the same method, the author establish the null controllability of a nonlinear Lotka-Mckendrick system with respect to a control localized in the space and age variables. In \cite{yac1}, following the same strategy as in \cite{dy}, Simpor\'e and Traor\'e proved a null controllability result of a nonlinear age, space and two-sex structured population dynamics model.\smallbreak In \cite{Genni}, Fragnelli and Yamamoto studied the null controllability of a degenerate population model structured by age, size, and spatial position. There, by establishing Carleman's type estimates, the authors proves the null controllability of the system \eqref{2}, except for a small interval of ages near zero. It is important to underline that the above null controllability result extend the main result of \cite{BBM} (the model in \cite{BBM} is independent of size $s$).\smallbreak
In this paper, we revisit the work of \cite{Genni} with a control localized in space, age and size. Using the approach developed in \cite{dy}, we establish a sharp estimate of the time $T$ needed to steer the whole population to zero without excluding age and size in a neighborhood of zero, which is not the case with Carleman's estimates. One note that by a direct application of \cite{abst}, we obtain a null controllability result for $T=a_1+A-a_2+2T_0.$  Unlike the strategy used in \cite{abst}, where the size variable is not part of the transport, we use the characteristics according to the three variables, namely time, age and size, thus improving the time to control and while relaxing the condition $T_0<\min\{\hat{a},a_2\}-a_1$.\smallbreak
In contrast to most approaches in the literature, our methodology does not require adaptations of the existing parabolic Carleman estimates to the adjoint system of \ref{2}. We simply combine the characteristics method with existing observability estimates for parabolic equations. Thus, our approach applies independently of the method used to derive final-state observability for the associated parabolic system, whether it be moment methods, local or global Carleman estimates, or the Lebeau-Robbiano strategy.
\subsubsection{The main novelties of the paper}
The main novelties introduced in our paper are as follows:
\begin{itemize}
   \item In this work, we extend the null-controllability analysis of age, size, and spatially-structured population models \cite{simporenull,simpore2022,Genni}. More specifically, we incorporate a probabilistic birth rate with respect to the size variable, while also relaxing constraints on the growth function.
   
   \item We enhance existing estimates regarding the time needed to achieve complete control of the system. Our results demonstrate that our global controllability result is applicable to individuals of all ages, without requiring the exclusion of ages close to zero. Our approach addresses scenarios where control is active for ages $a\in [a_1,a_2]$ and size $s\in [s_1,s_2]$, with arbitrary $a_1<a_2\leq A$ and $0\leq s_1<s_2\leq S,$ provided that $\text{supp} \beta\cap [0,a_1]\neq \emptyset$, thereby eliminating the need to control extremely low ages.
   
   \item By leveraging a recent abstract result by Maity, Tucsnak, and Zuazua, we could have obtained our desired result under the condition \[ T_0 = \max\{S^*_1, S^*_2\} < \max\{a_2, \hat{a}\} - a_1, \] accompanied by the estimated time \[ T > A - a_2 + a_1 + 2T_0 .\] However, we relaxed this condition to \[ S^*_2 < \max\{a_2, \hat{a}\} - a_1 \quad \text{and} \quad S^*_1 < \max\{a_2, \hat{a}\} ,\] refining the control time estimation as well.
   
   \item We emphasize the importance of random newborn size in estimating the control time. Particularly, the estimated control time is extended for models with random birth sizes. In cases where the birth rate at time $t$ is described by \[ y(x,0,s,t)=\int_{0}^{A}\int_{0}^{S}\beta(a,\hat{s},s)y(x,a,\hat{s},t)\,da\,d\hat{s}, \] the estimated time is \[ T>A-a_2+\max\{a_1+S^*_2,S^*_1\}+\max\{S^*_2,S^*_1\} ,\] contrasting with the scenario where the birth rate is \[ y(x,0,s,t)=\int_{0}^{A}\beta(a,s)y(x,a,s,t)\,da, \] resulting in an estimated time of \[ T>A-a_2+a_1+S_1^*+S_2^* ,\] which is more efficient. This efficiency stems from the fact that the initial size aligns with offspring size in this model. The second important result is stated and proved in Part II of the document.
   \item This methodology diverges from conventional approaches in the literature by obviating the need to adapt existing Carleman parabolic estimates to the adjoint system of equation (\ref{2}). Instead, we integrate characteristic methods with established observability estimates for parabolic equations.
   Consequently, this approach is universally applicable regardless of the technique employed to establish final-state observability for the associated parabolic system.
\end{itemize}
\subsubsection{A sketch of the proof}
The proof of the main result boils down to proving the observability inequality of the adjoint system below, which is done in three parts:
\begin{itemize}
    \item First, under the conditions
    \[ S^*_2 < \max\{a_2, \hat{a}\} - a_1 \quad \text{and} \quad S^*_1 < \max\{a_2, \hat{a}\} ,\]  
    and the condition $\text{supp} \beta\cap [0, a_1]\neq \emptyset$, and based on the characteristic curves
    \[\gamma(\lambda)=(t-\lambda,G^{-1}(G(s)+t-\lambda),\lambda); \lambda\in (0,t),\] 
    we estimate the non-local term 
    \[\int_{0}^{S}\beta(a, s, u)q(x, 0, u, t)du.\]
    \item We partition the interval $[0, S]$ as $[0, S] = [0, s_1] \cup [s_1, S]$. Using the characteristic curves, we estimate the state of the adjoint system at the final time $T$ as follows:
    \begin{enumerate}
        \item Under the constraints 
        \[\text{supp} \beta\cap [0, a_1]\neq \emptyset \text{ and } S^*_2 < \max\{a_2, \hat{a}\} - a_1,\]
        we establish an estimate for the term 
        \[\int_{s_1}^{S}\int_{0}^{a_0}\int_{\Omega}q^2(x, a, s, T)dxdads,\] where $a_0$ is appropriately chosen in $(a_1,a_2).$
        \item Under the constraints 
        \[\text{supp} \beta\cap [0, a_1] \neq \emptyset \text{ and } S^*_1 < \max\{a_2, \hat{a}\},\] 
        we establish an estimate for the term 
        \[\int_{0}^{s_1}\int_{0}^{a^*_0}\int_{\Omega}q^2(x, a, s, T)dxdads,\] where $a^*_0$ is appropriately chosen in $(0,a_2).$
    \end{enumerate}
    \item Finally, by combining these results with $L^\infty$ estimates of the associated semigroup, we estimate the state of the adjoint system at the final time $T$ on 
    \[[a_0, A]\times[s_1, S]\hbox{ and }[a^*_0, A]\times [0, s_1],\]
  obtained the desired observability inequality as a consequence of the main result.
\end{itemize}
 \subsubsection{Outline}
The remaining part of this work is organized as follows:
\begin{enumerate}
    \item In Section \ref{Seca},we recall basic facts about the Semigroups in population models structured by age, size, and spatial diffusion, and define the associated adjoint semigroup.
    \item  In Section \ref{MainT1}, we formulate our control problem in a semigroup setting, prove the final state observability for the adjoint system, and consequently, obtain the proof of our main result (Theorem \ref{MainT1}).
    \item In Part II, titled Degenerate Model, we state the second main result (concerning the case of the kernel $$y(x,0,s,t)=\int_{0}^{A}\beta(a,s)y(x,a,s,t)\,da),$$ and provide its proof.
\end{enumerate}
\subsection{Population Models Structured by Age, Size, and Spatial Position Semigroup}\label{Seca}
In this section, we present basic results on the population semigroup for the linear age, size structured model with diffusion, and its adjoint operator. We give the existence of the semigroup in the Hilbert space $K=L^2((0,A)\times(0,S);L^2(\Omega)).$\smallbreak	
We define the following operators:
\[
\mathbf{A_\mu} u := \mu(a,s)u\quad
\mathbf{A_a} u := \frac{\partial u}{\partial a}, \quad 
\mathbf{A_s} u := \frac{\partial}{\partial s} \big(g(s) u \big), \quad
L u := \sum_{i,j=1}^{n}\dfrac{\partial}{\partial x_i}\left(\sigma_{i,j}\dfrac{\partial u}{\partial x_j}\right).
\]

The total operator \(\mathbf{A}\) is defined as:
\[
\mathbf{A} u := \mathbf{A_a} u + \mathbf{A_s} u - L u+\mathbf{A_\mu} u.
\]
The domain of $L$, denoted by \(D(L)\), is defined as:
\[
D(L) = \left\{ u \in H^2(\Omega) \; \middle| \; \frac{\partial u}{\partial \nu} \big|_{\partial \Omega} = 0 \right\},
\]
where \(\nu\) is the outward unit normal vector to \(\partial \Omega\), and the Neumann boundary condition \(\frac{\partial u}{\partial \nu} = 0\) ensures compatibility with the physical model.
The domain of \(\mathbf{A}\), denoted by \(D(\mathbf{A})\), is defined as:
\[
D(\mathbf{A}) = \left\{ u \in L^2(Q_{A,S}; D(L)) \; \middle| \; \frac{\partial u}{\partial s}+\frac{\partial u}{\partial a} +\mu(a,s)u\in L^2(Q_{A,S}; H^1(\Omega)), \; u(x,0,s) = \int_0^S \int_0^A \beta(a, u, s) u(x, a, u) \, da \, du \right\}.
\]
The domain \(Q_{A,S}\) is given by:
\[
Q_{A,S} := (0, A) \times (0, S).
\]
The operator \(\mathbf{A}\) generates a strongly continuous semigroup on the space $L^2(Q_{A,S};L^2(\Omega))$ (see \cite [\hbox{ chapter } 1.4]{b9}  ).
\begin{theorem} \label{resum1}(See \cite{o} and \cite{b9})
Under the assumption {\bf H1}-{\bf H2}-{\bf H3} the population (age-size and diffusion) operator with diffusion $(\mathbf{A},D(\mathbf{A}))$ is the infinitesimal generator of a strongly continuous semigroup $\textbf{U}$ on $K.$	
\end{theorem}
It is already established that the null controllability in time $T$ of $(\mathbf{A},\mathbf{B})$ is equivalent to the final-state observability in time $T$ of the pair $(\mathbf{A}^{*},\mathbf{B}^{*})$, where $\mathbf{A}^*$ and $\mathbf{B}^*$ are the adjoint operators of $\mathbf{A}$ and $\mathbf{B}$, respectively (see, for instance, \cite[\hbox{ Section } 11.2]{b8}. For that, it is important to determine the adjoint of the operator $\mathbf{A}$. We obtain by proceeding as in \cite{dy} the following result:		
\begin{proposition}\label{a7}
The domain of the adjoint operator \(\mathbf{A}^*\), denoted \(D(\mathbf{A}^*)\), is given by:
\[
D(\mathbf{A}^*) := 
\left\{
\varphi \in K \; \middle| \; 
\begin{aligned}
& \varphi \in L^2(Q_{A,S}; H^2(\Omega)), \\
& \frac{\partial \varphi}{\partial a} + g(s) \frac{\partial \varphi}{\partial s} + L\varphi - \mu(a, s)\varphi \in K, \\
& \frac{\partial \varphi}{\partial \nu} \Big|_{\partial \Omega} = 0, \\
& \varphi(\cdot, A, \cdot) = 0, \quad \varphi(\cdot, S, \cdot) = 0.
\end{aligned}
\right\}.
\]

Moreover, for every \(\varphi \in D(\mathbf{A}^*)\), the adjoint operator \(\mathbf{A}^*\) is expressed as:
\[
\mathbf{A}^* \varphi = L \varphi 
+ \frac{\partial \varphi}{\partial a} 
+ g(s) \frac{\partial \varphi}{\partial s} 
- \mu(a, s) \varphi 
+ \int_{0}^{S} \beta(a, s, u) \varphi(x, 0, u) \, du.
\]
\end{proposition}
	\begin{proof}
    As demonstrated in \cite{dy}, the same technique can be effectively utilized to establish this proposition.
	 
	\end{proof}
	According to the Proposition \ref{a7} the adjoint system of $(\ref{2})$ is given by: 
\begin{equation}
\left\lbrace\begin{array}{ll}
\dfrac{\partial q}{\partial t}-\dfrac{\partial q}{\partial a}-g(s)\dfrac{\partial q}{\partial s}-L q+\mu(a,s)q=\displaystyle\int\limits_{0}^{S}\beta(a,s,u)q(x,0,u,t)du&\hbox{ in }Q ,\\ 
\dfrac{\partial q}{\partial\nu}=0&\hbox{ on }\Sigma,\\ 
q\left( x,A,s,t\right) =0&\hbox{ in } Q_{S,T} \\
q(x,a,S,t)=0& \hbox{ in } Q_{A,T}\\
q\left(x,a,s,0\right)=q_0(x,a,s)&
\hbox{ in }Q_{A,S};
\end{array}\right.
\label{112}
\end{equation}
Let $Q'$ be the domain $Q$ without the space variable.\smallbreak
We split the domain $Q'$ as follows: \[A_1=\left\{(a,s,t)\in Q' \text{ such that } 0< t< A-a \text{ and } 0< t< G(S)-G(s)\right\},\] \[A'_1=\left\{(a,s,t)\in Q' \text{ such that } G(S)-G(s)>t>A-a>0\text{ or } t>G(S)-G(s)>A-a>0\right\}\] and
\[A'_2=\left\{(a,s,t)\in Q' \text{ such that } A-a>t>G(S)-G(s)>0\text{ or } t>A-a>G(S)-G(s)>0\right\},\]  where \[G(s)=\int_{0}^{s}\frac{du}{g(u)}.\] To simplify, we denote by $A_2=A'_1\cup A'_2,$ then $Q'=A_1\cup A_2$.\smallbreak
Let $R$ be the operator defined in \begin{equation}
  D(\mathbf{A}^*) \text{ by } R\psi=(-\mu(a,s)+L)\psi.\label{33}
  \end{equation}
   The operator $R$ is an infinitesimal generator of a strongly continuous semigroup in $K$ and
we have the following result:\smallbreak
	\begin{proposition}
		For every $q_0\in K, $ under the assumptions {\bf H1}-{\bf H2}-{\bf H3} the system $(\ref{2})$ admits a unique solution $q.$ Moreover integrating along the characteristic lines, the solution $q$ of $(\ref{112})$ is given by:
		\begin{align}
		q(t) =\left\lbrace\begin{array}{l}
			 q_0(.,a+t,G^{-1}(G(s)+t))e^{t R}+\displaystyle\int\limits_{0}^{t}\left(e^{(t-l)R}\displaystyle\int\limits_{0}^{S}\beta(a+t-l,G^{-1}(G(s)+t-\lambda),u)q(x,0,u,l)du)\right)dl \hbox{ in } A_1, \\
			\displaystyle\int\limits_{\max\left\{t-G(S)+G(s),t-A+a\right\}}^{t}\left(e^{(t-l)R}\displaystyle\int\limits_{0}^{S}\beta(a+t-l,G^{-1}(G(s)+t-\lambda),u)q(x,0,u,l)du)\right)dl \hbox{ in } A_2;\label{alp}
		\end{array}\right.
		\end{align}
	\end{proposition}
	where \[e^{tR}=\dfrac{\pi_1(a)}{\pi_1(a-t)}\dfrac{\pi_2(s)}{\pi_2(s-t)}e^{tL} \hbox{ with } \pi_1(a)=\exp\left(-\int\limits_{0}^{a}\mu_1(r)dr\right) \hbox{ and   }\pi_2(s)=\exp\left(-\int\limits_{0}^{s}\mu_2(r)dr\right).\]	
 See the proof in the appendix.
 \subsection{Proof of the Theorem \ref{MainT1}}
\subsubsection{An Observability Inequality}\label{Secb}
 In the context where we exclusively focus on the observability inequality problem without considering space, the desired observability inequality can be interpreted as follows: estimate the minimal time required for all the backward characteristics staring from $(a,s,T)$ enters the observation region or exit \((0, A) \times (0, S) \times (0, T)\) via points \((a, S, t)\), while adhering to the constraint \(a < \min\{a_2, \hat{a}\}\).\smallbreak
As mentioned above, the null controllability of a pair \((\mathbf{A}, \mathbf{B})\) is equivalent to the final-state observability of the pair \((\mathbf{A}^*, \mathbf{B}^*)\); see \cite{b8}. This equivalence is crucial for our analysis, as it allows us to transition from considering controllability to focusing on observability. Recall that the final-state observability of \((\mathbf{A}^*, \mathbf{B}^*)\) is defined as follows:
\begin{definition}
	\cite[\hbox{Definition 6.1.1}]{b8}\smallbreak
	The pair \((\mathbf{A}^*, \mathbf{B}^*)\) is final-state observable in time \(T\) if there exists a \(K_T > 0\) such that 
	\begin{equation}
		\int_{0}^{T} \|\mathbf{B}^* \mathbf{U}_{t}^{*} q_{0}\|^2 \geq K_{T}^{2} \|\mathbf{U}_{T}^{*} q_0\|^2 \quad (q_{0} \in D(\mathbf{A}^*)).
	\end{equation}
\end{definition}
Building on this definition, we will combine the classical observability inequality of the heat equation with the method of characteristics. This combination is essential for establishing the desired observability inequality, as it leverages both the well-established properties of the heat equation and the dynamic behavior of characteristics. By doing so, we aim to derive a robust framework that ensures the observability criteria are met within the specified constraints.
	We consider the following adjoint system of $(\ref{2})$ given by:
\begin{equation}
	\left\lbrace
\begin{array}{ll}
\dfrac{\partial q}{\partial t}-\dfrac{\partial q}{\partial a}-g(s)\dfrac{\partial q}{\partial s}-L q+\mu(a,s)q=\int\limits_{0}^{S}\beta(a,s,\hat{s})q(x,0,\hat{s},t)d\hat{s}&\hbox{ in }Q ,\\ 
\dfrac{\partial q}{\partial\nu}=0&\hbox{ on }\Sigma,\\ 
q\left( x,A,s,t\right) =0&\hbox{ in } Q_{S,T} \\
q(x,a,S,t)=0&\hbox{ in } Q_{A,T}\\
q\left(x,a,s,0\right)=q_0(x,a,s)&\hbox{ in }Q_{A,S}.
\end{array}\right.\label{3}
\end{equation} 
We recall that \[T_0=\max\left\{S^{*}_1,S^{*}_2\right\}\] and \[T_1=\max\left\{a_1+S^{*}_2,S^{*}_1\right\}.\]
Taking into account \cite[\hbox{ Theorem } 11.2.1]{b8},
the result of Theorem \ref{MainT1} is then reduced to the following theorem, which will be proved later.
\begin{theorem}\label{a2}
Under the assumption of Theorem \ref{MainT1}, the pair $(\mathbf{A}^{*},\mathbf{B}^{*})$ is final state-observable for every $T>A-a_2+T_1+T_0.$\smallbreak
In other words, for every $T>A-a_2+T_1+T_0$ there exist $K_T>0$ such that the solution $q$ of (\ref{3}) satisfies \begin{equation}\displaystyle\int\limits_{0}^{S}\int\limits_{0}^{A}\int_{\Omega}q^2(x,a,s,T)dxdads\leq K_T\displaystyle\int\limits_{0}^{T}\int\limits_{s_1}^{s_2}\int\limits_{a_1}^{a_2}\int_{\omega}q^2(x,a,s,t)dxdadsdt \quad \hbox{ for every } q_0\in D(\mathcal{A}^*).
\end{equation}
\end{theorem}
In the following subsection, we will establish the necessary results to prove the observability inequality.
\subsubsection{Necessary Propositions for the Proof of the Observability Inequality}
The principle is based on the estimation of the non-local term $\displaystyle\int\limits_{0}^{S}\beta(a,s,\hat{s})q(x,0,\hat{s},t)d\hat{s}$. Hence, the following result:
  \begin{proposition}\label{Prop1}
		Let the assumptions of Theorem \ref{MainT1} be satisfied and let \[a_1<\hat{a}\hbox{, }S^{*}_2<\min\{\hat{a}-a_1,a_2-a_1\}\hbox{, }S^{*}_1<\min\{a_2,\hat{a}\}\hbox{ and } T_1<\eta<T.\] Then there exists a constant $C_{\eta,T}>0$ dependent on $\eta\hbox{ and }$ such that for every $q_{0}\in K$, the solution $q$ of \eqref{112} satisfies the following inequality:
	\begin{equation}\label{Bq}
		\int\limits_{\eta}^{T}\int\limits_{0}^{S}\int_{\Omega}q^2(x,0,s,t)dxdsdt\leq C_{\eta,T}\displaystyle\int\limits_{0}^{T}\int\limits_{s_1}^{s_2} \int\limits_{a_1}^{a_2}\int_{\omega}q^2(x,a,s,t)dxdadsdt .
	\end{equation}
	More explicitly, 
If $s_1<\alpha^*<s_2$ or $s_1>\alpha^*> 0$, then for every $\eta$, with $T_1<\eta<T$ and for every $\delta>0$ such that $0<\alpha^*-\delta$, there exists a constant $C_{\delta,T}>0$ dependent on $\eta,$ such that for every $q_{0}\in K,$ the solution $q$ of \eqref{112} satisfies the following inequality:
		\begin{equation}\label{Bq12}
			\int\limits_{\eta}^{T}\int\limits_{0}^{\alpha^*-\delta}\int_{\Omega}q^2(x,0,s,t)dxdsdt\leq C_{\delta,T}\displaystyle\int\limits_{0}^{T}\int\limits_{s_1}^{s_2} \int\limits_{a_1}^{a_2}\int_{\omega}q^2(x,a,s,t)dxdadsdt ;
		\end{equation} 
		and, $$q(x,0,s,t)=0\hbox{ a.e. in }\Omega\times (\alpha^*,S)\times \left(S^{*}_{2}+a_1,T\right).$$ 
\end{proposition}
The proof of this fundamental result is provided in the appendix of the document.
\begin{remark}
    We notice that for every $X_0\in [s_1,S)$  and for every $\eta$ such that \[a_1+S^*_2<\eta<T.\] there exists $C_{\eta,T}>0$
    \begin{equation}\label{Bq}
		\int\limits_{\eta}^{T}\int\limits_{X_0}^{S}\int_{\Omega}q^2(x,0,s,t)dxdsdt\leq C_{\eta,T}\displaystyle\int\limits_{0}^{T}\int\limits_{s_1}^{s_2} \int\limits_{a_1}^{a_2}\int_{\omega}q^2(x,a,s,t)dxdadsdt .
	\end{equation} 

\end{remark}
 Therefore, the proof strategy involves estimating the adjoint \( q \) at the final time \( T \) over a portion of the domain \([0, a_0] \times [0, S]\) and expressing \( q(a, s, T) \) in terms of the non-local term \( \int_{0}^{S} g(u)\beta(a, s, u) q(x, 0, u, t) \, du \) (referred to here as the renewal term), the estimation of which is provided by Proposition \ref{Prop1}. In fact, we need these results for the next steps.
\begin{proposition}\label{a3}
	Let us assume the hypothesis of Theorem \ref{a2}. Let, $T>T_1.$ Then for every $q_0\in L^2(\Omega\times (0,A)\times (0,S)),$ the solution $q$ of the system $(\ref{3}),$ obeys
	\begin{equation}
\displaystyle\int\limits_{s_1}^{S}\int\limits_{0}^{a_1}\int_{\Omega}q^2(x,a,s,T)dxdads\leq C_T\displaystyle\int\limits_{0}^{T}\int\limits_{s_1}^{s_2}\int\limits_{a_1}^{a_2}\int_{\omega}q^2(x,a,s,t)dxdadsdt \label{Cbb1}
	\end{equation}
 Moreover, let $a^{*}_0\in (0,\min\{a_2,\hat{a}\}-S^{*}_1),$ then for every $q_0\in L^2(\Omega\times (0,A)\times (0,S)),$ the solution $q$ of the system $(\ref{3}),$ obeys
	\begin{equation}
		\displaystyle\int\limits_{0}^{a^{*}_0}\int\limits_{0}^{s_1}\int_{\Omega}q^2(x,a,s,T)dxdads\leq C_T\displaystyle\int\limits_{0}^{T}\int\limits_{s_1}^{s_2}\int\limits_{a_1}^{a_2}\int_{\omega}q^2(x,a,s,t)dxdadsdt. \label{C8aa}
  \end{equation}
  Let $b_0\in (a^*_0,\min\{a_2,\hat{a}\})$ such that $G(s_1)=b_0-a^*_0;$ then we have in addition to the above
  \begin{equation}
\displaystyle\int\limits_{a^*_0}^{b_0}\int\limits_{G^{-1}(a-a^*_0)}^{s_1}\int_{\Omega}q^2(x,a,s,T)dxdads\leq C_T\displaystyle\int\limits_{0}^{T}\int\limits_{s_1}^{s_2}\int\limits_{a_1}^{a_2}\int_{\omega}q^2(x,a,s,t)dxdadsdt \label{xcxa}
  \end{equation}
\end{proposition}
The proof of this proposition is provided in detail in the appendix.
Now, consider the following cascade system:\smallbreak
\begin{equation}
	\left\lbrace
	\begin{array}{ll}
		\dfrac{\partial q}{\partial t}-\dfrac{\partial q}{\partial a}-g(s)\dfrac{\partial q}{\partial s}-L q+\mu(a,s)q=0&\hbox{ in }Q ,\\
		\dfrac{\partial q}{\partial\nu}=0&\hbox{ on }\Sigma,\\ 
		q\left( x,A,s,t\right) =0&\hbox{ in } Q_{S,T} \\
		q(x,a,S,t)=0&\hbox{ in } Q_{A,T}\\
		q\left(x,a,s,0\right)=q_0(x,a,s)&\hbox{ in }Q_{A,S}.
	\end{array}\right.\label{3a22}
\end{equation} 
We also need the following result for the Proof of the Theorem \ref{a2}.
\begin{proposition}\label{a4}
We assume the assumptions of Theorem \ref{a2}. Let $T>S^{*}_2$ and $a_1<a_0<a_2-S^{*}_2.$ There exists $C_T>0$ such that the solution $q$ of the system $(\ref{3a22})$ verifies the following inequality.
\begin{equation}
	\displaystyle\int\limits_{s_1}^{S}\int\limits_{a_1}^{a_0}\int_{\Omega}q^2(x,a,s,T)dxdads\leq C_T\displaystyle\int\limits_{0}^{T}\int\limits_{s_1}^{s_2}\int\limits_{a_1}^{a_2}\int_{\omega}q^2(x,a,s,t)dxdadsdt \label{Ccca}
\end{equation}
\end{proposition}
\begin{proof}{Proposition \ref{a4} }\smallbreak
By adopting the same strategy as above, the results of Proposition 3.6. of \cite{abst} considering the adjoint system on $\Omega\times(0,A)\times (s_1,S)\times (0,T)$ give the inequality (\ref{Ccca}).
\end{proof}
Let us now give preliminary results for the proof of Theorem \ref{a2}.\smallbreak
Let $q=u_1+u_2$ where $u_1$ and $u_2$ verify
\begin{equation}
	\left\lbrace
	\begin{array}{ll}
		\dfrac{\partial u_1}{\partial t}-\dfrac{\partial u_1}{\partial a}-g(s)\dfrac{\partial u_1}{\partial s}-L u_1+\mu(a,s)u_1=0&\hbox{ in }Q_{A.S}\times (\eta,T) ,\\ 
		\dfrac{\partial u_1}{\partial\nu}=0&\hbox{ on }\partial\Omega\times (0,A)\times(0,S)\times (\eta,T),,\\ 
		u_1\left( x,A,s,t\right) =0&\hbox{ in } \Omega\times(0,S)\times (\eta,T) \\
		u_1(x,a,S,t)=0&\hbox{ in } Q_{A,T}\\
		u_1\left(x,a,s,\eta\right)=q_{\eta}&\hbox{ in }Q_{A,S}.
	\end{array}\right.\label{322aa}
\end{equation} 
where $q_{\eta}=q(x,a,s,\eta)$ in $Q_{A.S}$ and
\begin{equation}
	\left\lbrace
	\begin{array}{ll}
		\dfrac{\partial u_2}{\partial t}-\dfrac{\partial u_2}{\partial a}-g(s)\dfrac{\partial u_2}{\partial s}-L u_2+\mu(a,s)u_2=\int\limits_{0}^{S}\beta(a,s,\hat{s})q(x,0,\hat{s},t)d\hat{s}&\hbox{ in }Q_{A.S}\times (\eta,T) ,\\ 
		\dfrac{\partial u_2}{\partial\nu}=0&\hbox{ on }\partial\Omega\times (0,A)\times(0,S)\times (\eta,T),\\ 
		u_2\left( x,A,s,t\right) =0&\hbox{ in } \Omega\times(0,S)\times (\eta,T) \\
		u_2(x,a,S,t)=0&\hbox{ in } Q_{A,T}\\
		u_2\left(x,a,s,\eta\right)=0&\hbox{ in }Q_{A,S}.
	\end{array}\right.\label{322bb}
\end{equation}
where $V(x,a,s,t)=\int\limits_{0}^{S}\beta(a,s,\hat{s})q(x,0,\hat{s},t)d\hat{s}.$\smallbreak 
Using Duhamel's formula, we can write \[u_2(x,a,s,t)=\int\limits_{\eta}^{t}\mathbb{T}_{t-f}V(.,.,.,f)df\] where $\mathbb{T}$ is the semigroup generated by the operator $\mathbf{A}^{*}.$ Furthermore, the solution $u_2$ of the system $(\ref{322bb})$ verifies the following estimates:
\begin{proposition}\label{a5}
	Under the assumptions $(H_1)$, $(H_2)$ and $(h_1)$, there exists $C=e^{\frac{T}{2}\left(1+\|\frac{g'}{2}\|_{\infty}\right)}\|\beta\|^{2}_{\infty}A$ such that the solution $u_2$ of the system $(\ref{322bb})$ verifies the following estimate.
	\begin{equation}
	\int\limits_{\eta}^{T}\int\limits_{s_1}^{s_2}\int\limits_{a_1}^{a_2}\int_{\omega}u_{2}^2(x,a,s,t)dxdadsdt\leq \int\limits_{\eta}^{T}\int\limits_{0}^{S}\int\limits_{0}^{A}\int_{\Omega}u_{2}^2(x,a,s,t)dxdadsdt\leq C \int\limits_{\eta}^{T}\int\limits_{0}^{S}\int_{\Omega}q^2(x,0,s,t)dxdsdt.	
	\end{equation}
\end{proposition}
\begin{proof}{Proof of Proposition \ref{a5}.}
We denote by $$u_2=\hat{u}_{2}e^{\lambda t}\hbox{ with }\lambda>0.$$ The function $u_2$ verifies \begin{equation}\dfrac{\partial \hat{u}_2}{\partial t}-\dfrac{\partial \hat{u}_2}{\partial a}-g(s)\dfrac{\partial \hat{u}_2}{\partial s}-\Delta \hat{u}_2+(\mu(a,s)+\lambda)\hat{u}_2=e^{-\lambda t}\int\limits_{0}^{S}\beta(a,s,\hat{s})q(x,0,\hat{s},t)d\hat{s}\hbox{ in }Q_{A.S}\times (\eta,T). \label{humm}
\end{equation}
Multiplying $(\ref{humm})$ by $\hat{u}_2$ and integrating over $Q_{A,S}\times (\eta,T)$, we get 
\[\int\limits_{0}^{S}\int\limits_{0}^{A}\int_{\Omega}\hat{u}_{2}^2(x,a,s,T)dxdads+\int\limits_{0}^{T}\int\limits_{0}^{S}\int_{\Omega}\hat{u}_{2}^2(x,0,s,t)dxdsdt+\int\limits_{0}^{T}\int\limits_{0}^{A}\int_{\Omega}\frac{g(0)}{2}\hat{u}_{2}^2(x,a,0,t)dxdadt\]\[+\int\limits_{\eta}^{T}\int_{Q_{A,S}}|\nabla\hat{u}_{2}|^2dxdadsdt+\int\limits_{\eta}^{T}\int_{Q_{A,S}}(\frac{-g'(s)}{2}+\mu_1(a)+\mu_2(s)+\lambda)\hat{u}_{2}^2dxdadsdt\]\[=\int\limits_{\eta}^{T}\int_{Q_{A,S}}\left(e^{-\lambda t}\int\limits_{0}^{S}\beta(a,s,\hat{s})q(x,0,\hat{s},t)d\hat{s}\right)\hat{u}_2dxdadsdt.\]
Using Young inequality and choosing $\lambda=\frac{3}{2}+\frac{1}{2}\|g'\|_{\infty},$
 we obtain
\[ \int\limits_{\eta}^{T}\int_{Q_{A,S}}\hat{u}_{2}^2dxdadsdt\leq \|\beta\|^{2}_{\infty}A \int\limits_{\eta}^{T}\int\limits_{0}^{S}\int_{\Omega} q^2(x,0,s,t)dxdsdt.\]
Finally, we get 
\[ \int\limits_{\eta}^{T}\int_{Q_{A,S}}u_{2}^2dxdadsdt\leq e^{\frac{3}{2}T}\|\beta\|^{2}_{\infty}A \int\limits_{\eta}^{T}\int\limits_{0}^{S}\int_{\Omega} q^2(x,0,s,t)dxdsdt.\]
\end{proof}
We are now in a position to provide the proof of the main result of this section.
\subsubsection{Proof of the Theorem \ref{a2}: Observability inequality}
\begin{proof}{Proof of the Theorem \ref{a2}} \smallbreak
We will use hypothesis \(S^{*}_2 < \max\{a_2, \hat{a}\} - a_1\) and \(S_1^* < \min\{\hat{a},a_2\}\) by dividing \[(0,A) \times (0,S) \hbox{ into }(0,A) \times (0,s_1) \cup (0,A) \times (s_1,S).\]
The proof will be in two parts.\smallbreak
We split the term to be estimated as the following
\[\int\limits_{0}^{S}\int\limits_{0}^{A}\int_{\Omega}q^2(x,a,s,T)dxdads=\int\limits_{s_1}^{S}\int\limits_{0}^{A}\int_{\Omega}q^2(x,a,s,T)dxdads+\int\limits_{0}^{s_1}\int\limits_{0}^{A}\int_{\Omega}q^2(x,a,s,T)dxdads.\]
\textbf{Estimation of $ \int\limits_{s_1}^{S}\int\limits_{0}^{A}\int_{\Omega}q^2(x,a,s,T)dxdads$}\smallbreak
We have \[\int\limits_{s_1}^{S}\int\limits_{0}^{A}\int_{\Omega}q^2(x,a,s,T)dxdads=\int\limits_{s_1}^{S}\int\limits_{0}^{a_1}\int_{\Omega}q^2(x,a,s,T)dxdads+\int\limits_{s_1}^{S}\int\limits_{a_1}^{A}\int_{\Omega}q^2(x,a,s,T)dxdads.\]
Using Proposition 3.3. we obtain the estimate; 
\begin{equation}\int\limits_{s_1}^{S}\int\limits_{0}^{a_1}\int_{\Omega}q^2(x,a,s,T)dxdads\leq C\int\limits_{0}^{A}\int\limits_{0}^{S}\int\limits_{0}^{A}\int_{\Omega}q^2(x,a,s,t)dxdadsdt.\label{couppa}\end{equation}
We are now left with estimating \[\int\limits_{s_1}^{S}\int\limits_{a_1}^{A}\int_{\Omega}q^2(x,a,s,T)dxdads.\]
But since $q=u_1+u_2\hbox{ on }(\eta,T)$ we must therefore estimate
\[\int\limits_{s_1}^{S}\int\limits_{a_1}^{A}\int_{\Omega}u_{1}^{2}(x,a,s,T)dxdads+\int\limits_{s_1}^{S}\int\limits_{a_1}^{A}\int_{\Omega}u_{2}^{2}(x,a,s,T)dxdads.\]
We have \[\int\limits_{s_1}^{S}\int\limits_{a_1}^{A}\int_{\Omega}u_{2}^{2}(x,a,s,T)dxdads\leq C(A-a_2)\int\limits_{\eta}^{T}\int\limits_{s_1}^{S}\int_{\Omega}q^2(x,0,s,t)dxdsdt.\]
And then, using Proposition \ref{Prop1}., we obtain the following.
\begin{equation}
	\int\limits_{s_1}^{S}\int\limits_{a_1}^{A}\int_{\Omega}u_{2}^{2}(x,a,s,T)dxdads\leq C_1\int\limits_{0}^{T}\int\limits_{s_1}^{s_2}\int\limits_{a_1}^{a_2}\int_{\omega}q^2(x,a,s,t)dxdadsdt.\label{debaay}
\end{equation}
As in $T>A-a_2+T_1+S^{*}_2$, there exists $\delta>0$ such that $$T>A-a_2+T_1+S^{*}_2+2\delta;$$ therefore $$T-(T_1+\delta)>A-(a_2-S^{*}_2+\delta).$$
Moreover for 
\[a\in (a_2-S^{*}_2-\delta,A)\hbox{ we have }T-(T_1+\delta)>A-(a_2-S^{*}_2-\delta)>A-a.\]
Then \[u_1(x,a,s,T)=0\hbox{ a.e. }x\in \Omega\hbox{, }a\in(a_2-(S^{*}_2-\delta),A).\]
Therefore 
\[\int\limits_{s_1}^{S}\int\limits_{a_1}^{A}\int_{\Omega}u_{1}^2(x,a,s,T)dxdads=\int\limits_{s_1}^{S}\int\limits_{a_1}^{a_2-S^{*}_2-\delta}\int_{\Omega}u_{1}^{2}(x,a,s,T)dxdads.\]
As $$T-\eta>A-a_2+S^{*}_2,$$ then using Proposition \ref{a4}., we obtain the following.
\begin{equation}
	\int\limits_{s_1}^{S}\int\limits_{a_1}^{A}\int_{\Omega}u_{1}^{2}(x,a,s,T)dxdads\leq K\int\limits_{\eta}^{T}\int\limits_{s_1}^{s_2}\int\limits_{a_1}^{a_2}\int_{\omega}u_{1}^{2}(x,a,s,t)dxdadsdt.\label{debaa}
\end{equation}
As $$q=u_1+u_2\Longleftrightarrow u_1=q-u_2$$ then 
\begin{equation}
	\int\limits_{s_1}^{S}\int\limits_{a_1}^{A}\int_{\Omega}u_{1}^{2}(x,a,s,T)dxdads\leq 2\left(\int\limits_{\eta}^{T}\int\limits_{s_1}^{s_2}\int\limits_{a_1}^{a_2}\int_{\omega}q^{2}(x,a,s,t)dxdadsdt+\int\limits_{\eta}^{T}\int\limits_{s_1}^{s_2}\int\limits_{a_1}^{a_2}\int_{\omega}u_{2}^{2}(x,a,s,t)dxdadsdt\right)..\label{debaa1}
\end{equation}
From the Proposition \ref{a5}.; we get 
\begin{equation}
	\int\limits_{s_1}^{S}\int\limits_{a_1}^{A}\int_{\Omega}u_{1}^{2}(x,a,s,T)dxdads\leq K\int\limits_{0}^{T}\int\limits_{s_1}^{s_2}\int\limits_{a_1}^{a_2}\int_{\omega}q^{2}(x,a,s,t)dxdadsdt..\label{debaa11}
\end{equation}
Combining the inequalities $(\ref{debaay})$ and $(\ref{debaa11})$, we obtain
\begin{equation}
	\int\limits_{s_1}^{S}\int\limits_{a_1}^{A}\int_{\Omega}q^{2}(x,a,s,T)dxdads\leq K_1 \int\limits_{0}^{T}\int\limits_{s_1}^{s_2}\int\limits_{a_1}^{a_2}\int_{\Omega}q^{2}(x,a,s,t)dxdadsdt.\label{coup1a}
\end{equation}
Finally, $(\ref{couppa})$ and $(\ref{coup1a})$ give
\begin{equation}
	\int\limits_{s_1}^{S}\int\limits_{0}^{A}\int_{\Omega}q^{2}(x,a,s,T)dxdads\leq K_T \int\limits_{0}^{T}\int\limits_{s_1}^{s_2}\int\limits_{a_1}^{a_2}\int_{\Omega}q^{2}(x,a,s,t)dxdadsdt.\label{simple}
\end{equation}
\textbf{Estimation of $ \int\limits_{0}^{s_1}\int\limits_{0}^{A}\int_{\Omega}q^2(x,a,s,T)dxdads$}\smallbreak
By Proposition \ref{a3}, there exists $a_{0}^*\in (0,a_2)$ such that
\begin{equation}\int\limits_{0}^{s_1}\int\limits_{0}^{a^{*}_0}\int_{\Omega}q^2(x,a,s,T)dxdads\leq C\int\limits_{0}^{A}\int\limits_{s_1}^{s_2}\int\limits_{a_1}^{a_2}\int_{\Omega}q^2(x,a,s,t)dxdadsdt.\label{coupp}\end{equation}
We are now left with the estimate of \[\int\limits_{0}^{s_1}\int\limits_{a^{*}_0}^{A}\int_{\Omega}q^2(x,a,s,T)dxdads.\]
But since $q=u_1+u_2$ we must therefore estimate
\[\int\limits_{0}^{s_1}\int\limits_{a^{*}_{0}}^{A}\int_{\Omega}u_{1}^{2}(x,a,s,T)dxdads+\int\limits_{0}^{s_1}\int\limits_{a^{*}_0}^{A}\int_{\Omega}u_{2}^{2}(x,a,s,T)dxdads.\]
We have \[\int\limits_{0}^{s_1}\int\limits_{a^{*}_0}^{A}\int_{\Omega}u_{2}^{2}(x,a,s,T)dxdads\leq C(A-a_2)\int\limits_{\eta}^{T}\int\limits_{s_1}^{S}\int_{\Omega}q^2(x,0,s,t)dxdsdt.\]
And then, using Proposition \ref{Prop1}., we obtain
\begin{equation}
	\int\limits_{0}^{s_1}\int\limits_{a^{*}_0}^{A}\int_{\Omega}u_{2}^{2}(x,a,s,T)dxdads\leq C_1\int\limits_{0}^{T}\int\limits_{s_1}^{s_2}\int\limits_{a_1}^{a_2}\int_{\omega}q^2(x,a,s,t)dxdadsdt.
\end{equation}
As in $T>A-a_2+T_1+S^{*}_1$, there exists $\delta>0$ such that $$T>A-a_2+T_1+S^{*}_1+2\delta;$$ therefore $$T-(T_1+\delta)>A-(a_2-S^{*}_1+\delta).$$
Moreover for 
\[a\in (a_2-S^{*}_1-\delta,A)\hbox{ we have }T-(T_1+\delta)>A-(a_2-S^{*}_1-\delta)>A-a.\]
Then \[u_1(x,a,s,T)=0\hbox{ a.e. }x\in \Omega\hbox{, }a\in(a_2-(S^{*}_1-\delta),A).\]

Therefore 
\[u_1(x,a,s,T)=0\hbox{ a.e. }x\in \Omega\hbox{, }a\in(a^{*}_0,A)\hbox{ and }s\in(0,s_1).\] 
Then
\begin{equation}
	\int\limits_{0}^{s_1}\int\limits_{0}^{A}\int_{\Omega}q^{2}(x,a,s,T)dxdads\leq K\int\limits_{0}^{T}\int\limits_{s_1}^{s_2}\int\limits_{a_1}^{a_2}\int_{\omega}q^{2}(x,a,s,t)dxdadsdt.\label{simplo}
\end{equation}
Finally, $(\ref{simple})$ and $(\ref{simplo})$ give
\begin{equation}
	\int\limits_{0}^{S}\int\limits_{0}^{A}\int_{\Omega}q^{2}(x,a,s,T)dxdads\leq K_T \int\limits_{0}^{T}\int\limits_{s_1}^{s_2}\int\limits_{a_1}^{a_2}\int_{\Omega}q^{2}(x,a,s,t)dxdadsdt.\label{coup}
\end{equation}
A figure illustrating the concept of observability (the observability inequality). We consider a vector field propagating within a box $[0, A] \times [0, S] \times [0, T]$. The field starts from the top at points with coordinates $(a, s, T)$, where $(a, s) \in [0, A] \times [0, S]$, and moves downward along the line defined by the vector $(-1, -g(s), 1)$.
\begin{figure}[H]
\begin{adjustbox}{center,left}
			\begin{overpic}[scale=0.25]{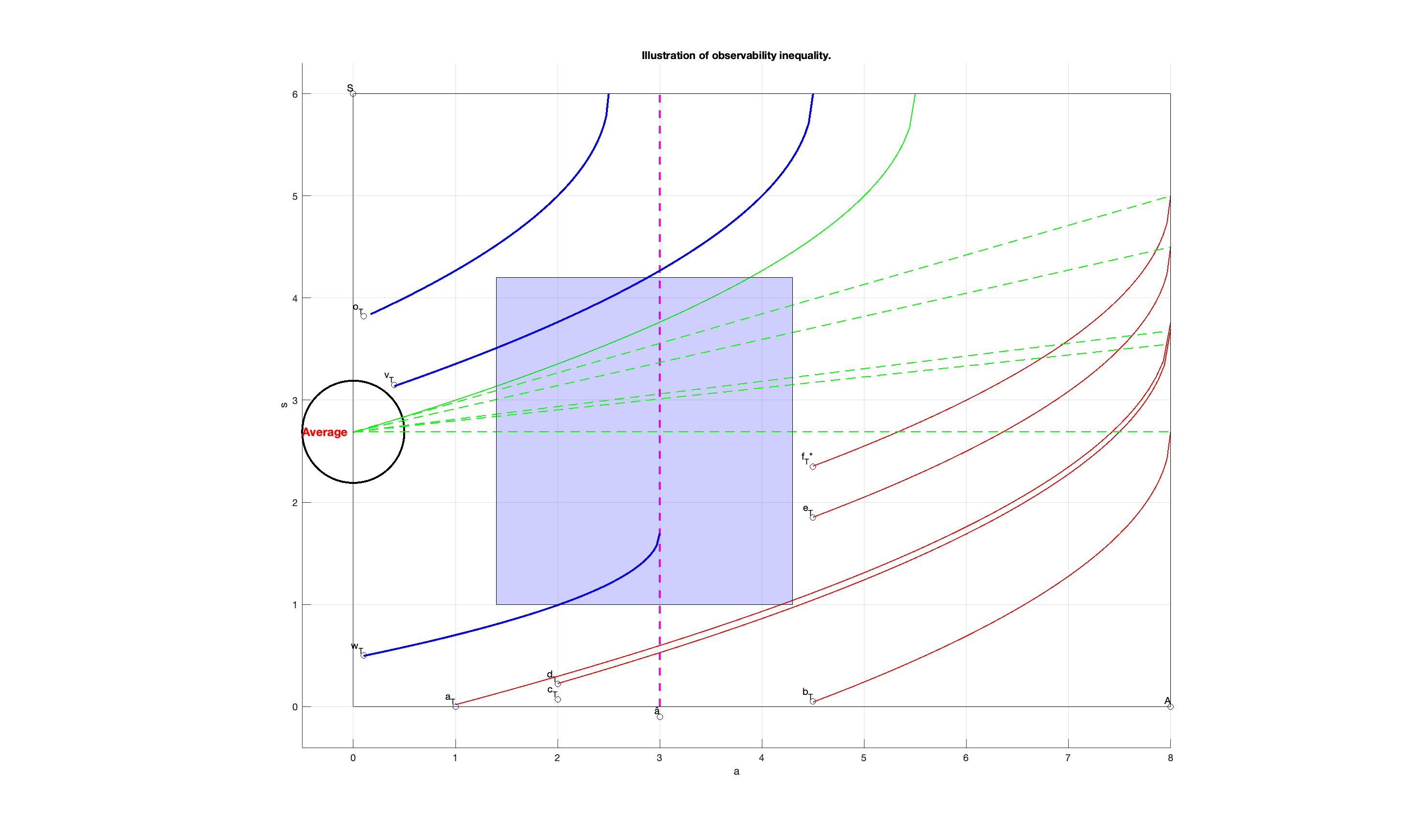}	
					\end{overpic}
\end{adjustbox}
\caption{The backward characteristics starting from \((a, s, T)\), are illustrated. The red lines correspond to characteristics (points \(a_T, b_T, c_T, d_T, e_T,\) and \(f_T\)) that reach the boundary at $a=A$ and are renewed via the renewal condition
$
\int\limits_{0}^{S} \beta(a, s, u) q(x, 0, u, t) \, du,
$
which represents an average value over the size variable of newborns. Once renewed, these characteristics can subsequently be observed as described in Proposition I.9.
The blue lines represent backward characteristics originating from \((a, s, T)\) (points \(o_T, v_T,\) and 
\(w_T\)) that either enter the control domain or leave the region \([0, A] \times [0, S] \times [0, T]\) through the section \(s = S\). 
In the latter case, the exit 
occurs at a point $(a, S, t),$
where $a\in(0,\hat{a})\hbox{ and } T-t>T_1$.}
\end{figure}
\end{proof}
\subsubsection{Outline of Proofs for the Introductory Comments}
In the case where the study of the system \ref{2} is reduced to the set $V$ define by
\[ V = \{(x,a,s,t) \in (0,A) \times (0,S) \times (0,T) \;|\;x\in \Omega;\; 0 \leq a \leq A;\; G^{-1}(a) \leq s \leq S;\; 0 \leq t \leq T \}. \] The proof of the result of null controllability will be to show the existence of $ K_T> $ 0 such that the state of the adjoint system \ref{3} verifies
\begin{equation}
	\int\limits_{0}^{A}\int\limits_{G^{-1}(a)}^{S}\int_{\Omega}q^{2}(x,a,s,T)dxdads\leq K_T \int\limits_{0}^{T}\int\limits_{a_1}^{a_2}\int\limits_{\max\{s_1,G^{-1}(a)\}}^{s_2}\int_{\Omega}q^{2}(x,a,s,t)dxdadsdt.
\end{equation}
for any $q_0\in L^2(L^2(Q_G))$  where \[Q_G= \{(x,a,s) \in (0,A) \times (0,S)  \;|\;x\in \Omega;\; 0 \leq a \leq A;\; G^{-1}(a) \leq s \leq S. \}\]
The following inequality,
\begin{equation}
	\int\limits_{s_1}^{S}\int\limits_{0}^{A}\int_{\Omega}q^{2}(x,a,s,T)dxdads\leq K_T \int\limits_{0}^{T}\int\limits_{s_1}^{s_2}\int\limits_{a_1}^{a_2}\int_{\Omega}q^{2}(x,a,s,t)dxdadsdt,
\end{equation}
established in the evidence of the main result allows us to lead to the following inequality
\begin{equation}
	\int\limits_{0}^{A}\int\limits_{s_1}^{S}\int_{\Omega}q^{2}(x,a,s,T)dxdads\leq K_T \int\limits_{0}^{T}\int\limits_{a_1}^{a_2}\int\limits_{\max\{s_1,G^{-1}(a)\}}\int_{\Omega}q^{2}(x,a,s,t)dxdadsdt,\label{wora}
\end{equation} 
for $T$ verifying \[T>A-a_2+T_1+S^*_2.\]
Now let $a^* \in (0, a_2)$ such that $G(s_1) = a^*$ which is equivalent to $G^{-1}(a^*) = s_1$. Then, we need to estimate $q(x,a,s,T)$ over $(x,a,s)\hbox{ such that }x\in\Omega \hbox{, }0\leq a\leq a^*\hbox{ and }G^{-1}(a)\leq s\leq s_1.$ represented in Figure 7-1 in green.\smallbreak
Considering the following characteristics $\gamma(\lambda)=(T+a-\lambda,G^{-1}(G(s)+T-\lambda),\lambda),$ and by setting \( w(\lambda, x) = \tilde{q}(x, \gamma(\lambda)) \) with \( \tilde{q} = e^{-\mu(a, s)} q \), we can show that for all \( T > T_1 \), we have
\begin{equation}
	\int\limits_{0}^{a^*}\int\limits_{G^{-1}(a)}^{s_1}\int_{\Omega}q^{2}(x,a,s,T)dxdads\leq K_T \int\limits_{0}^{T}\int\limits_{a_1}^{a_2}\int\limits_{\max\{s_1,G^{-1}(a)\}}^{s_2}\int_{\Omega}q^{2}(x,a,s,t)dxdadsdt,\label{worb}
\end{equation} 
Combining \ref{wora} and \ref{worb}
we obtain for every $T>A-a_2+T_1+S^*_2$
\begin{equation}
	\int\limits_{0}^{A}\int\limits_{G^{-1}(a)}^{S}\int_{\Omega}q^{2}(x,a,s,T)dxdads\leq K_T \int\limits_{0}^{T}\int\limits_{a_1}^{a_2}\int\limits_{\max\{s_1,G^{-1}(a)\}}^{s_2}\int_{\Omega}q^{2}(x,a,s,t)dxdadsdt.
\end{equation} 
\smallbreak
\smallbreak
\section{Degenerate Model}
\subsection{Presentation of the model}
In this second part, we consider the same type of model as the first, but with a degenerate diffusion term and a classical (simpler) birth term, which is less realistic than the first but mathematically provides a better controllability time. The controllability of this model was introduced by Fragnelli et al. In this work, we significantly improve the result obtained.
Indeed, this part is concerned with the null controllability of a degenerate, size, and age structured population model described as:
\begin{equation}
	\left\lbrace\begin{array}{ll}
		\dfrac{\partial y}{\partial t}+\dfrac{\partial y}{\partial a}+\dfrac{\partial g(s)y }{\partial s}-k(x)\dfrac{\partial^2 y}{\partial x^2}-b(x)\dfrac{\partial y}{\partial x}+\mu(a,s)y=mu, &\hbox{ in }Q ,\\ 
		y(0,a,s,t)=y(1,a,s,t)=0, &\hbox{ on }\Sigma,\\ 
		y\left( x,0,s,t\right) =\displaystyle\int\limits_{0}^{A}\beta(a)y(x,a,s,t)da, &\hbox{ in } Q_{S,T} \\
		y\left(x,a,s,0\right)=y_{0}\left(x,a,s\right),&
		\hbox{ in }Q_{A,S};\\
		y(x,a,0,t)=0,& \hbox{ in } Q_{A,T},
	\end{array}\right.
	\label{2cc}
\end{equation}
where $Q_{A,S}:=(0,1)\times (0,A)\times (0,S)$, $Q:=(0,1)\times (0,A)\times (0,S)\times(0,T)$, $\Sigma:=(0,A)\times (0,S)\times(0,T)$, $Q_{S,T}:=(0,1)\times (0,S)\times (0,T)$,  and $Q_{A,T}:=(0,1)\times (0,A)\times (0,T)$. Here $y(x,a,s,t)$ represents the population density of certain species of age $a\in (0,A)$ and size $s\in (0,S)$ at time $t\geq 0$ and location $x\in (0,1)$, where $A>0$ and $S>0$ are the maximum age of life and the maximum size of individuals, respectively. The age-dependent function $\beta$ denotes natural fertility, and therefore the formula $\int_{0}^{A}\beta(a)yda $ determines the density of newborn individuals of size $s$ at time $t$ and the location of the point $x$. The age and size-dependent function $\mu$ denotes the death rate and we assume that it satisfies $\mu(a,s)=\mu_1(a)+\mu_2(s)$. The space-dependent diffusion coefficient $k$ can degenerate at one point of the domain or at both extreme points. For $0\leq a_1<a_2\leq A$ and $\omega=(l_1,l_2)\subset (0,1)$, we denote by $m$ the characteristic function of $Q_{1}:=\omega\times (a_1,a_2)\times (s_1,s_2)$, which is the region where the control $u$ acts. The initial distribution of the population is $y_{0}\left(x,a,s\right)$. For more details on the modeling of such a system, we refer to G. Webb \cite{b9}.
In addition to the hypotheses about $\mu$, we also have the following.
\begin{align*}
(	{\bf H5}):&
	\left\lbrace\begin{array}{l}
     (g(x)>0\hbox{ for all }x\in [0,S]\hbox{ and }g\in C^1([0,S]);\\
     \int\limits_{a}^{b}\dfrac{ds}{g(s)}<+\infty \hbox{ for all }[a,b]\subset [0,S].
	\end{array}
	\right. \qquad
(	{\bf H6}):&
	\left\lbrace\begin{array}{l}
		\beta(\cdot)\in C\left([0,A]\right) 
		\beta(s) \ge 0 \quad {\rm  for \; a.e.} \quad a\in (0,A),\\
	\end{array}
	\right. \qquad .\cr
(	{\bf H7}):&
\begin{array}{c}
	\beta(a)=0 \quad \forall a\in (0,\hat{a}) \text{ for some } \hat{a}\in (0,A). 
\end{array}
\end{align*}
For more details on the modeling of such a system and the biological significance of the hypotheses, we refer to G. Webb \cite{b9}.
Before going further and stating the main result of this paper, we need to introduce some notation and assumptions. We start with the following assumptions about the diffusion coefficient $k$.
\begin{mainassumptions}\label{Main1}
To study the well-posedness of \eqref{2cc} we assume that $k\in C([0,1])$ and there exist $M_1\hbox{, }M_2\in (0,2)$ such that \[k\in C^1(0,1)\hbox{, }k>0\hbox{ in }(0,1)\hbox{, }k(0)=k(1)=0,\]
\[xk'(x)\leq M_1k(x)\hbox{ and }(x-1)k'(x)\leq M_2k(x)\]

for all $x\in[0,1],$ or \[k\in C^1((0,1])\hbox{, }k>0\hbox{ in }(0,1]\hbox{, }k(0)=0\hbox{ and }xk'(x)\leq M_1k(x)\]

for all $x\in[0,1],$ or

\[k\in C^1((0,1])\hbox{, }k>0\hbox{ in }[0,1)\hbox{, }k(1)=0\hbox{ and }(x-1)k'(x)\leq M_2k(x)\]

for all $x\in[0,1]$. On the other hand, \[b\in C([0,1])\cap C^2(0,1)\hbox{ and }\dfrac{b}{k}\in L^1(0,1).\] 

\end{mainassumptions}

Next, consider the weight function $\gamma$ defined by (cf. \cite{Feller})

$$ \gamma(x):=\exp\left(\int\limits_{x}^{\frac{3}{4}}\dfrac{b(y)}{k(y)}dy\right)\qquad x\in[0,1],$$

and define 

$$\sigma(x):=k(x)\dfrac{1}{\gamma(x)}, \qquad x\in[0,1].$$

We then select the following weighted Hilbert spaces.

\begin{align*}
L^{2}_{\frac{1}{\sigma}}(0,1)&:=\{u\in L^2(0,1)|\;\|u\|_{\frac{1}{\sigma}}<+\infty\},\qquad\|u\|^{2}_{\frac{1}{\sigma}}=\int\limits_{0}^{1}u^2\frac{1}{\sigma}dx,\cr
H^{1}_{\frac{1}{\sigma}}(0,1)&:=L^{2}_{\frac{1}{\sigma}}(0,1)\cap H^{1}_{0}(0,1) \qquad\|u\|^{2}_{1,\frac{1}{\sigma}}=\|u\|^{2}_{\frac{1}{\sigma}}+\int\limits_{0}^{1}u^{2}_{x}dx,
\end{align*}

For $u$ sufficiently smooth, e.g., $u\in W^{2,1}_{loc}(0,1)$, we set

$$L_1 u:=ku_{xx}-b(x)u_{x},$$ 

for almost every $x\in (0,1)$. Moreover, for all $(u,v)\in H^{2}_{\frac{1}{\sigma}}(0,1)\times H^{1}_{\frac{1}{\sigma}}(0,1),$ 

$$(L_1 u,v)_{\frac{1}{\sigma}}=-\int\limits_{0}^{1}\gamma u_x v_x dx,$$

where

\[H^{2}_{\frac{1}{\sigma}}(0,1):=\{u\in H^{1}_{\frac{1}{\sigma}}(0,1)|L u\in L^{2}_{\frac{1}{\sigma}}(0,1)  \} \quad \|u\|^{2}_{2,\frac{1}{\sigma}}=\|u\|^{2}_{1,\frac{1}{\sigma}}+\|L u\|^{2}_{\frac{1}{\sigma}}.\]
On the other hand, to establish the null controllability result of \eqref{2cc}, we additionally assume that $k$ satisfies the following assumptions (see \cite{canard1}).
\begin{mainassumptions}\label{Main2}
		The function $k\in C^0([0,1])\cap C^3(0,1)$ is such that $k(0)=k(1)=0,$ and $k>0$ on $(0,1)$. Moreover,
  \begin{enumerate}
		\item there exists $\epsilon\in (0,1)\hbox{, }N_1\in (0,2)\hbox{ and }C_1>0$ such that the function \[\dfrac{x(b-k_x)}{k}\in L^{\infty}(0,\epsilon)\hbox{, }x(k)_x\leq N_1k(x)\hbox{ for all }x\in (0,\epsilon),\]

		there exists a function \[C_1=C_1(\epsilon)>0\hbox{ defined in }(0,\epsilon)\hbox{ such that }C_1(\epsilon')\longrightarrow 0\hbox{ as }\epsilon'\longrightarrow 0^+\]

		and

		\[\left|\left(\dfrac{x(b-k_x)}{k}\right)_{xx}-\dfrac{b}{a}\left(\dfrac{x(b-a_x)}{a}\right)\right|\leq \dfrac{C_1(\epsilon')}{x^2}\hbox{ for all }x\in(0,\epsilon').\]

		\item

		\[\dfrac{(x-1)(b-k)_x}{k}\in L^{\infty}(1-\epsilon,1)\hbox{ and there exists } N_2\in(0,2) \hbox{ and }C_2>0\hbox{ such that },\]
		\[\left(\dfrac{(x-1)(k)_x}{k}\right)_{xx}<\dfrac{C_2}{k}\hbox{ for all }x\in (1-\epsilon,1),\]

		and there exists a function \[C_2=C_2(\epsilon)>0\hbox{ defined in }(1-\epsilon,1)\hbox{ such that }C_1(\epsilon')\longrightarrow 0\hbox{ as }\epsilon'\longrightarrow 0^+\]

		\[\left|\left(\dfrac{(x-1)(b-k_x)}{k}\right)_{xx}-\dfrac{b}{a}\left(\dfrac{(x-1)(b-a_x)}{a}\right)\right|\leq \dfrac{C_1(\epsilon')}{(x-1)^2}\hbox{ for all }x\in(1-\epsilon',1).\]
	\end{enumerate}
\end{mainassumptions}
After such a long but necessary preparation, we can now clearly state the main results of this paper. The well-posedness of \eqref{2cc} follows from the following result (see, e.g., \cite{PugT}).
\begin{theorem}\label{MainTb}
	Let conditions {\bf (H1)}-{\bf (H2)}-{\bf (H5)}-{\bf (H6)}-{\bf (H7)} be satisfied. Furthermore, we assume that Assumption \ref{Main1} holds. For all $u\in L^{2}_{\frac{1}{\sigma}}(Q)$ and $y_0\in L^{2}_{\frac{1}{\sigma}}(Q_{A,S})$, the system \eqref{2cc} has a unique solution \[y\in \mathcal{U}:=C\left([0,T];L^{2}_{\frac{1}{\sigma}}(Q_{A,S})\right) \cap L^2\left(0,T;H^{1}(0,A)\times H^1(0,S) \times H^{1}_{\frac{1}{\sigma}}(0,1)\right).\]
If, in addition, $u=0$, then 
		\begin{align*}
			y\in C^1\left([0,T];L^{2}_{\frac{1}{\sigma}}(Q_{A,S})\right).
		\end{align*}
\end{theorem}
We select the following definition.
\begin{definition}

The real $\hat{a}$ is the minimal age at which individuals become fertile. We will call it \emph{ minimal age of fertility}.

\end{definition}

For $0 \leq s_1 < s_2 \leq S $, let it remain

\[S^{*}_{1}=\int\limits_{0}^{s_1}\frac{ds}{g(s)}\hbox{ ,}S^{*}_{2}=\int\limits_{s_2}^{S}\frac{ds}{g(s)}\hbox{ and }T_0=\max\left\{S^{*}_{1}, S^{*}_{2}\right\} \hbox{ and } T_1=\max\left\{S^{*}_{1},a_1+S^{*}_{2}\right\}.\]
\subsection{Main result and comments}
The following is the main result of this paper.
\begin{theorem}\label{MainTa}
Let conditions {\bf (H1)}-{\bf (H4)} be satisfied and assume that Assumptions \ref{Main1}-\ref{Main2} hold. Furthermore, we assume that $a_1< \hat{a}$ and $S^{*}_2<\min\{a_2-a_1,\hat{a}-a_1\}\hbox{, } S^{*}_1<\min\{a_2,\hat{a}\}$. Then, for every $T>A-a_2+a_1+S^{*}_1+S^{*}_2$ and for every $y_0\in L^{2}_{\frac{1}{\sigma}}(Q_{A,S}) $
there exists a control $u\in L^{2}_{\frac{1}{\sigma}}(Q_1)$ such that the solution $y$ of \eqref{2cc} satisfies 
\[ y(x,a,s,T)=0\hbox{ for a.e. }x \in (0,1)\hbox{ } a \in (0,A)\hbox{ and } s \in (0,S).\]
\end{theorem}
The remaining section of this work is organized as follows: In Subsection \ref{Sec2}, we establish the final state observability for the adjoint system and in our proof we combine characteristics method with observability estimates for degenerate parabolic equations (see for instance \cite{canard1,canard}). Thanks to such estimates, we obtain the proof of our main result (Theorem \ref{MainT2})
\subsection{An Observability Inequality}\label{Sec2}
It is well known that the system \eqref{2cc} can equivalently be rewritten as an abstract evolution system:

\begin{align}\label{Abstract}
\begin{cases}
\frac{\partial}{\partial t} y(t)=\mathcal{A}y(t)+\mathcal{B}u(t),& t\geq 0,\cr
y(0)=y_0,
\end{cases}
\end{align}

also to be referred to as $(\mathcal{A},\mathcal{B})$ in the sequel, where we can identify the operators $\mathcal{A}$ and $\mathcal{B}$ through their adjoints by formally taking the inner product of \eqref{Abstract} with a smooth function $\varphi$ see, for instance, \cite[Section 11.2]{b8}. The state and control spaces are

\begin{align*}
	K:=L^{2}_{\frac{1}{\sigma}}(Q_{A,S}),\qquad U:= L^{2}_{\frac{1}{\sigma}}(Q_1).
\end{align*}

The unbounded linear operator $\mathcal{A}:D(\mathcal{A})\subset K\longrightarrow K$, is defined for every $y\in D(\mathcal{A})$ by
$$\mathcal{A}y=L_1 y-\dfrac{\partial y}{\partial a}-\dfrac{\partial g(s) y}{\partial s}-\left(\mu_1(a)+\mu_2(s)\right)y,$$
 with domain
 $$D(\mathcal{A}):= \begin{cases}
 	y\in L^2((0,A)\times (0,S);H^{2}_{\frac{1}{\sigma}}(0,1)):\quad  \mathcal{A}y\in L^{2}_{\frac{1}{\sigma}}(Q_{A,S}),\cr
 	\dfrac{\partial y}{\partial a},\dfrac{\partial y}{\partial s} \in L^2((0,A)\times (0,S);H^{1}_{\frac{1}{\sigma}}(0,1)),\cr
 	y(1,\cdot,\cdot)=y(0,\cdot,\cdot)=0,\cr
 	y(x,a,0)=0,\quad y(x,0,s)=\int_{0}^{A}\beta(a)y(x,a,s,t)da\quad {\rm for \; a.e}\quad x\in (0,1).
 \end{cases}$$
An argument similar to \cite{Genni} yields the following generation result.
\begin{lemma}\label{generation-result}
	Let conditions {\bf (H1)}-{\bf (H3)} be satisfied, and assume that assumption \ref{Main1} holds. Then the operator $(\mathcal{A},D(\mathcal{A}))$ generates a C$_0$-semigroup $\mathfrak{T}:=(e^{t\mathcal{A}})_{t\geq 0}$ on $K$.
\end{lemma}
As such, a similar computation to that in \cite{dy} shows that
\begin{align*}
D(\mathcal{A}^*):= \begin{cases}
	\varphi\in K,\quad \varphi\in L^2((0,A)\times (0,S);H^{2}_{\frac{1}{\sigma}}(0,1)),\cr
	\dfrac{\partial y}{\partial a}+\dfrac{\partial g(s) y}{\partial s}+L_1\varphi-\mu(a,s)\varphi \in K,\cr
	\varphi(1,\cdot,\cdot)=\varphi(0,\cdot,\cdot)=0,\cr
	\varphi(\cdot,A,\cdot)=0,\quad \varphi(\cdot,S,\cdot)=0,
\end{cases}
\end{align*}
and we have, for every $\varphi\in D(\mathcal{A}^*)$,
\begin{align*}
\mathcal{A}^*\varphi=L_1 \varphi+\dfrac{\partial \varphi}{\partial a}+\dfrac{\partial g(s)\varphi}{\partial s}-\mu(a,s)\varphi+\beta(a)\varphi(x,0,s),
\end{align*}
On the other hand, the control operator $\mathcal{B}\in \mathcal{L}(U,K)$ is given for every $u\in U$
$$\mathcal{B}u=mu.$$
With the above notation, one can see that the adjoint problem of \eqref{2cc} is given by:
\begin{equation}
	\left\lbrace\begin{array}{ll}
		\dfrac{\partial q}{\partial t}-\dfrac{\partial q}{\partial a}-\dfrac{\partial g(s)q}{\partial s}-L_1 q+\mu(a,s)q=\beta(a)q(x,0,s,t)&\hbox{ in }Q ,\\ 
		q(0,a,s,t)=q(1,a,s,t)=0&\hbox{ on }\Sigma,\\ 
		q\left( x,A,s,t\right) =0&\hbox{ in } Q_{S,T} \\
		q(x,a,S,t)=0& \hbox{ in } Q_{S,T}\\
		q\left(x,a,s,0\right)=q_0(x,a,s)&
		\hbox{ in }Q_{A,S};
	\end{array}\right.
	\label{112aa}
\end{equation}
where we recall that the operator $(L_1,D(L_1))$ is defined by 
\begin{align*}
L_1y:=k(x)\dfrac{\partial^2 y}{\partial x^2}+b(x)\dfrac{\partial y}{\partial x},\qquad y\in D(L_1):=H^{2}_{\frac{1}{\sigma}}(0,1)\cap H^{1}_{\frac{1}{\sigma}}(0,1).
\end{align*}
Next, consider the operator $P$ defined by
\begin{equation}\label{33}
	P:=-\mu_1(a)-\mu_2(s)+L_1, \qquad {\rm on }\qquad D(\mathcal{A}^*),
\end{equation}
and denotes by $(e^{tP})_{t\geq 0}$ its associated C$_0$-semigroup. 
Since the operator $\mathcal{A}$ generates a C$_0$-semigroup on $K$ (see Lemma \ref{generation-result}) and that $\mathcal{B}$ is a bounded operator. It also follows that the abstract system \eqref{Abstract} is well-posed in the sense that: for every $y_0\in K$ and every $u\in L^2(0,+\infty;U)$, there exists a unique solution $y\in C(0,+\infty;K)$ to \eqref{Abstract} given by the Duhamel formula (see e.g. \cite[Proposition 4.2.5]{b8})
\begin{align}
y(T)=e^{T\mathcal{A}}y_0+\Phi_Tu,\qquad T\geq 0, 
\end{align}
where $\Phi_T$ is the so-called input map of $(\mathcal{A},\mathcal{B})$, that is the linear operator defined for every $u\in L^2(0,+\infty;U)$ by 
\begin{align*}
\Phi_Tu=\int_{0}^{T} e^{(T-s)\mathcal{A}}\mathcal{B}u(s)ds.
\end{align*}
We recall that $(\mathcal{A},\mathcal{B})$ (and hence \eqref{2cc}) is null controllable in time $T>0$ if ${\rm Ran}\, e^{T\mathcal{A}}\subset {\rm Ran}\,\Phi_T$ (see, e.g. \cite[Definition 11.1.1]{b8}). It is also well known that the null controllability of the pair $(\mathcal{A},\mathcal{B})$ is equivalent to the final-state observability of the pair $(\mathcal{A}^*,\mathcal{B}^*)$. More precisely (see e.g. \cite[Theorem 11.2.1]{b8}), $(\mathcal{A},\mathcal{B})$ is null controllable in time $T$ if, and only if, there exists $C_T>0$ such that
\begin{equation}
	\displaystyle\int\limits_{0}^{T}\|\mathcal{B}^*e^{t\mathcal{A}^*}q_{0}\|^2\geq C_{T}^{2}\|e^{T\mathcal{A}^*}q_0\|^2, \qquad \forall\, q_{0}\in D(\mathcal{A}^*),
\end{equation}
where $\mathfrak{T}^{*}:=(e^{t\mathcal{A}^*})_{t\geq 0}$ is the adjoint semigroup of $\mathfrak{T}$ generated by $\mathcal{A}^{*}$. 
\subsubsection{Proof of the main result}
In this part we provide the proof of the main result of this paper, namely Theorem \ref{Main2}. Indeed, in view of \cite[Theorem 11.2.1]{b8} we will mainly perform computations on the adjoint system \eqref{112aa} in the sequel. Then, it is more convenient to restate the result of Theorem \ref{Main2} as follows:

\begin{theorem}\label{Main3}
	Let the assumptions of Theorem \ref{MainTa} be satisfied. Then, for every $q_0\in D(\mathcal{A}^*)$, the pair $(\mathcal{A}^{*},\mathcal{B}^{*})$ is final-state observable in any time $T>A-a_2+a_1+S^{*}_1+S^{*}_2$. In particular, for every $T>A-a_2+a_1+S^{*}_1+S^{*}_2$ there exists $C_T>0$ such that the solution $q$ of \eqref{112aa} satisfies 
	\begin{equation}
		\int\limits_{0}^{S}\int\limits_{0}^{A}\int_{0}^1\dfrac{1}{\sigma(x)}q^2(x,a,s,T)dxdads\leq C_T\displaystyle\int\limits_{0}^{T}\int\limits_{s_1}^{s_2}\int\limits_{a_1}^{a_2}\int_{\omega}\dfrac{1}{\sigma(x)}q^2(x,a,s,t)dxdadsdt.
	\end{equation}
\end{theorem}	 
As in Part 2, we complete the necessary results for the proof of the main result in this part.
\begin{proposition}\label{p2}
Let the assumptions of Theorem \ref{MainTa} be satisfied and let $T>a_1+T_1$. Then for every $q_{0}\in K$, the solution $q$ of the system \eqref{112aa} obeys
	\begin{equation}
\displaystyle\int\limits_{s_1}^{S}\int\limits_{0}^{a_1}\int\limits_{0}^1\dfrac{1}{\sigma(x)}q^2(x,a,s,T)dxdads\leq C_T\displaystyle\int\limits_{0}^{T}\int\limits_{s_1}^{s_2}\int\limits_{a_1}^{a_2}\int\limits_{\l_1}^{l_2}\dfrac{1}{\sigma(x)}q^2(x,a,s,t)dxdadsdt \label{C81}
	\end{equation}
  Moreover, there exists $a^{*}_0\in (0,a_2-S^{*}_1)$ such that for every $q_0\in L^2(\Omega\times (0,A)\times (0,S)),$ the solution $q$ of the system $(\ref{112aa}),$ verifies
  \begin{equation}
\displaystyle\int\limits_{0}^{s_1}\int\limits_{0}^{a^{*}_0}\int\limits_{0}^1\dfrac{1}{\sigma(x)}q^2(x,a,s,T)dxdads\leq C_T\displaystyle\int\limits_{0}^{T}\int\limits_{s_1}^{s_2}\int\limits_{a_1}^{a_2}\int\limits_{\l_1}^{l_2}\dfrac{1}{\sigma(x)}q^2(x,a,s,t)dxdadsdt \label{C81}
	\end{equation}
\end{proposition}

Now, let us consider the following system
\begin{equation}\label{322}
	\left\lbrace
\begin{array}{ll}
		\dfrac{\partial w}{\partial t}-\dfrac{\partial w}{\partial a}-\dfrac{\partial g(s) w}{\partial s}-L_1 w+\mu(a,s)w=0, &\hbox{ in }Q ,\\ 
		w(0,a,s,t)=w(1,a,s,t)=0, &\hbox{ on }\Sigma,\\ 
		w\left( x,A,s,t\right) =0, &\hbox{ in } Q_{S,T}, \\
		w(x,a,S,t)=0, &\hbox{ in } Q_{S,T},\\
	w\left(x,a,s,0\right)=w_0,&\hbox{ in }Q_{A,S}.
	\end{array}\right.
\end{equation} 
We also need the following result for the proof of Theorem \ref{Main3}.
\begin{proposition}\label{1p}
	Let the assumption of Theorem \ref{MainTa} be satisfied. Let $s_0\in(s_1,s_2)$, $T>S^{*}_{2}$ and $a_1<a_0<a_2-S^{*}_{2}$ Then, there exists a constant $C_T>0$ such that the solution $w$ of the system \eqref{322} satisfies the following inequality
	\begin{equation}\label{C8}
		\displaystyle\int\limits_{s_1}^{S}\int\limits_{a_1}^{a_0}\int\limits_{0}^1\dfrac{1}{\sigma(x)}w^2(x,a,s,T)dxdads\leq C_T\displaystyle\int\limits_{0}^{T}\int\limits_{s_1}^{s_2}\int\limits_{a_1}^{a_2}\int_{\omega}\dfrac{1}{\sigma(x)}w^2(x,a,s,t)dxdadsdt.
	\end{equation}
	Moreover,
	\begin{align}\label{C9}
		w(x,a,s,T)=0,\quad\hbox{ a.e. in }(0,1)\times (a_1,A)\times (s_{2},S).
	\end{align}
\end{proposition}

\begin{proof}[Proof of Proposition \ref{1p}]
	
The proof is provided in two parts:

\underline{\textbf{1)- Proof of the inequality $\eqref{C8}$:}}
\underline{\textbf{2)- Proof of the equality \eqref{C9}:}}
To this end, let $v$ a function so that 
\begin{equation}\label{xx1}
	\begin{cases}
		\dfrac{\partial v}{\partial t}-\dfrac{\partial v}{\partial a}-\dfrac{\partial g(s)v}{\partial s}-L_1 u+\mu(a,s)v=0,&\hbox{ in }(0,1)\times (0,A)\times(0,S)\times (\eta,T) ,\\ 
		v(0,a,s,t)=v(1,a,s,t)=0,&\hbox{ on }(0,A)\times(0,S)\times (\eta,T),\\ 
		v\left( x,A,s,t\right) =0,&\hbox{ in } (0,1)\times(0,S)\times (\eta,T) \\
		v(x,a,S,t)=0,&\hbox{ in }(0,1)\times (0,A)\times (\eta,T)\\
		v(x,a,s,\eta)=w_{\eta},&\hbox{ in }(0,1)\times (0,A)\times(0,S),
\end{cases}
\end{equation} 
where $w_{\eta}=w(x,a,s,\eta)$ in $(0,1)\times(0,A)\times(0,S)$. It is quite obvious that the solution $w$ of \eqref{322} satisfies \[w(x,a,s,t)=v(x,a,s,t),\quad\hbox{ a.e. in }(0,1)\times(0,A)\times(0,S)\times(0,T).\]
Since $T>\int_{s_2}^{S}\frac{ds}{g(s)}=S^*_2$.
Then, \[T>\int_{s}^{S}\frac{ds}{g(s)}, \quad \hbox{ for all } \quad s\in (s_2,S),\] thus, using the method of characteristics \[w(x,a,s,T)=0,\quad\hbox{ a.e. in }\quad(0,1)\times (0,A)\times (s_2,S).\] 
This complete the proof of Proposition \ref{1p}.
\end{proof}

After such a long but necessary preparation, we can now clearly prove Theorem \ref{Main3}. To this end, let $q=v_1+v_2$ with $v_1$ and $v_2$ are functions satisfying, respectively,
\begin{equation}\label{3220}
	\begin{cases}
		\dfrac{\partial v_1}{\partial t}-\dfrac{\partial v_1}{\partial a}-\dfrac{\partial g(s) v_1}{\partial s}-L_1 v_1+\mu(a,s)v_1=0,& \hbox{ in }\Omega\times(0,A)\times (s_1,S) \times(\eta,T) ,\\ 
		v_1(0,a,s,t)=v_1(1,a,s,t)=0,&\hbox{ on }(0,A)\times(s_1,S)\times (\eta,T),,\\ 
		v_1\left( x,A,s,t\right) =0,&\hbox{ in } (0,1)\times(s_1,S)\times (\eta,T) \\
		v_1(x,a,S,t)=0,&\hbox{ in }\Omega\times(0,A)\times (\eta,T)\\
		v_1\left(x,a,s,\eta\right)=q_{\eta},&\hbox{ in }\Omega\times(0,A)\times (s_1,S) ,
	\end{cases}
\end{equation} 
where $q_{\eta}=q(x,a,s,\eta)$ and
\begin{equation}\label{3221}
	\begin{cases}
		\dfrac{\partial v_2}{\partial t}-\dfrac{\partial v_2}{\partial a}-\dfrac{\partial g(s)v_2}{\partial s}-L v_2+\mu(a,s)v_2=V, &\hbox{ in }\Omega\times(0,A)\times (s_1,S)\times(\eta,T),\\ 
		v_2(0,a,s,t)=v_2(1,a,s,t)=0, &\hbox{ on } (0,A)\times(s_1,S)\times (\eta,T),\\
		v_2\left( x,A,s,t\right) =0, &\hbox{ in } (0,1)\times(s_1,S)\times (\eta,T) \\
		v_2(x,a,S,t)=0,&\hbox{ in } \Omega\times(0,A)\times (\eta,T)\\
		v_2\left(x,a,s,\eta\right)=0, &\hbox{ in }\Omega\times(0,A)\times (s_1,S) .
	\end{cases}
\end{equation}
where $V(x,a,s,t)=\beta(a)q(x,0,s,t)$. According to Duhamel’s formula we clearly have \[v_2(x,a,s,t)=\int_{\eta}^{t}e^{(t-r)\mathcal{A}^{*}}V(\cdot,\cdot,\cdot,r)dr,\]
where we recall that $\mathfrak{T}^{*}:=(e^{t\mathcal{A}^*})_{t\geq 0}$ is the adjoint semigroup of $\mathfrak{T}$ generated by $\mathcal{A}^{*}$. Moreover, the solution $v_2$ of the system \eqref{3221} satisfies the following estimate.
\begin{proposition}\label{Propo1}
	Let the assumption of Theorem \ref{MainTa} be satisfied. Then, the solution $v_2$ of the system \eqref{3221} satisfies
	\begin{align}
		\int\limits_{\eta}^{T}\int\limits_{s_1}^{s_2}\int\limits_{a_1}^{a_2}\int_{\omega}\dfrac{1}{\sigma(x)}v_{2}^2(x,a,s,t)dxdadsdt&\leq \int\limits_{\eta}^{T}\int\limits_{s_1}^{S}\int\limits_{0}^{A}\int\limits_{0}^{1}\dfrac{1}{\sigma(x)}v_{2}^2(x,a,s,t)dxdadsdt\nonumber\\
		&\leq C \int\limits_{\eta}^{T}\int\limits_{s_1}^{S}\int\limits_{0}^{1}\dfrac{1}{\sigma(x)}q^2(x,0,s,t)dxdsdt,
	\end{align}
where $C=e^{\frac{3}{2}T}\|\beta\|^{2}_{\infty}A$. 
\end{proposition}
\subsection{Proof of the Theorem \ref{MainTa}}
\begin{proof}[Proof of the Theorem \ref{MainTa}] We split the proof into two parts. 
	
	
    We split the term to be estimated as follows:
	\[\int\limits_{0}^{S}\int\limits_{0}^{A}\int\limits_{0}^{1}\dfrac{1}{\sigma(x)}q^2(x,a,s,T)dxdads=\int\limits_{0}^{S}\int\limits_{0}^{a_1}\int\limits_{0}^{1}\dfrac{1}{\sigma(x)}q^2(x,a,s,T)dxdads+\int\limits_{0}^{S}\int\limits_{a_1}^{A}\int\limits_{0}^{1}\dfrac{1}{\sigma(x)}q^2(x,a,s,T)dxdads.\]
 \underline{	\textbf{Upper bound on $(s_1,S)$:}}
According to Proposition \ref{p2}, we have 
	\begin{equation}\label{coupp}
		\int\limits_{s_1}^{S}\int\limits_{0}^{a_1}\int\limits_{0}^{1}\dfrac{1}{\sigma(x)}q^2(x,a,s,T)dxdads\leq C_T\int\limits_{0}^{T}\int\limits_{s_1}^{s_2}\int\limits_{a_1}^{a_2}\int\limits_{l_1}^{l_2}\dfrac{1}{\sigma(x)}q^2(x,a,s,t)dxdadsdt.
	\end{equation}
	So it remains to estimate the following term \[\int\limits_{s_1}^{S}\int\limits_{a_1}^{A}\int\limits_{0}^{1}\dfrac{1}{\sigma(x)}q^2(x,a,s,T)dxdads=\int\limits_{s_1}^{S}\int\limits_{0}^{a_1}\int\limits_{0}^{1}\dfrac{1}{\sigma(x)}q^2(x,a,s,T)dxdads+\int\limits_{s_1}^{S}\int\limits_{a_1}^{A}\int\limits_{0}^{1}\dfrac{1}{\sigma(x)}q^2(x,a,s,T)dxdads.\]
Since $q=v_1+v_2$, we need to estimate 
	\[\int\limits_{s_1}^{S}\int\limits_{a_1}^{A}\int\limits_{0}^{1}\dfrac{1}{\sigma(x)}v_{1}^{2}(x,a,s,T)dxdads+\int\limits_{s_1}^{S}\int\limits_{a_1}^{A}\int\limits_{0}^{1}\dfrac{1}{\sigma(x)}v_{2}^{2}(x,a,s,T)dxdads;\]
	In view of Proposition \ref{Propo1}, we have    \[\int\limits_{s_1}^{S}\int\limits_{a_1}^{A}\int\limits_{0}^{1}\dfrac{1}{\sigma(x)}v_{2}^{2}(x,a,s,T)dxdads\leq C(A-a_2)\int\limits_{\eta}^{T}\int\limits_{s_1}^{S}\int\limits_{0}^{1}\dfrac{1}{\sigma(x)}q^2(x,0,s,t)dxdsdt.\]
	Proposition \ref{Prop1} further yields that
	\begin{equation}\label{debay}
		\int\limits_{s_1}^{S}\int\limits_{a_1}^{A}\int\limits_{0}^{1}\dfrac{1}{\sigma(x)}v_{2}^{2}(x,a,s,T)dxdads\leq C_{\eta,T}\int\limits_{0}^{T}\int\limits_{s_1}^{s_2}\int\limits_{a_1}^{a_2}\int_{\omega}\dfrac{1}{\sigma(x)}q^2(x,a,s,t)dxdadsdt.
	\end{equation}
	Since $T>A-a_2+a_1+S^{*}_{1}+S^{*}_{2}$, then there exists $\delta>0$ such that 
	$$T>A-a_2+a_1+S^{*}_{1}+S^{*}_{2}+2\delta.$$ 
	Therefore,
	 $$T-\left(a_1+S^{*}_{2}+\delta\right)>A-\left(a_2-\left(S^{*}_{1}+\delta\right)\right).$$
  By noting that for \( s \) in \([s_1, S]\), \(\eta\) can be taken in the form \( a_1 + S_{2}^* + \delta \).
	For every $a\in\left(a_2,A\right) $, we have\[T-\eta>A-\left(a_2-\left(S^{*}_{1}+\delta\right)\right)>A-a,\] where $\eta=a_1+S^{*}_{2}+\delta.$
	Then \[v_1(x,a,s,T)=0\hbox{ a.e. in }\quad(0,1)\times \left(a_2,A\right)\times (s_1,s_2).\]
	It follows that
	\[\int\limits_{s_1}^{s_2}\int\limits_{a_1}^{A}\int\limits_{0}^{1}\dfrac{1}{\sigma(x)}v_{1}^2(x,a,s,T)dxdads\leq C_{\eta,T}\int\limits_{\eta}^{T}\int\limits_{s_1}^{s_2}\int\limits_{a_1}^{a_2}\int_{\omega}\dfrac{1}{\sigma(x)}v_{1}^{2}(x,a,s,t)dxdadsdt.\]
	Moreover, as $T-S_{1}^*>S_{2}^*,$ then according to Proposition \ref{1p}, \[v_1(x,a,s,T)=0\hbox{ a.e. in }\quad(0,1)\times \left(a_2,A\right)\times (s_2,S).\]
	\begin{equation}
		\int\limits_{s_1}^{S}\int\limits_{a_1}^{A}\int\limits_{0}^{1}\dfrac{1}{\sigma(x)}v_{1}^{2}(x,a,s,T)dxdads\leq C\int\limits_{\eta}^{T}\int\limits_{s_1}^{s_2}\int\limits_{a_1}^{a_2}\int_{\omega}\dfrac{1}{\sigma(x)}v_{1}^{2}(x,a,s,t)dxdadsdt.\label{deba}
	\end{equation}
	The fact that $q=v_1+v_2\Longleftrightarrow v_1=q-v_2$ yields  
	\begin{align}\label{deba1}
		\int\limits_{s_1}^{S}\int\limits_{a_1}^{A}\int\limits_0^{1}\dfrac{1}{\sigma(x)}v_{1}^{2}(x,a,s,T)dxdads\leq 2\left(\int\limits_{\eta}^{T}\int\limits_{s_1}^{s_2}\int\limits_{a_1}^{a_2}\int\limits_{l_1}^{l_2}\dfrac{1}{\sigma(x)}q^{2}(x,a,s,t)dQ+\int\limits_{\eta}^{T}\int\limits_{s_1}^{s_2}\int\limits_{a_1}^{a_2}\int\limits_{l_1}^{l_2}\dfrac{1}{\sigma(x)}v_{2}^{2}(x,a,s,t)dQ\right),
	\end{align}
where we have set $dQ=dxdadsdt$. 
	Thus, by Proposition \ref{1p}, we get the following. 
	\begin{equation}\label{deba11}
		\int\limits_{s_1}^{S}\int\limits_{a_1}^{A}\int\limits_{0}^{1}\dfrac{1}{\sigma(x)}v_{1}^{2}(x,a,s,T)dxdads\leq C\int\limits_{0}^{T}\int\limits_{s_1}^{s_2}\int\limits_{a_1}^{a_2}\int_{\omega}\dfrac{1}{\sigma(x)}q^{2}(x,a,s,t)dxdadsdt.
	\end{equation}
for a constant $C_{\eta,T}>0$. Inequalities $(\ref{debay})$ and $(\ref{deba11})$ yield the following. 
	\begin{equation}
		\int\limits_{s_1}^{S}\int\limits_{a_1}^{A}\int\limits_{0}^{1}\dfrac{1}{\sigma(x)}q^{2}(x,a,s,T)dxdads\leq C_{\eta,T} \int\limits_{0}^{T}\int\limits_{s_1}^{s_2}\int\limits_{a_1}^{a_2}\int_{\omega}\dfrac{1}{\sigma(x)}q^{2}(x,a,s,t)dxdadsdt.\label{coup1}
	\end{equation}
for a constant $C_{T}>0$. Finally, $(\ref{coupp})$ and $(\ref{coup1})$ yield
	\begin{equation}
		\int\limits_{s_1}^{S}\int\limits_{0}^{A}\int\limits_{0}^{1}\dfrac{1}{\sigma(x)}q^{2}(x,a,s,T)dxdads\leq C_T \int\limits_{0}^{T}\int\limits_{s_1}^{s_2}\int\limits_{a_1}^{a_2}\int_{\omega}\dfrac{1}{\sigma(x)}q^{2}(x,a,s,t)dxdadsdt.\label{coup}
	\end{equation}
\underline{	\textbf{Upper bound on $(0,s_1)$: }}
	We split the term to be estimated as follows
	\[\int\limits_{0}^{s_1}\int\limits_{0}^{A}\int\limits_0^1\dfrac{1}{\sigma(x)}q^2(x,a,s,T)dxdads=\int\limits_{0}^{s_1}\int\limits_{0}^{a^{*}_0}\int\limits_0^1\dfrac{1}{\sigma(x)}q^2(x,a,s,T)dxdads+\int\limits_{0}^{s_1}\int\limits_{a^*_0}^{A}\int\limits_0^1\dfrac{1}{\sigma(x)}q^2(x,a,s,T)dxdads,\] with $a^*_0$ defined in Proposition \ref{p2}.\smallbreak
	According to Proposition \ref{p2}, we have
	\begin{equation}\int\limits_{0}^{s_1}\int\limits_{0}^{a_0^*}\int\limits_0^1\dfrac{1}{\sigma(x)}q^2(x,a,s,T)dxdads\leq C\int\limits_{0}^{T}\int\limits_{s_1}^{s_2}\int\limits_{a_1}^{a_2}\int_{\omega}\dfrac{1}{\sigma(x)}q^2(x,a,s,t)dxdadsdt.\label{coupp1}\end{equation}
	Since $q=v_1+v_2$ one need to estimate  
	\[\int\limits_{0}^{s_1}\int\limits_{a_0^*}^{A}\int_{\Omega}\dfrac{1}{\sigma(x)}v_{1}^{2}(x,a,s,T)dxdads+\int\limits_{0}^{s_1}\int\limits_{a^*_0}^{A}\int_{\Omega}\dfrac{1}{\sigma(x)}v_{2}^{2}(x,a,s,T)dxdads.\]
 From Proposition \ref{Propo1} with the variable $s\in (0,S)$, we obtain
 \[\int\limits_{0}^{s_1}\int\limits_{a^*_0}^{A}\int\limits_{0}^{1}\dfrac{1}{\sigma(x)}v_{2}^{2}(x,a,s,T)dxdads\leq C_{\eta,A}\int\limits_{\eta}^{T}\int\limits_{0}^{S}\int\limits_{0}^{1}\dfrac{1}{\sigma(x)}q^2(x,0,s,t)dxdsdt.\]
 Proposition \ref{Prop1} further yields that
	\begin{equation}
\int\limits_{0}^{s_1}\int\limits_{a^*_0}^{A}\int\limits_{0}^{1}\dfrac{1}{\sigma(x)}v_{2}^{2}(x,a,s,T)dxdads\leq C_{\eta,T}\int\limits_{0}^{T}\int\limits_{s_1}^{s_2}\int\limits_{a_1}^{a_2}\int_{\omega}\dfrac{1}{\sigma(x)}q^2(x,a,s,t)dxdadsdt.
  \label{zum1}
	\end{equation}
 Given the time $T>A-a_2+a_1+S^*_1+S^*_2$ and that $T_1=\max\{a_1+S^*_2,S^*_1\}$, two situations can arise for the estimation of \[\int\limits_{0}^{s_1}\int\limits_{a^*_0}^{A}\int\limits_{0}^{1}\dfrac{1}{\sigma(x)}v_{1}^{2}(x,a,s,T)dxdads.\]
 \underline{\textbf{Part 1: the case $S^{*}_{1}<a_1+S^{*}_{2}=T_1$.}}
In this case, the estimation is simpler because \[T-\eta>A-a_0^*>A-a\hbox{ for all }a\in (a_0^*,A) ;\] that imply \begin{equation} 
v_1(x,a,s,T)=0\hbox{ a.e. in }\Omega \times (a_0^*,A)\times(0,s_1).\label{zum}
\end{equation}.
The inequalities \eqref{zum1} and \eqref{zum} give \begin{equation}\int\limits_{0}^{s_1}\int\limits_{0}^{a_0^*}\int\limits_0^1\dfrac{1}{\sigma(x)}q^2(x,a,s,T)dxdads\leq C\int\limits_{0}^{T}\int\limits_{s_1}^{s_2}\int\limits_{a_1}^{a_2}\int_{\omega}\dfrac{1}{\sigma(x)}q^2(x,a,s,t)dxdadsdt.\label{cap}
\end{equation}
Finally, combining \eqref{coupp1} and  \eqref{cap} we obtain 
\begin{equation}\int\limits_{0}^{s_1}\int\limits_{0}^{A}\int\limits_0^1\dfrac{1}{\sigma(x)}q^2(x,a,s,T)dxdads\leq C\int\limits_{0}^{T}\int\limits_{s_1}^{s_2}\int\limits_{a_1}^{a_2}\int_{\omega}\dfrac{1}{\sigma(x)}q^2(x,a,s,t)dxdadsdt.\label{alpha}\end{equation}
\underline{\textbf{Part 2: the case $S^{*}_{1}>a_1+S^{*}_{2}$}}.
Then there exists $\beta^*< s_1$ such that $G(s_1)-G(\beta^*)=a_1.$ Therefore, we split $(0, s_1)$ as follows $ (0.s_1)= (0, \beta^*) \cup (\beta^*, s_1)$.\smallbreak 
We notice that, if \[(a,s)\in (a_0^*,A)\times (\beta^*,s_1)\hbox{ then }S^*_1-G(s)\leq a_1+S_2^*,\] and by proceeding as above, we arrive at the desired estimate.\smallbreak
Now, we will estimate in this part \[q(.,a,s,T)\hbox{ for } (a,s)\in \left(a_0^*,A\right)\times (0,\beta^*)\] in the case where $T>A-a_2+S^{*}_{2}+a_1+S^{*}_{1}.$\smallbreak For it, we split this section $(a_0^*,A)\times (0,\beta^*)$ as follows: \begin{align*}S_{T}^{1}=\{(a,s,T)\hbox{ such that }0<s<\beta^* \hbox{ and } G(s)+A-G(\beta^*)<a<A\}\end{align*} and
\begin{align*}S_{T}^{2}=\{(a,s,T)\hbox{ such that }0<s<\beta^* \hbox{ and } a^{*}_{0}<a<A-G(\beta^*)+G(s)\}.\end{align*}
All backward characteristics starting at $(a,s,T)$, where $(a,s,T) \in S_T^1$, are replaced by the renewal term with $a$ in $(A-\delta,A)$ and require a minimum of $S_1^* - \delta$ time to enter the observation domain.\smallbreak
In $S^{1}_{T}$ we have
\[q(x,a,s,T)=\int\limits_{T-A+a}^{T}\beta(a+T-l)q(x,0,G^{-1}(G(s)+T-l),l)dl.\]
Then
\[\int_{\Omega}\dfrac{1}{\sigma}q^2(x,a,s,T)\leq A\|\beta\|^{2}_{\infty}\int\limits_{T-A+a}^{T}\int_{\Omega}\dfrac{1}{\sigma}q^2(x,0,G^{-1}(G(s)+T-l),l)dl\]
Let $\delta>0$ be such that $T>S^{*}_{1}+2\delta.$ Subdividing the interval $(0,\beta^*)$ in $M$ intervals as follows:
\[0=\delta_0<\delta_1<\delta_2<\delta_3<...<\delta_M\hbox{ such that }G(\delta_{i+1})-G(\delta_{i})<\delta\hbox{, } i\in\{0,1,...,M-1\},\] with
\[\delta_0=0\hbox{ and }\delta_M=\beta^*\]
We denote by \[S_{T,i}^{1}=\{(a,s,T)\hbox{ such that }\delta_i<s<\delta_{i+1} \hbox{ and } G(s)+A-G(\delta_{i+1})<a<A-G(\delta_i)\}.\]
Moreover, in $S_{T,i},$ \[\max_{S_{T,i}^{1}}\{T-G(S)+G(s),T-A+a\}>T-G(\delta_{i+1})+G(s).\]
Therefore, if $(a,s,T)\in S^{1}_{T,i}\hbox{ and }T>S^{*}_{1}+2\delta,$ then \[\max_{S_{T,i}^{1}}\{T-G(S)+G(s),T-A+a\}>T-G(\delta_{i+1})+G(\delta_i)>S^*_1+\delta.\]  
Moreover, choosing $\eta=S^*_1+\delta$, then $S^{*}_{1}<\eta<T,$ and \[\int\limits_{\eta}^{T}\int\limits_{\delta_i}^{\delta_{i+1}}\int_{\Omega}\dfrac{1}{\sigma}q^2(x,0,s,t)dxdsdt\leq C_{T}\int\limits_{0}^{T}\int\limits_{s_1}^{s_2}\int\limits_{a_1}^{a_2}\int_{\omega}\dfrac{1}{\sigma}q^2(x,a,s,t)dxdadsdt.\] Therefore,
\[\int\limits_{0}^{T}\int\limits_{\delta_i}^{\delta_{i+1}}\dfrac{1}{\sigma}q^2(x,a,s,T)dxads\leq C_T\int\limits_{0}^{T}\int\limits_{s_1}^{s_2}\int\limits_{a_1}^{a_2}\int_{\omega}\dfrac{1}{\sigma}q^2(x,a,s,t)dxdadsdt.\] Then for $T>S^{*}_{1},$ we have
\[\int_{S_{T}^{1}}\dfrac{1}{\sigma}q^2(x,a,s,T)dxdadt\leq C_T\int\limits_{0}^{T}\int\limits_{s_1}^{s_2}\int\limits_{a_1}^{a_2}\int_{\omega}\dfrac{1}{\sigma}q^2(x,a,s,t)dxdadsdt.\]
The same concept will allow us to obtain an estimate of $S_T^2$. By combining the estimates of $S_T=S_T^1\cup S_T^2$ and $(a^*_0,A)\times (\beta^*,s_1),$ we can derive the desired estimate.
\end{proof}
  \begin{figure}[H]
    \begin{adjustbox}{center,left}
        \begin{overpic}[scale=0.25]{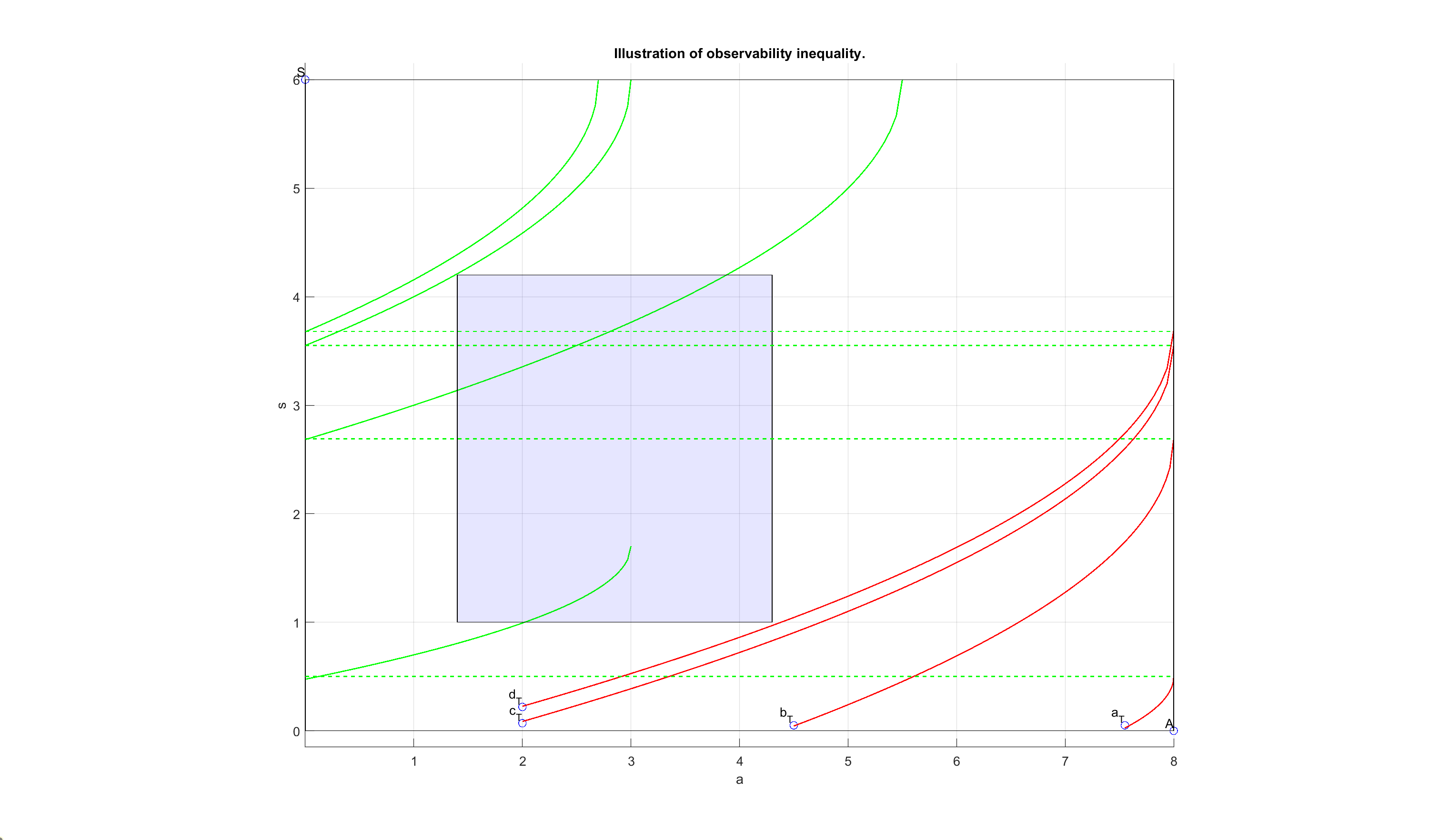}
            
        \end{overpic}
        \end{adjustbox}
    \caption{Here $\hat{a}=a_2$. The backward characteristics starting from $(a,s,T)$ with $a\in (a_0,A)$ (red line) hits the boundary $(a=A)$, gets renewed by the renewal condition $g(s)\beta(a)q(x,0,s,t)$ (who is local) and then enters the observation domain (red line) or without and leave the domain $\Sigma$ by the boundary $s=S$ (green line).	So, with the conditions $T>A-a_2+S^*_2+a_1+S^*_1$ all the characteristics starting at $(a,s,T)$ which hit the section $a=A$ without having encountered the control domain renewed by the renewal condition $g(s)\beta(a)q(x,0,s,t)$ with $t>a_1+S^*_2$ if $s>\beta^*$ and $t>G(s_1)-G(\beta^*)$ if $s<\beta^*$ and enter the observation domain or without all the domain by the boundary $s=S.$ }
\end{figure}
\section{appendix}
\begin{proof}{Proof of Proposition \ref{a7} } The proof of existence is given by Theorem \ref{a8}\smallbreak
	We recall by \[G(s)=\int_{0}^{s}\frac{du}{g(u)}.\]  Let
	\[w(x,\lambda)=q(x,a+t-\lambda,G^{-1}(G(s)+t-\lambda),\lambda)\] and 
	then the function $w$ verifies the following system
	\begin{align}
		\left\lbrace\begin{array}{l}
			w'(x,\lambda)=(-\mu_1(a+t-\lambda-\mu_2(G^{-1}(G(s)+t-\lambda))+L) w(\lambda,x)+ f(x,\lambda)\smallbreak
			\frac{\partial w}{\partial \nu}=0\smallbreak
			w (0,x) = q (x, a + t, G^{-1}(G(s)+t), 0),\end{array}\right.
	\end{align}
	with \[f(\lambda,x)=\int\limits_{0}^{S}\beta(a+t-\lambda,G^{-1}(G(s)+t-\lambda),\hat{s})q(x,0,\hat{s},\lambda)d\hat{s}.\]\smallbreak
	The solution of the homogeneous equation is given by
	\[w_H(\lambda,x)=Ce^{\lambda R}.\]
		We notice that $$q(x,a,s,t)=w(x,t).$$\smallbreak
		\textbf{For the taking into account the initial condition,} we consider the domain $A_1.$ Using the Duhamel formula, we obtain\[q(x,a,s,t)= w(x,t)=e^{tL}w(x,0)+\]\[\int\limits_{0}^{t}\left(e^{(t-\alpha)R}\int\limits_{0}^{S}\beta(a+t-\alpha,G^{-1}(G(s)+t-\alpha),\hat{s})q(x,0,\hat{s},\alpha)d\hat{s}\right)d\alpha.\]
As
	\[w(0,x)=q(x,a+t,G^{-1}(G(s)+t),0)=q_0(x,a+t,G^{-1}(G(s)+t)),\]
then
	\[q(x,a,s,t)=q_{0}(x,a+t,G^{-1}(G(s)+t))e^{t R}+\int\limits_{0}^{t}\left(e^{(t-\alpha)R}\int\limits_{0}^{S}\beta(a+t-\alpha,G^{-1}(G(s)+t-\lambda),\hat{s})q(x,0,\hat{s},\alpha)d\hat{s}\right)d\alpha\]	
	in $A_1.$\smallbreak
	\textbf{Taking into account the boundary condition in $a.$}\smallbreak
	 For the boundary condition in $\{a=A\}$ we use the set
	\[A'_1=\left\{(a,s,t)\in Q' \text{ such that } G(S)-G(s)>t>A-a> 0\text{ or } t>G(S)-G(s)>A-a> 0\right\},\]
	   then using the Duhamel formula (boundary condition in age is $q(x,A,s,t)$) we obtain 
	\begin{equation}
q(x,a,s,t)=e^{(A-a)R}w(t-(A-a),x)+
\int\limits_{t-A+a}^{t}\left(e^{(t-l)R}\int\limits_{0}^{S}\beta(a+t-l,G^{-1}(G(s)+t-\lambda),\hat{s})q(x,0,\hat{s},l)d\hat{s})\right)dl. \label{souff}
\end{equation}
	But $$w(t-(A-a),x)=q(x,A,,a+t-A)=0,$$ then 
	\[q(x,a,s,t)=\int\limits_{t-A+a}^{t}\left(e^{(t-l)R}\int\limits_{0}^{S}\beta(a+t-l,G^{-1}(G(s)+t-\lambda),\hat{s})q(x,0,\hat{s},l)d\hat{s})\right)dl .\]	
	\textbf{Taking into account of the boundary condition in $s.$}\smallbreak
	For the boundary condition in $\{s=S\}$ we use the set \[A'_2=\left\{(a,s,t)\in Q' \text{ such that } A-a>t>G(S)-G(s)> 0\text{ or } t>A-a>G(S)-G(s)> 0\right\},\] then using the Duhamel formula, we obtain:
	\[q(x,a,s,t)=e^{(S-s)R}w(t-G(S)+G(s),x)+\int\limits_{t-G(S)+G(s)}^{t}\left(e^{(t-l)R}\int\limits_{0}^{S}\beta(a+t-l,G^{-1}(G(s)+t-\lambda),\hat{s})q(x,0,\hat{s},l)d\hat{s})\right)dl .\]
	As before   $$w\left(t-G(S)+G(s),x\right)=q\left(x,a+G(S)-G(s),G^{-1}\left(G(s)+G(S)-G(s)\right),t-G(S)+G(s)\right),$$
 \[=q\left(x,a+G(S)-G(s),G^{-1}\left(G(S)\right),t-G(S)+G(s)\right)=0\]
 then
	\begin{equation}
	q(x,a,s,t)=\int\limits_{t-G(S)+G(s)}^{t}\left(e^{(t-l)R}\int\limits_{0}^{S}\beta(a+t-l,(G^{-1}(G(s)+t-\lambda),\hat{s})q(x,0,\hat{s},l)d\hat{s})\right)dl \hbox{ in }A'_2. \label{souf}
	\end{equation} 
	In the rest of the paper, we will adopt the following representation of the solution.
	\begin{align}
		q(t) = 
		\left\lbrace\begin{array}{l}
			 q_0(.,a+t,G^{-1}(G(s)+t))e^{t R}+\displaystyle\int\limits_{0}^{t}\left(e^{(t-l)R}\int\limits_{0}^{S}\beta(a+t-l,G^{-1}(G(s)+t-\lambda),\hat{s})q(x,0,\hat{s},l)d\hat{s})\right)dl \text{ in } A_1, \\
			\displaystyle\int\limits_{\max\left\{t-A+a,  t-G(S)+G(s)\right\}}^{t}\left(e^{(t-l)R}\int\limits_{0}^{S}\beta(a+t-l,G^{-1}(G(s)+t-\lambda),\hat{s})q(x,0,\hat{s},l)d\hat{s})\right)dl \text{ in } A_2.\label{alp}
		\end{array}\right.
		\end{align}
Indeed, we have:
	\[(a,s,t)\in A'_1\Leftrightarrow 0<t-G(S)+G(s)< t+a-A<t\text{ or } t-G(S)+G(s)<0<t+a-A,\] and
	\[(a,s,t)\in A'_2\Leftrightarrow 0<t+a-A<t-G(S)+G(s)<t  \text{ or } t+a-A<0<t-G(S)+G(s),\] So we notice that
\[\max\left\{t-A+a,  t-G(S)+G(s)\right\}=t-A+a \quad in\quad A^{'}_{1}.\] and\[\max\left\{t-A+a,  t-G(S)+G(s)\right\}=t-G(S)+G(s) \quad in\quad A^{'}_{2}.\] Then we obtain:  
	\begin{equation}
		q(x,a,s,t)=\displaystyle\int\limits_{\max\left\{t-A+a,  t-G(S)+G(s)\right\}}^{t}\left(e^{(t-l)R}\int\limits_{0}^{S}\beta(a+t-l,(G^{-1}(G(s)+t-\lambda),\hat{s})q(x,0,\hat{s},l)d\hat{s})\right)dl \text{ in } A_2
	\end{equation}
 \end{proof}
We also required this result, which allowed us in certain situations to significantly simplify the calculations by referring to existing results.
Therefore, in this subsection, we present the following statement regarding null controllability of a linearized Crocco-type model:
\subsection{Null Controllability of an auxiliary system}.
In this part, we assume that the assumptions $(H_3)$ hold. Therefore, we consider the following auxiliary system:
\begin{equation}
\left\{ 
\begin{array}{ccc}
\dfrac{\partial p}{\partial t}+\dfrac{\partial g(s)p}{\partial s} -L p&=\chi_{\Theta}U &\text{ in }Q_{S,T},\\  
\dfrac{\partial p}{\partial \nu}&=0&\text{ on }\bar{\Sigma},\\
p\left(x,0,t\right) &= 0 &\text{ in } \Omega\times (0,T), \\
p\left(x,s,0\right)&=p_{0}&
\text{ in }\Omega\times (0,S),
\end{array}
\right.
\label{3a}
\end{equation}
We define  his corresponding adjoint system given by:
\begin{equation}
\left\{ 
\begin{array}{ccc}
\dfrac{\partial h}{\partial t}-g(s)\dfrac{\partial h}{\partial s}-L h&=0 &\text{ in }Q_{S,T},\\  
\dfrac{\partial h}{\partial \nu}&=0&\text{ on } \bar{\Sigma},\\
h\left(x,S,t\right) &=0&\text{ in }  \Omega\times (0,T), \\
h\left(x,s,T\right)&=h_{T}&
\text{ in } \Omega\times (0,S).
\end{array}
\right.
\label{4a}
\end{equation}
where $h_T\in L^2(\Omega\times (0,S))$ and $\bar{\Sigma}=\partial\Omega\times (0,S)\times (0,T).$\smallbreak
Let $s_1\hbox{ and }s_2$ such that $0\leq s_1<s_2\leq S,$ we have follows
\begin{theorem}\label{resum}
 Assume that the assumptions $(H_3)$ hold.\smallbreak
Then for every $T>T_0$ and for every $p_0\in L^2(\Omega\times (0,S)\times(0,T)) $
 there exists a control $U\in L^2(\omega\times(s_1,s_2))$ such that the solution $p$ of $(\ref{3a})$ satisfies \[ p(x,s,T)=0\hbox{ a.e }x \in \Omega\hbox{ and } s \in (0,S).\]	
By equivalence, the solution $h$ of the system $(\ref{4a})$ is final-state observable for every \[
T > T_0.
\] In other words, for every \[
T > T_0,
\] and for every $h_0\in L^2(\Omega\times(0,S))$ there exists $K_T>0$ such that the solution $h$ of $(\ref{4a})$ satisfies:
\begin{equation}
	\int_{0}^{A}\int_{\Omega}h^{2}(x,s,T)dxds\leq K_T\int_{0}^{T}\int_{s_1}^{s_2}\int_{\omega}h^{2}(x,s,t)dxdsdt \label{A}
\end{equation}
\end{theorem}
\begin{proof}
		We give below specific arguments for each of several cases.\smallbreak
		\textbf{Case 1}: \[T>\int_{0}^{S}\frac{ds}{g(s)}\hbox{ then } h(x,s,T)=0\hbox{ for all }(x,s)\in \Omega\times (0,S),\] and the estimate $\ref{A}$ is trivially valid.\smallbreak
\textbf{Case 2}: \[\int_{0}^{s_1}\dfrac{ds}{g(s)}\leq \int_{s_2}^{S}\dfrac{ds}{g(s)}<T\leq \int_{0}^{S}\frac{ds}{g(s)}\] and let \[b_0\in (0,s_1)\hbox{ and } s_0\in(s_1,s_2)\hbox{ such that }\int_{b_0}^{s_2}\frac{ds}{g(s)}=\int_{s_0}^{S}\frac{ds}{g(s)}=T.\] We also have other possibilities if \( b_0 \in [s_1, s_2) \) or \( s_0 \in [0, s_1] \).\smallbreak
Next, we split the interval $(0,S)$ as
\[(0,S)=(0,b_0)\cup (b_0,s_1)\cup (s_1,s_0)\cup (s_0,S).\]
Basically, this division depends on the point in which the trajectory  (or equivalently
 the backward characteristics staring from $(s,T)$) enters the observation region $(s_1,s_2)$ and exits
 from the same region (see figure .).\smallbreak
	
	\textbf{Upper bound on $(0,b_0)$}:\smallbreak
	For $s\in(0,b_0)$, we first set $w(x,\lambda)=h(x,G^{-1}(G(s)+T-\lambda),\lambda)\text{    } (\lambda\in (0,T)\text{ and } x\in\Omega).$ \smallbreak
	Then, $w$ verifies 
	\begin{equation}
		\left\{
		\begin{array}{ccc}
		\dfrac{\partial w(x,\lambda)}{\partial \lambda}-L w(x,\lambda)&=0&\text{ in } (0,T),\\
		\dfrac{\partial w}{\partial \nu}&=0&\text{ on } \partial \Omega\times (0,T)\\
		w(T)&=,h(x,G^{-1}(G(s)),T)&\text{ in }\Omega.
		\end{array}
		\right.,\label{ooo}
	\end{equation}
	By applying the Proposition \ref{a1} with $t_0=0$ and $t_1=G(s)+T-G(s_1),$ we obtain the following.\smallbreak
\[\int_{\Omega}w^2(x,T)dx\leq c_1e^{\dfrac{c_2}{G(s)+T-G(s_1)}}\int_{0}^{G(s)+T-G(s_1)}\int_{\omega}w^2(x,\lambda)dxd\lambda.\]
	Then, we have: 
\[\int_{\Omega}h^2(x,s,T)dx\leq c_1e^{\dfrac{c_2}{T-G(s_1)}}\int_{G^{-1}(G(s_1))}^{G^{-1}(G(s)+T)}\int_{\omega}\frac{1}{g(\alpha)}h(x,\alpha,T+G(s)-G(\alpha))dxd\lambda\]\[=C\int_{G^{-1}(G(s_1))}^{G^{-1}(G(s)+T)}\int_{\omega}h^2(x,\alpha,T+G(s)-G(\alpha))dxd\alpha\leq C\int_{s_1}^{G^{-1}(G(s)+T)}\int_{\omega}h(x,\alpha,T+G(s)-G(\alpha))dxd\alpha.\]	
Integrating with respect $s$ over $(0,b_0)$ we get:
\[\int_{0}^{b_0}\int_{\Omega}h^2(x,s,T)dxds\leq C\int_{0}^{b_0}\int_{s_1}^{G^{-1}(G(s)+T)}\int_{\omega}h^2(x,\alpha,T+G(s)-G(\alpha))dxd\alpha ds.\]
As
\[\int_{0}^{b_0}\int_{s_1}^{G^{-1}(G(s)+T)}\int_{\omega}h^2(x,\alpha,T+G(s)-G(\alpha))dxd\alpha ds=\int_{s_1}^{G^{-1}(G(s)+T)}\int_{G^{-1}(G(\alpha)-T)}^{b_0}\int_{\omega}h^2(x,\alpha,T+G(s)-G(\alpha))dxdsd\alpha ,\]
then,
\[\int_{0}^{b_0}\int_{\Omega}h^2(x,s,T)dxda\leq C\int_{s_1}^{s_2}\int_{0}^{T+G(b_0)-G(\alpha)}\int_{\omega}h^2(x,\alpha,t)dxdtd\alpha \]
Finally,
\begin{equation}
	\int_{0}^{b_0}\int_{\Omega}h^2(x,s,T)dxds\leq C\int_{0}^{T} \int_{s_1}^{s_2}\int_{\omega}h^2(x,s,t)dxdsdt \label{7}
\end{equation}
\textbf{Upper bound $(b_0,s_1)$}:\smallbreak
For $a\in (b_0,s_1)$ we always consider the system $(\ref{ooo})$ but $\lambda\in (T+G(s)-G(s_2),T).$\smallbreak
Applying the Proposition 2.3 with $t_0=G(s)+T-G(s_2)$ and $t_1=G(s)+T-G(s_1)$, we obtain
\[\int_{\Omega}h^2(x,s,T)dx\leq C\int_{s_1}^{s_2}\int_{\omega}\frac{1}{g(\alpha)}h^2(x,\alpha,\alpha-s)dxd\alpha.\]	
And as before, we get:
\begin{equation}
	\int_{b_0}^{s_1}\int_{\Omega}h^2(x,s,T)dxds\leq C\int_{0}^{T}\int_{s_1}^{s_2}\int_{\omega}h^2(x,s,t)dxdsdt \label{9}
	\end{equation}
	$\textbf{Upper bound} (s_1,s_0)$\smallbreak
	As above, for $t_0=T+G(s)-G(_2)$ and $t_1=T$ we obtain this inequality:
		\begin{equation}
	\int_{s_1}^{s_0}\int_{\Omega}h(x,s,T)dxd\leq C\int_{0}^{T}\int_{s_1}^{s_2}\int_{\omega}h^2(x,s,t)dxdsdt \label{10}
	\end{equation}
	Consequently, combining $(\ref{7})$, $(\ref{9})$ and $(\ref{10})$ we obtain:
	\[\int_{0}^{s_0}\int_{\Omega}h^2(x,s,T)dxds\leq C\int_{0}^{T}\int_{s_1}^{s_2}\int_{\omega}h(x,s,t)dxdsdt.\]	 
  \begin{figure}[H]
      \centering
\begin{tikzpicture}[scale=3]
    \begin{scope}
        \clip (-0.5,-0.5) rectangle (4.5,2.5);
        
        \draw[thick] (0,0) rectangle (4,2);
        
        \fill[gray!20] (1,0) rectangle (3,2);
        \draw[thick] (1,0) rectangle (3,2);
    
        \foreach \x in {1.4, 1.5, 1.6, 1.7, 1.8, 1.9, 2, 2.1, 2.2, 2.3, 2.4, 2.5, 2.6, 2.7, 2.8, 2.9, 3} {
            \draw[green, domain=0:2, samples=100, smooth, variable=\t] 
                plot ({\x - sqrt(\t)}, {\t});
        }
        \foreach \x in {3.2, 3.3, 3.4, 3.5, 3.6, 3.7, 3.8, 3.9, 4} {
            \draw[red, domain=0:2, samples=100, smooth, variable=\t] 
                plot ({\x - sqrt(\t)}, {\t});
        }
        \foreach \x in {4.2, 4.3, 4.4, 4.5, 4.6, 4.7, 4.8} {
            \draw[blue, domain=0:2, samples=100, smooth, variable=\t] 
                plot ({\x - sqrt(\t)}, {\t});
        }
    \end{scope}

    \draw[->] (-0.5,0) -- (4.5,0) node[right] {$x$};
    \draw[->] (0,-0.5) -- (0,2.5) node[above] {$y$};
    
    \node at (0,0) [below left] {$(0,0)$};
    \node at (4,0) [below right] {$(4,0)$};
    \node at (4,2) [above right] {$(4,2)$};
    \node at (0,2) [above left] {$(0,2)$};
    
    \node at (1,0) [below left] {$(1,0)$};
    \node at (3,0) [below right] {$(3,0)$};
    \node at (3,2) [above right] {$(3,2)$};
    \node at (1,2) [above left] {$(1,2)$};
\end{tikzpicture}
  \caption{Illustration of the observability inequality. The backward characteristic curves staring at $(s,T)$ in green and red have enough time to reach the observation region before reaching the boundary $t=0,$, while the backward characteristic curves $(s,T)$ in blue exit the domain through the boundary $s=S,$ (which is equivalent to $h(x,s,T)=0$)}
\end{figure}
\end{proof}
Before presenting the proof of Proposition \label{a1}, let us first recall the following observability inequality for the parabolic equation (see, for instance, Imanuvilov and Fursikov \cite{b10})
\begin{proposition}\label{a1}
	Let $T>0,$ $t_0$ and $t_1$ be such that $0<t_0<t_1<T.$ Then for every $w_0\in L^2(\Omega),$ the solution $w$ of the initial and boundary problem.
	\begin{equation}
		\left\lbrace
		\begin{array}{l}
		\dfrac{\partial w(x,\lambda)}{\partial \lambda}-L w(x,\lambda)=0\text{ in } \Omega\times(t_0,T)\\
		\dfrac{\partial w}{\partial \nu}=0 \text{ on } \partial\Omega\times(t_0,T)	\\
		w(x,t_0)=w_0(x)\text{ in }\Omega
		\end{array}\right.,
	\end{equation}
	satisfies the estimate
	\[\displaystyle\int_{\Omega}w^2(T,x)dx\leq\int_{\Omega}w^2(x,t_1)dx\leq c_1e^{\dfrac{c_2}{t_1-t_0}}\int\limits_{t_0}^{t_1}\int_{\omega}w^2(x,\lambda)dxd\lambda,\]
	where the constants $c_1$ and $c_2$ depend on $T$ and $\Omega.$\smallbreak
\end{proposition}
\begin{proof}{Proof of the Proposition \ref{Prop1}}\smallbreak
Throughout the rest of the proof, without losing generality, we assume that $a_2=\hat{a}.$\smallbreak
For $a\in (0,\hat{a})$ (assumption $H_3$) we have $\beta(a,s,\hat{s})=0$, therefore the system (\ref{3}) is can be written by
\begin{equation}
\left\lbrace \begin{array}{l}
\dfrac{\partial q}{\partial t}-\dfrac{\partial q}{\partial a}-
g(s)\dfrac{\partial q}{\partial s}-L q+(\mu_1(a)+\mu_2(s)) q=0\text{ in }\Omega\times(0,\hat{a})\times(0,S)\times (0,T),\\ 
q(x,a,s,0)=q_0(x,a,s) \text{ in }\Omega\times(0,\hat{a})\times(0,S).
\end{array}
\right.
\label{ad11}
\end{equation}	
We denote by \[\tilde{q}(x,a,s,t)=q(x,a,s,t)\exp\left(-\int\limits_{0}^{a}\mu_1(\alpha)d\alpha-\int\limits_{0}^{s}\mu_2(r)dr\right).\] then $\tilde{q}$ satisfies
\begin{align}
	\dfrac{\partial \tilde{q}}{\partial t}-\dfrac{\partial \tilde{q}}{\partial a}-g(s)\dfrac{\partial \tilde{q}}{\partial s}-\Delta\tilde{q}=0\text{ in }\Omega\times(0,\hat{a})\times(0,S)\times (0,T). \label{ad1}
\end{align}
Let $S^*<S;$ (the real $S^*$ verifying $q(x,0,s,t)=0\hbox{ in }\Omega\times(S^*,S)\times(\eta,T)$ (to be explained later) proving the inequality (\ref{Bq}) leads also to show that,
there exits a constant $C>0$ such that the solution $\tilde{q}$ of (\ref{ad11}) satisfies 
\begin{equation}
\int\limits_{\eta}^{T}\int\limits_{0}^{S^*}\int\limits_{\Omega}q^2(x,0,s,t)dxdsdt\leq C\int\limits_{0}^{T}\int\limits_{s_1}^{s_2}\int\limits_{a_1}^{a_2}\int_{\omega}\tilde{q}^2(x,a,s,t)dxdadsdt. \label{cci}
\end{equation}
Indeed, we have 
\[
\int\limits_{\eta}^{T}\int\limits_{0}^{S^*}\int_{\Omega}q^2(x,0,s,t)dxdsdt\leq e^{2\int\limits_{0}^{\hat{a}}\mu_1(r)dr+2\int\limits_{0}^{S^*}\mu_2(r)dr}\int\limits_{\eta}^{T}\int\limits_{0}^{S}\int_{\Omega}\tilde{q}^2(x,0,s,t)dxdsdt\]
and
\[\int\limits_{0}^{T}\int\limits_{s_1}^{s_2}\int\limits_{a_1}^{a_2}\int_{\omega}\tilde{q}^2(x,a,s,t)dxdadsdt\leq C(e^{2\|\mu_1\|_{L^1(0,\hat{a})}+2\|\mu_2\|_{L^1(0,S^*)}})\int\limits_{0}^{T}\int\limits_{s_1}^{s_2}\int\limits_{a_1}^{a_2}\int_{\omega}q^2(x,a,s,t)dxdadsdt.
\]
We consider the following characteristic trajectory $\gamma(\lambda)=(t-\lambda,G^{-1}(G(s)+t-\lambda),\lambda).$ If $t-\lambda=0$ the backward characteristics starting from $(0,s,t).$ If $T<T_1,$ we cannot have information on all characteristics (see Figure 5 6). So we choose $T>T_1.$\smallbreak
	Without loss of generality, let us assume here $\eta< T\leq a_2.$\smallbreak
	The proof will be in two steps:\smallbreak
	\textbf{Step 1 : In this step, we show that the non local term $$q(x,0,s,t)=0\text{ for all }s\in (\alpha^*,S)\text{ and }t\in\left(a_1+S^{*}_{2},T\right).$$}\smallbreak
	First, we assume the existence of $\alpha^*>0$ such that $G(s_2)-G(\alpha^*)=a_1.$ This remains consistently true under the condition \( S^*_2 < \min\{\hat{a}, a_2\} \) and \( G(S) > \hat{a} \). \smallbreak
According to the assumptions $S_{2}^*<a_2-a_1$, 
 it's easy to proof (using the representation of $q$ by the semi group method) that $q(x,0,s,t)=0$ for all $$s\in (\alpha^*,S)\hbox{ and } t\in \left(a_1+S^{*}_2, T\right)$$ (see Figure 4).\smallbreak
Indeed, for every $$(s,t)\in (\alpha^*,S)\times \left(a_1+S^{*}_{2},T\right)\hbox{ and }S_{2}^*<a_2-a_1$$ we have, $$t-A<t-G(S)+G(s)\hbox{ and }G(S)-G(s)<S^{*}_{2}+a_1<t$$ then,
\[q(x,0,s,t)=\int\limits_{t-G(S)+G(s)}^{t}\left(e^{(t-l)L}\int\limits_{0}^{S}\beta(t-l,G^{-1}(G(s)+t-l),\hat{s})q(x,0,\hat{s},l)d\hat{s})\right)dl.\]
Moreover, we have $$0<t-l<G(S)-G(s)\hbox{ and as }s\in (\alpha^*,S)$$ then $$0<t-l<G(S)-G(s)<S^{*}_2+a_1<a_2.$$ The fertility $\beta$ being assumed to be zero on $(0,a_2),$ then $q(x,0,s,t)=0.$\smallbreak
Likewise, if $\alpha^*<0,$ the non local term verifies $$q(x,0,s,t)=0\hbox{ for all }s\in (0,S)\hbox{ and } t\in \left(G(S), T\right).$$\smallbreak
Let us then assume the existence of $\alpha^*>0$ such that $G(s_2)-G((\alpha^*)=a_1.$\smallbreak
\textbf{Step 2: Estimation of the non local term $$q(x,0,s,t)\hbox{ for }(s,t)\in (0,\alpha^*-\delta_2)\times (\eta,T),$$ where $\delta_2>0$ and $T_1<\eta<T$}.\smallbreak
In the case where \(\alpha^* > 0\), there can exist a real number \(s_1 > \beta^* > 0\) such that \(G(s_1)-G(\beta^*)= a_1\), or not.

\textbf{Case 1: no existence of $\beta^*>0$\hbox{ such that } $G(s_1)-G(\beta^*)=a_1$}.\smallbreak
We denote by:\smallbreak
\[w(x,\lambda)=\tilde{q}(x,t-\lambda,G^{-1}(G(s)+t-\lambda),\lambda) \text{ ; }(x\in\Omega\hbox{, }\lambda\in (0,t))\]
Then $w$ satisfies:
\begin{align}
	\left\lbrace
	\begin{array}{l}
		\dfrac{\partial w(x,\lambda)}{\partial \lambda}-L w(x,\lambda)=0\text{ in } ( \Omega\times (0,t))\\
		\dfrac{\partial w}{\partial \nu}=0 \text{ on } \partial\Omega\times (0,t)	\\
		w(x,0)	=\tilde{q}(x,t,G^{-1}(G(s)+t),0)\text{ in }\Omega
	\end{array}
	\right.,
\end{align}
Using the Proposition \ref{a1}  with $0<t_0<t_1<t$ we obtain:\smallbreak
\[\int_{\Omega}w^2(x,t)dx\leq\int_{\Omega}w^2(x,t_1)dx\leq c_1e^{\dfrac{c_2}{t_1-t_0}}\int\limits_{t_0}^{t_1}\int_{\omega}w^2(x,\lambda)dxd\lambda.\]
That is equivalent to
\[\int_{\Omega}\tilde{q}^2(x,0,s,t)dx\leq c_1e^{\frac{c_2}{t_1-t_0}}\int\limits_{t_0}^{t_1}\int_{\omega}\tilde{q}^2(x,t-\lambda,G^{-1}(G(s)+t-\lambda),\lambda)dxd\lambda\]\[=c_1e^{\frac{c_2}{t_1-t_0}}\int\limits_{t-t_1}^{t-t_0}\int_{\omega}\tilde{q}^2(x,\alpha,G^{-1}(G(s)+\alpha),t-\alpha)dxd\alpha.\]	
Then for $t_0=t-a_1-\delta_1$ and $t_1=t-a_1,$ we obtain\smallbreak
\[\int_{\Omega}\tilde{q}^2(x,0,s,t)dx\leq c_1e^{\frac{c_2}{\delta_1}}\int\limits_{a_1}^{a_1+\delta_1}\int_{\omega}\tilde{q}^2(x,\alpha,G^{-1}(G(s)+\alpha),t-\alpha)dxd\alpha.\]
Integrating with respect $s$ over $(0,\alpha^*-\delta_2)$ we get
\[\int\limits_{0}^{\alpha^*-\delta_2}\int_{\Omega}\tilde{q}^2(x,0,s,t)dxds\leq c_1e^{\frac{c_2}{\delta_1}}
\int\limits_{a_1}^{a_1+\delta_1}\int\limits_{G^{-1}(G(0)+a)}^{G^{-1}(G(\alpha^*-\delta_2)+a)}\int_{\omega}\frac{1}{g(l)}\tilde{q}^2(x,a,l,t-a)dxdlda.\]
But \[a_1>G(s_1)\] due to the non-existence of $\beta^*,$ then
\[\int\limits_{a_1}^{a_1+\delta_1}\int\limits_{G^{-1}(G(s_1))}^{G^{-1}\left(G(\alpha^*-\delta_2)+a\right)}\int_{\omega}\frac{1}{g(l)}\tilde{q}^2(x,a,l,t-a)dxdlda\]\[\leq C_{\delta_1,\delta_2}\int\limits_{a_1}^{a_1+\delta_1}\int\limits_{G^{-1}(G(s_1))}^{G^{-1}\left(\int\limits_{0}^{\alpha^*-\delta_2}\frac{ds}{g(s)}+\int\limits_{\alpha^*}^{s_2}\frac{ds}{g(s)}+\delta_1\right)}\int_{\omega}\frac{1}{g(l)}\tilde{q}^2(x,a,l,t-a)dxdlda.\]
Now, we choose $\delta_1>0\hbox{ and }\delta_2>0$ such that \[G(s_1)<\int\limits_{0}^{\alpha^*-\delta_2}\frac{ds}{g(s)}+\int\limits_{\alpha^*}^{s_2}\frac{ds}{g(s)}+\delta_1\leq G(s_2)\hbox{ and }a_1+\delta_1\leq a_2\] in order to have a lower cost.\smallbreak
The above inequality becomes
\[\int\limits_{a_1}^{a_1+\delta_1}\int\limits_{G^{-1}(G(s_1))}^{G^{-1}\left(\int\limits_{0}^{\alpha^*-\delta_2}\frac{ds}{g(s)}+\int\limits_{\alpha^*}^{s_2}\frac{ds}{g(s)}+\delta_1\right)}\int_{\omega}\frac{1}{g(l)}\tilde{q}^2(x,a,l,t-a)dxdlda\]\[\leq C_{\delta_1,\delta_2}\int\limits_{a_1}^{a_1+\delta_1}\int\limits_{G^{-1}(G(s_1))}^{G^{-1}\left(\int\limits_{0}^{\alpha^*}\frac{ds}{g(s)}+\int\limits_{\alpha^*}^{s_2}\frac{ds}{g(s)}\right)}\int_{\omega}\frac{1}{g(l)}(l)\tilde{q}^2(x,a,l,t-a)dxdlda.\]
Finaly, integrating with respect $t$ over $(\eta,T)$, we obtain
\[\int\limits_{\eta}^{T}\int\limits_{0}^{\alpha^*-\delta}\int_{\Omega}\tilde{q}^2(x,0,s,t)dxdsdt\]\[\leq C_{\delta_1,\delta_2} \int\limits_{a_1}^{a_1+\delta_1}\int\limits_{s_1}^{s_2}\int\limits_{\eta-a}^{T-a}\int_{\omega}\tilde{q}^2(x,a,s,t)dxdtdsda\]\[\leq C_{\delta_1,\delta_2,\eta}\int\limits_{0}^{T} \int\limits_{a_1}^{a_2}\int\limits_{s_1}^{s_2}\int_{\omega}\tilde{q}^2(x,a,s,t)dxdsdadt.\]
Then
\begin{equation}
	\int\limits_{\eta}^{T}\int\limits_{0}^{\alpha^*-\delta_2}\int_{\Omega}\tilde{q}^2(x,0,s,t)dxdsdt\leq C_{\delta_1,\delta_2,\eta}\int\limits_{0}^{T} \int\limits_{s_1}^{s_2}\int\limits_{a_1}^{a_2}\int_{\omega}\tilde{q}^2(x,a,s,t)dxdsdadt.
\end{equation}
\textbf{Case 2: there exists $\beta^*>0$\hbox{ such that } $G(s_1)-G(\beta^*)=a_1$}\smallbreak 
Then, we split the interval $(0,\alpha^*)=(0,\beta^*)\cup (\beta^*,\alpha^*).$
The case $s\in (\beta^*,\alpha^*-\delta_2)$ is similar to the above.\smallbreak
Indeed,
\[w(x,\lambda)=\tilde{q}(x,t-\lambda,G^{-1}(G(s)+t-\lambda),\lambda) \text{ ; }(x\in\Omega\text{ , }\lambda\in(0,t))\]
Then $w$ satisfies:
\begin{equation}
	\left\lbrace
	\begin{array}{l}
		\dfrac{\partial w(x,\lambda)}{\partial \lambda}-L w(x,\lambda)=0\text{ in } \Omega\times (0,t)\\
		\dfrac{\partial w}{\partial \nu}=0 \text{ on } \partial\Omega\times (0,t)	\\
		w(x,0)	=\tilde{q}(x,t,G^{-1}(G(s)+t),0)\text{ in }\Omega
	\end{array}\right.,
\end{equation}
Using the Proposition \ref{a1}  with $0<t_0<t_1<t$ we obtain:\smallbreak
\[\int_{\Omega}w^2(x,t)dx\leq\int_{\Omega}w^2(x,t_1)dx\leq c_1e^{\dfrac{c_2}{t_1-t_0}}\int\limits_{t_0}^{t_1}\int_{\omega}w^2(x,\lambda)dxd\lambda.\]
That is equivalent to
\[\int_{\Omega}\tilde{q}^2(x,0,s,t)dx\leq c_1e^{\frac{c_2}{t_1-t_0}}\int\limits_{t_0}^{t_1}\int_{\omega}\tilde{q}^2(x,t-\lambda,G^{-1}(G(s)+t-\lambda),\lambda)dxd\lambda\]\[=c_1e^{\frac{c_2}{t_1-t_0}}\int\limits_{t-t_1}^{t-t_0}\int_{\omega}\tilde{q}^2(x,\alpha,G^{-1}(G(s)+\alpha),t-\alpha)dxd\alpha.\]	
and we denote by $t_0=t-a_1-\delta_1$ and $t_1=t-a_1,$ we obtain\smallbreak
\[\int_{\Omega}\tilde{q}^2(x,0,s,t)dx\leq C_{\delta_1}\int\limits_{a_1}^{a_1+\delta_1}\int_{\omega}\tilde{q}^2(x,\alpha,G^{-1}(G(s)+\alpha),t-\alpha)dxd\alpha.\]
Integrating with respect to $s$ over $(\beta^*,\alpha^*-\delta_2)$ we get
\[\int\limits_{\beta^*}^{\alpha^*-\delta_2}\int_{\Omega}\tilde{q}^2(x,0,s,t)dxds\leq C_{\delta_1,\delta_2}
\int\limits_{a_1}^{a_1+\delta_1}\int\limits_{s_1}^{s_2}\int_{\omega}\frac{1}{g(l)}\tilde{q}^2(x,a,l,t-a)dxdlda.\]
Finally, integrating with respect to $t$ over $(\eta,T)$, we obtain
\[\int\limits_{\eta}^{T}\int\limits_{\beta^*}^{\alpha^*-\delta_2}\int_{\Omega}\tilde{q}^2(x,0,s,t)dxdsdt\leq C_{\delta_1,\delta_2,\eta} \int\limits_{a_1}^{a_1+\delta}\int\limits_{s_1}^{s_2}\int\limits_{\eta-a}^{T-a}\int_{\omega}\tilde{q}^2(x,a,s,t)dxdtdsda\]\[\leq C_{\delta_1,\delta_2,\eta}\int\limits_{0}^{T} \int\limits_{a_1}^{a_2}\int\limits_{s_1}^{s_2}\int_{\omega}\tilde{q}^2(x,a,s,t)dxdsdadt.\]
Then
\begin{equation}
	\int\limits_{\eta}^{T}\int\limits_{\beta^*}^{\alpha^*-\delta_2}\int_{\Omega}\tilde{q}^2(x,0,s,t)dxdsdt\leq C_{\delta_1,\delta_2,\eta}\int\limits_{0}^{T} \int\limits_{s_1}^{s_2}\int\limits_{a_1}^{a_2}\int_{\omega}\tilde{q}^2(x,a,s,t)dxdsdadt.\label{rr}
\end{equation}
\textbf{The case $s\in (0,\beta^*)$} \smallbreak 
Here again we have two situations:
\[G(s_2)-G(\beta^*)> a_2\hbox{ and }G(s_2)-G(\beta^*)< a_2.\]
If $G(s_2)-G(\beta^*)< a_2,$ we split $(0,\beta^*)$ as follows
\[(0,\beta^*)=(0,\gamma^*)\cup(\gamma^*,\beta^*),\] if there exists $\gamma^* $ such that $\int\limits_{\gamma^*}^{s_2}\frac{ds}{g(s)}=a_2,$ otherwise we estimate directly on $(0,\beta^*).$
Here we will perform the calculations only in the case $G(s_2)-G(\beta^*)> a_2.$\smallbreak
We denote by:\smallbreak
\[w(x,\lambda)=\tilde{q}(x,t-\lambda,G^{-1}(G(s)+t-\lambda),\lambda) \text{ ; }(x\in\Omega\hbox{, }\lambda\in (0,t))\]
Then $w$ satisfies:
\begin{align}
	\left\lbrace
	\begin{array}{l}
		\dfrac{\partial w(x,\lambda)}{\partial \lambda}-L w(x,\lambda)=0\text{ in } ( \Omega\times (0,t))\\
		\dfrac{\partial w}{\partial \nu}=0 \text{ on } \partial\Omega\times (0,t)	\\
		w(x,0)	=\tilde{q}(x,t,G^{-1}(G(s)+t),0)\text{ in }\Omega
	\end{array}
	\right.,
\end{align}
Using the Proposition \ref{a1}  with $0<t_0<t_1<t$ we obtain:\smallbreak
\[\int_{\Omega}w^2(x,t)dx\leq\int_{\Omega}w^2(x,t_1)dx\leq c_1e^{\dfrac{c_2}{t_1-t_0}}\int\limits_{t_0}^{t_1}\int_{\omega}w^2(x,\lambda)dxd\lambda.\]
That is equivalent to
\[\int_{\Omega}\tilde{q}^2(x,0,s,t)dx\leq c_1e^{\dfrac{c_2}{t_1-t_0}}\int\limits_{t_0}^{t_1}\int_{\omega}\tilde{q}^2(x,t-\lambda,G^{-1}(G(s)+t-\lambda),\lambda)dxd\lambda\]\[=c_1e^{\frac{c_2}{t_1-t_0}}\int\limits_{t-t_1}^{t-t_0}\int_{\omega}\tilde{q}^2(x,\alpha,G^{-1}(G(s)+\alpha),t-\alpha)dxd\alpha.\]
Then for $t_0=t-G(s_1+\kappa)$ where $\kappa>0$ and $t_1=t-G(s_1),$ we obtain the following.\smallbreak
\[\int_{\Omega}\tilde{q}^2(x,0,s,t)dx\leq C_{\kappa}\int\limits_{G(s_1)}^{G(s_1+\kappa)}\int_{\omega}\tilde{q}^2(x,\alpha,G^{-1}(G(s)+\alpha),t-\alpha)dxd\alpha.\]
We known $a_1<G(s_1).$\smallbreak
Integrating with respect to $s$ over $(0,\beta^*)$ we get
\[\int\limits_{0}^{\beta^*}\int_{\Omega}\tilde{q}^2(x,0,s,t)dxds\leq C_{\kappa}
\int\limits_{G(s_1)}^{G(s_1+\kappa)}\int\limits_{G^{-1}(G(s_1))}^{G^{-1}(G(\beta^*)+G(s_1+\kappa))}\int_{\omega}\frac{1}{g(l)}\tilde{q}^2(x,a,l,t-a)dxdlda.\] We obtain
\[\int\limits_{0}^{\beta^*}\int_{\Omega}\tilde{q}^2(x,0,s,t)dxds\leq C_{\kappa}
\int\limits_{a_1}^{a_2}\int\limits_{s_1}^{G^{-1}(G(\beta^*)+G(s_1+\kappa))}\int_{\omega}\tilde{q}^2(x,a,l,t-a)dxdlda.\]
 we choose $\kappa>0$ such that $G(s_1+\kappa)+G(s_1)<G(s_2)>a_2+G(\beta^*);$ which is possible because $a_2>G(s_1)+a_1$ and $a_1=G(s_1)-G(\beta^*)$, therefore $a_2+G(\beta^*)>2G(s_1). $
 Then \[\int\limits_{0}^{\beta^*}\int_{\Omega}\tilde{q}^2(x,0,s,t)dxds\leq C_{\kappa}
\int\limits_{a_1}^{a_2}\int\limits_{s_1}^{s_2}\int_{\omega}\tilde{q}^2(x,a,l,t-a)dxdlda.\]
Finally, integrating with respect to $t$ over $(\eta,T)$, we obtain
\[\int\limits_{\eta}^{T}\int\limits_{0}^{\beta^*}\int_{\Omega}\tilde{q}^2(x,0,s,t)dxdsdt\leq C_{\kappa,\eta} \int\limits_{s_1}^{s_1+\kappa}\int\limits_{s_1}^{G(s_1+\kappa)+G(s_1)}\int\limits_{\eta-a}^{T-a}\int_{\omega}\tilde{q}^2(x,a,s,t)dxdtdsda.\]
We choose $\kappa$ small enough such that \[G(s_1+\kappa)<T<a_2\]
Then, we get
\begin{equation}
	\int\limits_{\eta}^{T}\int\limits_{0}^{\beta^*}\int_{\Omega}\tilde{q}^2(x,0,s,t)dxdsdt\leq C_{\kappa,\eta}\int\limits_{0}^{T} \int\limits_{s_1}^{s_2}\int\limits_{a_1}^{a_2}\int_{\omega}\tilde{q}^2(x,a,s,t)dxdsdadt.
\end{equation}
Finally,
combining $(\ref{rr})$ and the fact that $$q(x,0,s,t)=0\hbox{ for } t\in \left(S^{*}_{2}+a_1+\delta_2,+\infty\right)\hbox{ } s\in (\alpha^*-\delta_2,S)$$, we obtain:
\[\int\limits_{\eta}^{T}\int\limits_{0}^{S}\int_{\Omega}\tilde{q}^2(x,0,s,t)dxdsdt\leq C_{\kappa,\eta,\delta_1,\delta_2}\int\limits_{0}^{T} \int\limits_{s_1}^{s_2}\int\limits_{a_1}^{a_2}\int_{\omega}\tilde{q}^2(x,a,s,t)dxdsdadt\] where $\max\left\{S^{*}_{2}+a_1+\delta_2,S^{*}_{1}\right\}<\eta<T.$
\begin{remark}
	In all the cases, when $\delta_1 \longrightarrow 0\hbox{ or }\kappa \longrightarrow0 \hbox{ or }\eta \longrightarrow T_1$ we have $C_{\eta,\kappa,\delta_1,\delta_2}\longrightarrow +\infty.$ 
\end{remark}
\end{proof}
In this illustration, we have taken the characteristics defined by \((- \lambda + a_0, -g(\lambda) + s_0, \lambda + t_0)\) with \(g(\lambda) = \sqrt{2\lambda}\).
 \begin{figure}[H]
 \begin{adjustbox}{center,left}
			\begin{overpic}[scale=0.2]{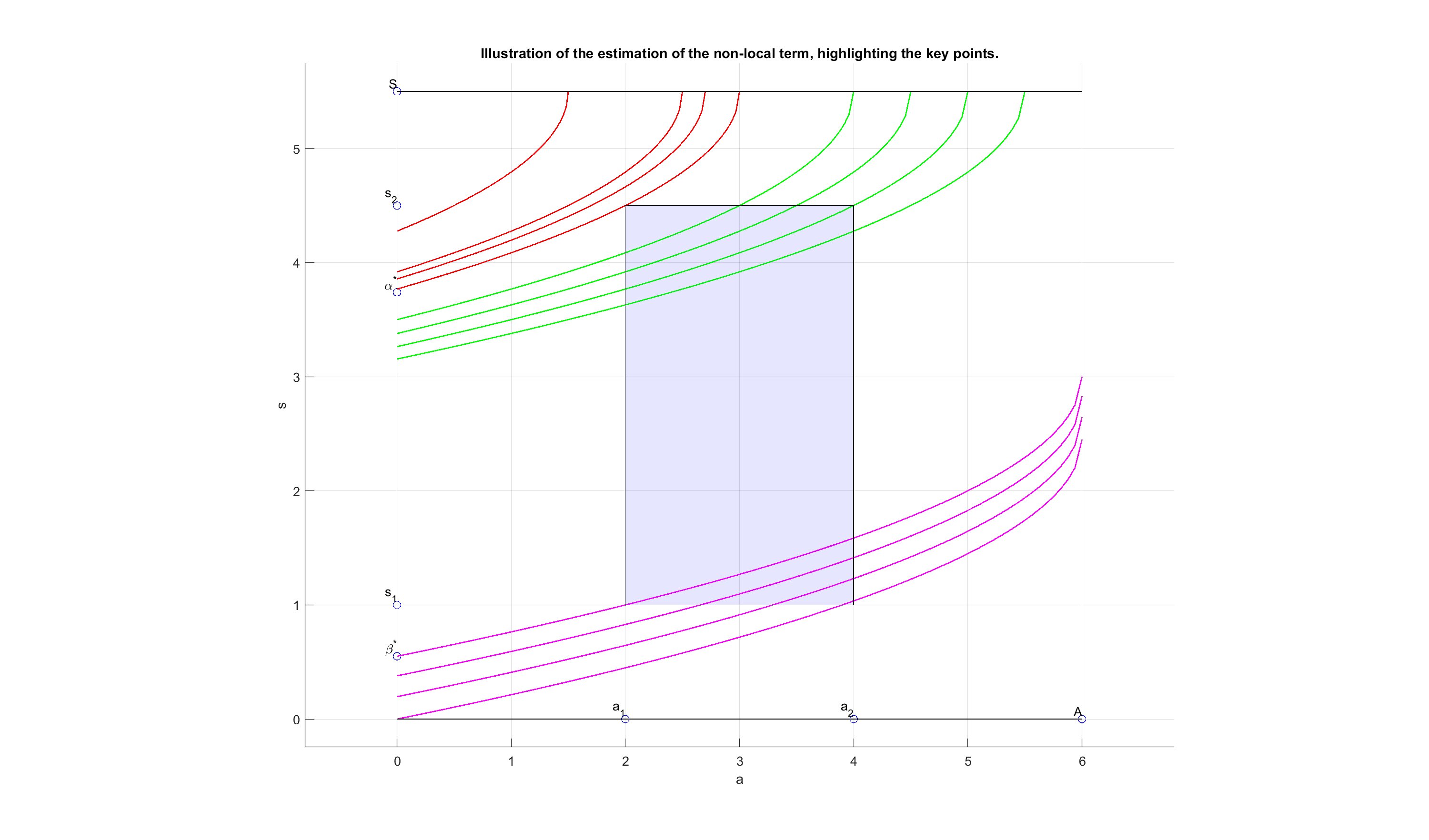}
\end{overpic}
\end{adjustbox}
\caption{An illustration of the estimate of $q(x,0,s,t)$\smallbreak Here we have chosen $a_2=\hat{a}$. Since $t>T_1$ all the backward characteristics starting from $(0,s,t)$ enters the observation domain (the green and blue lines), or without the domain by the boundary $s=S$ (red line)}
\end{figure}
\begin{figure}[H]
\begin{adjustbox}{center,left}
		\begin{overpic}[scale=0.2]{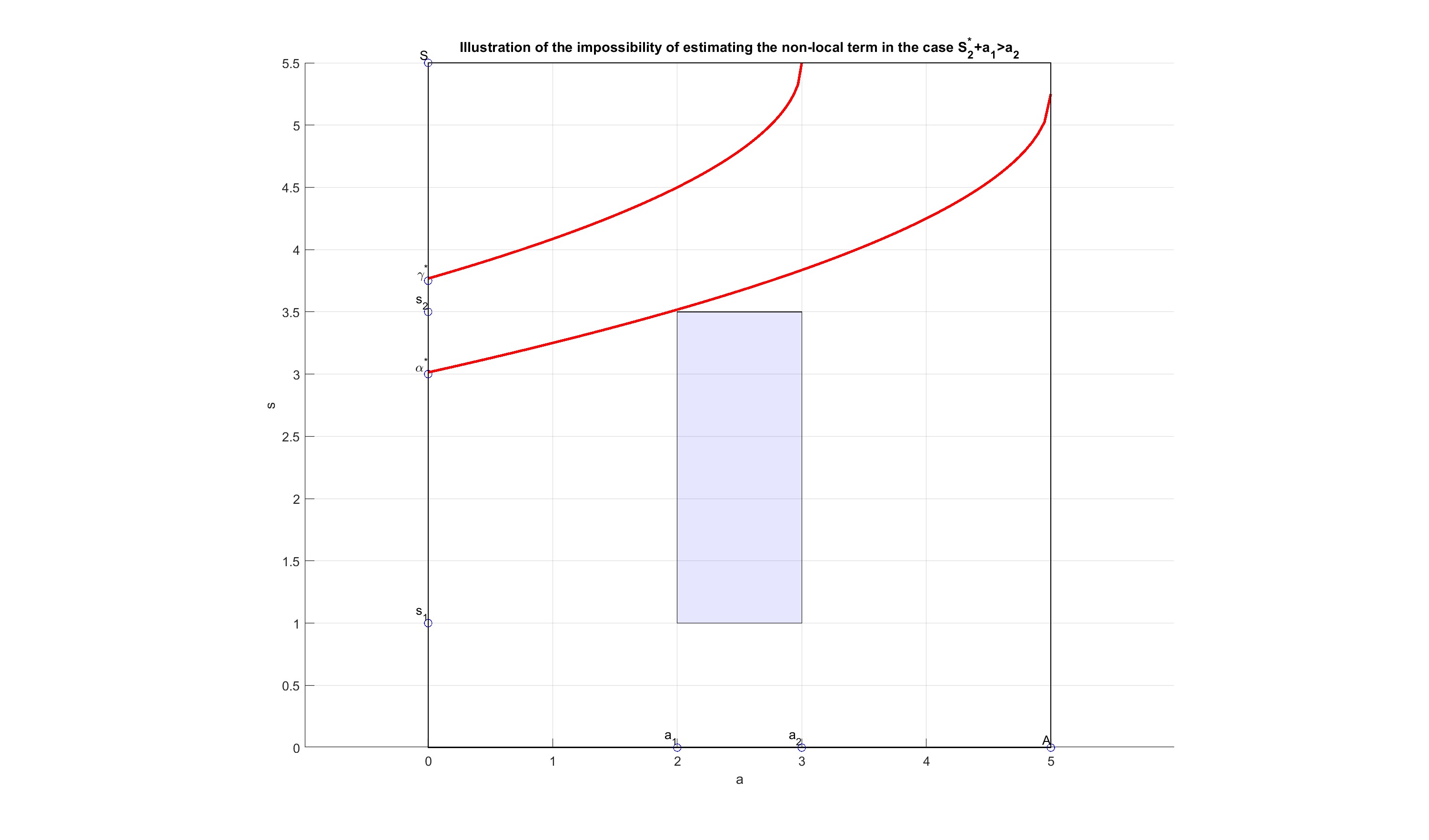}
		\end{overpic}
  \end{adjustbox}
  \caption{For $S^{*}_2>a_2-a_1,$ we can not estimate $q (x,0,s,t)$ for $s\in(\alpha^*,v^*)\hbox{ or }s\in (\alpha^*,S),$ (where $v^*$ is such that $\int_{v^*}^{S}\frac{ds}{g(s)}=a_2$) by the characteristic method. Indeed for even $t>T_1\hbox{ and }s\in (\alpha^*,v^*)$ the characteristics starting at $(0,s,t)$ without the domain by the boundary $t=0 \hbox{ or enter in the region } a>a_2,$ without going through the observation domain.}
	\end{figure}
	\begin{figure}[H]
 \begin{adjustbox}{center,left}
		\begin{overpic}[scale=0.2]{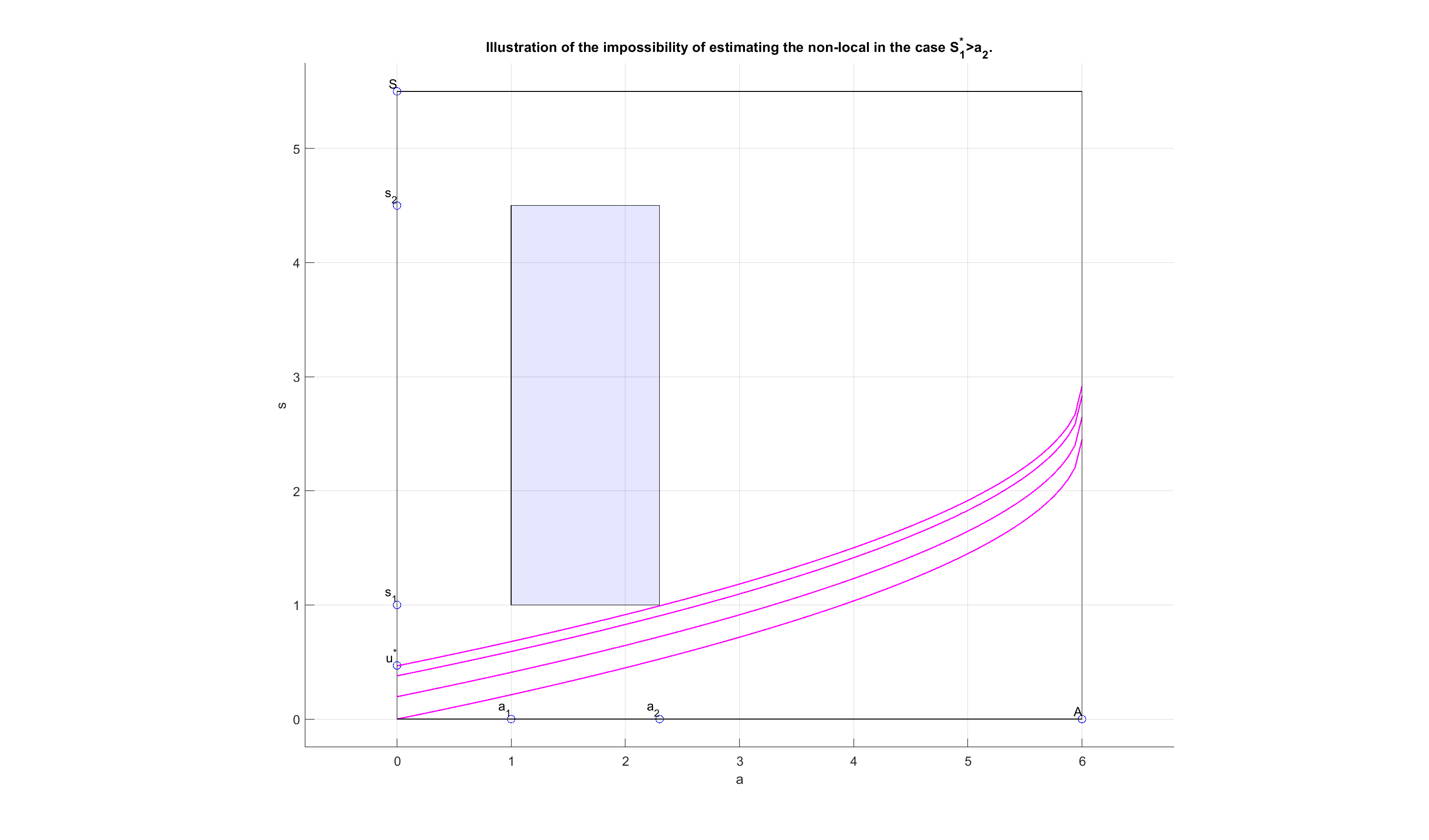}
		
		\end{overpic}
\end{adjustbox}
\caption{For the second case if $S^{*}_1>a_2,$ we can not estimate $q (x,0,s,t)$ for $s\in (0,u^*)$ by the characteristic method, where $u^*$ is such that $\int_{u^*}^{s_1}\frac{ds}{g(s)}=a_2$. Indeed, for even $t>T_1\hbox{ and } s\in (0,u^*)$ the characteristics starting at $(0,s,t)$ without the domain by the boundary $t=0 \hbox{ or enter in the region } a>a_2,$ without going through the observation domain.}\end{figure}
\begin{proof}{Proposition \ref{a3}.}\smallbreak
\textbf{Proof of inequality \ref{Cbb1}}\smallbreak
We denote by $$\mathbf{A}\psi=\partial_{s}(g(s)\psi)-L\psi+\mu_2(s)\psi$$ the operator defined on \[D(\mathbf{A})=\{\varphi/ \mathbf{A}\varphi\in L^{2}(\Omega\times (s_1,S)), \quad \varphi(x,s_1)=\varphi_{s_1}\quad  \dfrac{\partial\varphi}{\partial\nu}|_{\partial\Omega}=0 \}\], $$\mathbf{B}=\mathbf{1}_{(s_1,s_2)}\mathbf{1}_{\omega}\hbox{ and } \mathcal{B}=\mathbf{1}_{(a_1,a_2)}\mathbf{B}.$$\smallbreak
		The operator $\mathcal{A}$ can be rewritten by: \[\mathcal{A}\varphi=-\partial_{a}\varphi-\mathbf{A}\varphi-\mu_1(a)\varphi\quad for \quad \varphi\in D(\mathcal{A}),\]
and the adjoint $\mathcal{A}^*$ of the operator $\mathcal{A}$  is defined by:
		\[\mathcal{A}^*\varphi=\partial_{a}\varphi-\mathbf{A}^*\varphi-\mu_1(a)\varphi +\int\limits_{0}^{S}\beta(a,s,\hat{s})\varphi(x,0,\hat{s})d\hat{s}\hbox{ for } \varphi\in D(\mathcal{A^*}).\]
		 The $(\mathbf{A}^*,\mathbf{B}^*)$ is final state observable for every $T>S_{2}^*$ according to Theorem \ref{resum} in the Appendix. Finally, the results of Proposition 3.5. of \cite{abst} considering the adjoint system on $\Omega\times(0,A)\times (s_1,S)\times (0,T),$ give the inequality (\ref{Cbb1}).\smallbreak
   \textbf{ Proof of inequality \ref{C8aa} and \ref{xcxa}}
   
   Here also, we consider the following characteristic trajectory \[\gamma(\lambda)=(a+T-\lambda,G^{-1}(G(s)+T-\lambda),\lambda).\] 
   If $\lambda=T$ the backward characteristics starting from $(a,s,T).$ If $T<T_1,$ we cannot have information on all characteristics. So we choose $T>T_1.$\smallbreak
   Let $\kappa^*$ such that $G(s_1)=a_2-\kappa^*$ and consider $a^*_0>0$ such that $a^*_0<\kappa^*.$\smallbreak
For inequality the \ref{C8aa}, wee proceed as in the proof of the inequality \ref{Cbb1}.\smallbreak
Let as proof the inequality \ref{xcxa}.\smallbreak
Let $b_0$ such that $G(s_1)=b_0-a^*_0$ and we define 
\[D_1=\{(a,s)\in [a^*_0,b_0]\times[0,s_1]\hbox{ such that }a^*_0\leq a\leq b_0\hbox{ and }G^{-1}(a-a^*_0)\leq s\leq s_1\}.\] 
Without loss of generality, we assume that $G(s_2) - G(s_1) \geq a_2 - a_0^*$; otherwise, let $b_1 \in (a_0^*, a_2)$ and $b_2 \in (0, s_1)$ such that $G(s_2) - G(s_1) = a_2 - b_1$ and $G(s_2) - G(b_2) = a_2 - a_0^*$. Therefore, we divide $D_1$ into $D_1^1 \cup D_1^2 \cup D_1^3$ as follows:
\[D^1_1=\{(a,s)\hbox{ such that }a^*_0\leq a\leq b_1\hbox{ and }G^{-1}(G(b_2)+a-a^*_0)\leq s\leq s_1\},\] 
\[D^2_1=\{(a,s)\hbox{ such that }a^*_0\leq a\leq b_1\hbox{ and }G^{-1}(a-a^*_0)\leq s\leq G^{-1}(G(b_2)+a-a^*_0)\},\] and
\[D^3_1=\{(a,s)\hbox{ such that }b_1\leq a\leq b_0\hbox{ and }G^{-1}(a-a^*_0)\leq s\leq s_1\}\]
	For $(a,s)\in D_1$, we first set $w(x,\lambda)=\tilde{q}(x,a+T-\lambda,G^{-1}(G(s)+T-\lambda),\lambda)\text{    } (\lambda\in (0,T)\text{ and } x\in\Omega).$ \smallbreak
	Then, $w$ verifies 
	\begin{equation}
		\left\{
		\begin{array}{ccc}
		\dfrac{\partial w(x,\lambda)}{\partial \lambda}-L w(x,\lambda)&=0&\text{ in } (0,T),\\
		\dfrac{\partial w}{\partial \nu}&=0&\text{ on } \partial \Omega\times (0,T)\\
		w(0)&=,q(x,a+T,G^{-1}(G(s)+T),0)&\text{ in }\Omega.
		\end{array}
		\right.,\label{ooo}
	\end{equation}
	By applying the Proposition \ref{a1} with $t_0=a+T-a_2$ and $t_1=T-a_2+\delta,$  we obtain the following.\smallbreak
\[\int_{\Omega}w^2(x,T)dx\leq c_1e^{\dfrac{c_2}{\delta}}\int_{a+T-a_2}^{T-a_2+\delta}\int_{\omega}w^2(x,\lambda)dxd\lambda.\]
Then, we get:
\[\int_{\Omega}\tilde{q}^2(x,a,s,T)dx\leq c_1e^{\dfrac{c_2}{\delta}}\int_{a_2-\delta}^{a_2}\int_{\omega}\tilde{q}(x,u,G^{-1}(G(s)+u-a),a+T-u)dxdu\]
Integrating with respect $s$ over $(G^{-1}(a-a^*_0),s_1)$ we obtain:
\[\int_{G^{-1}(a-a^*_0)}^{s_1}\int_{\Omega}\tilde{q}^2(x,a,s,T)dxds\leq C\int_{a_2-\delta}^{a_2}\int_{G^{-1}(a-a^*_0)}^{s_1}\int_{\omega}\tilde{q}^2(x,u,G^{-1}(G(s)+u-a),a+T-u)dxds du.\]
Then,
\[\int_{G^{-1}(a-a^*_0)}^{s_1}\int_{\Omega}\tilde{q}^2(x,a,s,T)dxds\leq C\int_{a_2-\delta}^{a_2}\int_{G^{-1}(u-a^*_0)}^{G^{-1}(G(s_1)+u-a))}\int_{\omega}\tilde{q}^2(x,u,l,a+T-u)dxdl du.\]
Using the fact that $G(s_2)-G(s_1)\geq a_2-a_0^*$ and choosing $\delta$ such that $G^{-1}(a_2-\delta-a^*_0)\geq s_1$ we get after we integrate the previous result over $(a^*_0,b_0)$ that
\[\int_{a^*_0}^{b_0}\int_{G^{-1}(a-a^*_0)}^{s_1}\int_{\omega}\tilde{q}^2(x,a,s,T)dx ds da \leq C\int_{a_2-\delta}^{a_2}\int_{s_1}^{s_2}\int_{a^*_0}^{b_0}\int_{\omega}\tilde{q}^2(x,u,l,a+T-u)dxdadldu .\]
Therefore,
\[\int_{a^*_0}^{b_0}\int_{G^{-1}(a-a^*_0)}^{s_1}\int_{\omega}\tilde{q}^2(x,a,s,T)dx ds da \leq C\int_{a_2-\delta}^{a_2}\int_{0}^{T}\int_{s_1}^{s_2}\int_{\omega}\tilde{q}^2(x,u,l,v)dxdldu dv.\]
 \begin{figure}[H]
 \begin{adjustbox}{center,left}
			\begin{overpic}[scale=0.18]{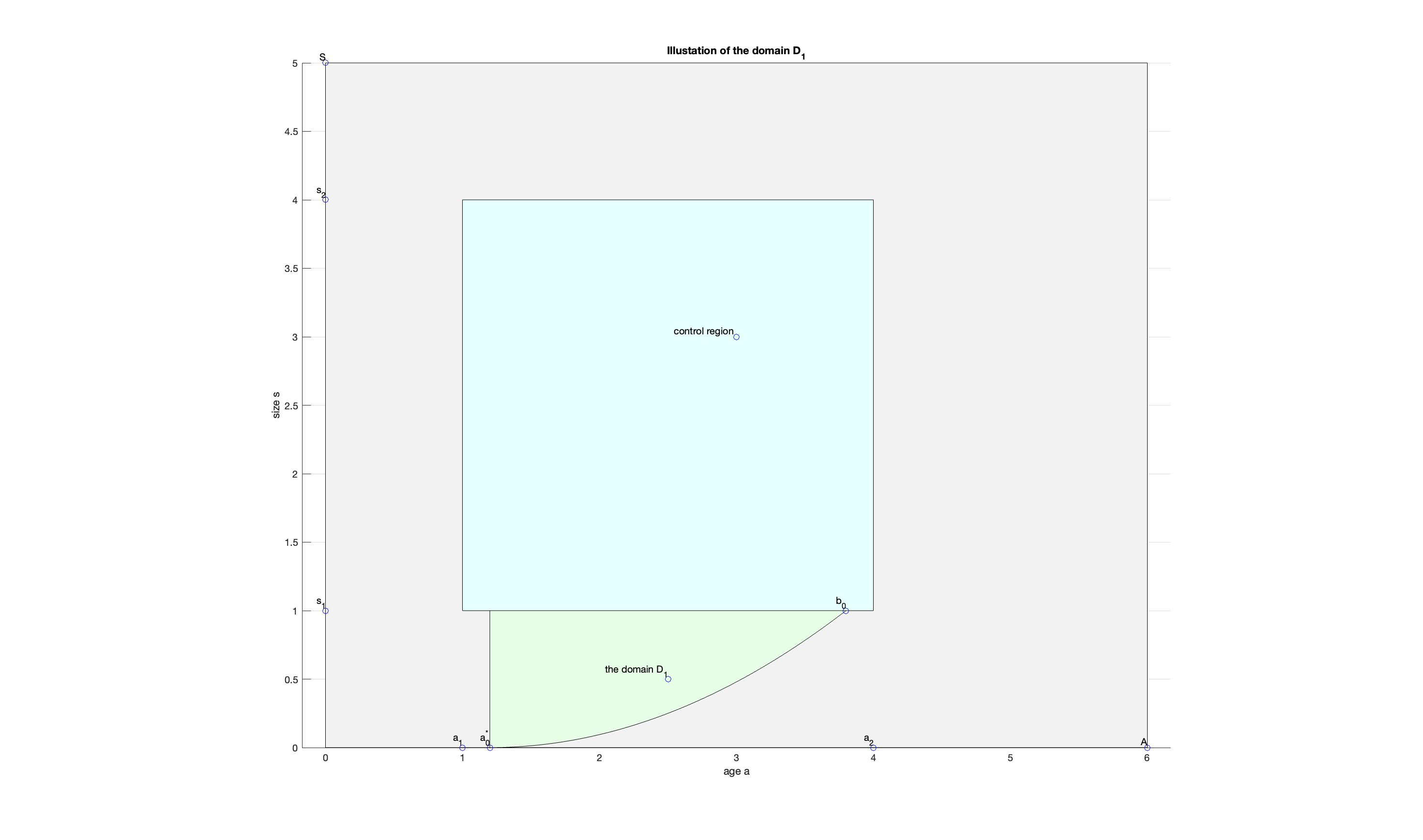}
\end{overpic}
\caption{Illustration of the integration domain in the case of the inequality \ref{xcxa}.}
\end{adjustbox}

\end{figure}
\end{proof}
      \part*{Conclusion}
 Considering a population dynamics model with age, size structuring, and diffusion, we establish null controllability results for two models with control localized with respect to age, size, and space variables.
First, we show that if the birth rate at time \(t\) is given by
\[
y(x,0,s,t) = \int_{0}^{A} \int_{0}^{S} \beta(a,\hat{s},s) y(x,a,\hat{s},t) \, da \, d\hat{s},
\]
then the null controllability estimated time is
\[T > A - a_2 + \max\{a_1 + S^*_2, S^*_1\} + \max\{S^*_2, S^*_1\}.\]
In the second part, we show that if the birth rate is
\[y(x,0,s,t) = \int_{0}^{A} \beta(a,s) y(x,a,s,t) \, da,\]
then the estimated time is
\[T > A - a_2 + a_1 + S_1^* + S_2^*,\]
which is more optimal.
These two results allow us to conclude that the form of the birth rate influences the minimal controllability time. In fact, we highlight that the random size of the newborns is a very important factor in estimating the control time.\smallbreak
 \textbf{Acknowledgement}
The authors wish to thank Prof. Enrique Zuazua and Dr. Thimoth\'ee Crin-Barat for his comments, suggestions and for fruitful discussions.\smallbreak
The author would like to th for fruitful discussions.

\end{document}